\documentclass[11pt,oneside]{article}
\usepackage{tikz}
\usetikzlibrary{arrows.meta, positioning, quotes}
\usetikzlibrary{arrows}
\usepackage{amssymb,amsmath,latexsym,amsfonts,amscd,amsthm,multirow,ctable,colortbl,mathdots,caption,array}
\usepackage{upgreek}
\usepackage{amsmath}
\usepackage{geometry}
\usepackage{wasysym}
\usepackage{mathtools}
\usepackage{fancyhdr,tabularx,cite,mathrsfs,graphicx}
\usepackage{authblk}
\usepackage[headings]{fullpage}
\usepackage{stmaryrd}
\usepackage[all]{xy}
\usepackage{stmaryrd,rotating}
\usepackage{dsfont}
\usepackage{pdflscape,longtable}

\usepackage[font=footnotesize]{caption}

\usepackage{setspace}

\usepackage{hyperref}

\newcommand{\bea}{\begin{eqnarray}} 
\newcommand{\eea}{\end{eqnarray}} 
\newcommand{\bee}{\begin{eqnarray*}} 
\newcommand{\eee}{\end{eqnarray*}} 
\newcommand{\al}{\begin{align*}} 
\newcommand{\eal}{\end{align*}} 
\newcommand{\be}{\begin{equation}} 
\newcommand{\ee}{\end{equation}} 
 
\newcommand{\bem}{\begin{pmatrix}} 
\newcommand{\eem}{\end{pmatrix}}

\def\a{\alpha} 
 
\def\c{\gamma}

\def\m{\mu}

\def\u{\upsilon}

\def\D{D_1}

\newcommand{\gt}[1]{\mathfrak{#1}}
\newcommand{\mc}[1]{\mathcal{#1}}

\newcommand{\comment}[1]{}

\newcommand{\RR}{{\mathbb R}}
\newcommand{\CC}{{\mathbb C}}
\newcommand{\PP}{{\mathbb P}}
\newcommand{\ZZ}{{\mathbb Z}}
\newcommand{\QQ}{{\mathbb Q}}
\newcommand{\HH}{{\mathbb H}}

\newcommand{\Aut}{\operatorname{Aut}}

\newcommand{\Irr}{\operatorname{Irr}}

\newcommand{\Id}{\operatorname{Id}}
\newcommand{\tr}{\operatorname{{tr}}}

\newcommand{\sgn}{\operatorname{sgn}}
\newcommand{\ex}{\operatorname{e}} 

\newcommand{\wh}{{\rm wh}}	

\newcommand{\reg}{{\rm reg}}

\newcommand{\vh}{h}
\newcommand{\xmod}{{\rm \;mod\;}}

\newcommand{\SL}{\operatorname{\textsl{SL}}}      
\newcommand{\mpt}{\widetilde{\SL}_2}      
\newcommand{\GL}{{\textsl{GL}}}      

\newcommand{\vn}{V^{\natural}} 
\newcommand{\G}{\Gamma}	

\newcommand{\MM}{\mathbb{M}}

\newcommand{\Th}{\textsl{Th}}

\newcommand{\Kb}{\kappa}	
\newcommand{\Km}{M}	
\newcommand{\SQ}{\operatorname{\textsl{SQ}}}

\newcommand{\pd}{\uplambda}

\newtheorem{thm}{Theorem}[subsection]

\newtheorem{lem}[thm]{Lemma}
\newtheorem{pro}[thm]{Proposition}

\theoremstyle{definition}

\newtheorem{eg}[thm]{Example}

\theoremstyle{remark}
\newtheorem{rmk}[thm]{Remark}

\numberwithin{equation}{subsection}

\begin{document}

\setstretch{1.26}

\renewcommand{\thefootnote}{\fnsymbol{footnote}} 

\title{
\vspace{-35pt}
\textsc{\huge{ 
Modular Products and 
Modules for\\ \hspace{1pt} Finite Groups
}}
    }
    
\renewcommand{\thefootnote}{\arabic{footnote}}

\author[1,2]{John F. R. Duncan\thanks{jduncan@gate.sinica.edu.tw}\thanks{john.duncan@emory.edu}}
\author[3]{Jeffrey A. Harvey\thanks{j-harvey@uchicago.edu}}
\author[4]{Brandon C. Rayhaun\thanks{brayhaun@stanford.edu}}

\affil[1]{Institute of Mathematics, Academia Sinica, Taipei, Taiwan.}
\affil[2]{Department of Mathematics, Emory University, Atlanta, GA 30322, U.S.A.}
\affil[3]{Enrico Fermi Institute and Department of Physics, University of Chicago, Chicago, IL 60637, U.S.A.}
\affil[4]{Institute for Theoretical Physics, Stanford University, Palo Alto, CA 94305, U.S.A.}

\date{\vspace{-.45in}}

\maketitle

\abstract{
Motivated by the appearance of 
penumbral moonshine,
and 
by evidence that penumbral moonshine enjoys an extensive relationship to 
generalized monstrous moonshine via infinite products,
we establish a general construction in this work
which uses 
singular theta lifts and a concrete construction at the level of modules for a finite group to
translate between 
moonshine in weight one-half
and 
moonshine 
in weight zero.
This construction serves as a foundation for a companion paper in which we explore the connection between penumbral Thompson moonshine and a special case of generalized monstrous moonshine in detail.
}

\clearpage

\tableofcontents

\section{Introduction}\label{sec:int}

Penumbral moonshine \cite{pmo} provides a counterpart to umbral moonshine \cite{Cheng:2012tq,cheng2014umbral,cheng2018weight} in which mock modularity is replaced by modularity, 
and Fricke genus zero groups take on the role (see \cite{Cheng:2016klu}) played by non-Fricke genus zero groups.

Curiously, the finite groups that appear in penumbral moonshine are generally much larger than those that appear in umbral moonshine, and in many cases are identifiable as, or closely related to, centralizer subgroups of the monster. 
It is natural to ask if this 
is purely coincidental, or a consequence of 
a concrete connection.
We provide 
evidence in support of the latter alternative in \cite{pmt}, wherein we explore a relationship between penumbral Thompson moonshine (which was first exposited in \cite{Harvey:2015mca}), and a special case of generalized moonshine for the monster. 

The present work provides a foundation for the analysis of \cite{pmt} by establishing a ``lift'' 
of the singular theta lift of 
\cite{Harvey:1995fq,Borcherds:1996uda},
from modular forms to modules for a finite group. 
More concretely, the construction we present here 
associates a (virtual) graded $G$-module for which the McKay--Thompson series arising are weakly holomorphic modular forms of weight $0$, to any (virtual) graded $G$-module for which the McKay--Thompson series arising are weakly holomorphic vector-valued modular forms of weight $\frac12$ of a suitable type,
for any finite group $G$. 
Moreover, the McKay--Thompson series of weight $0$ that arise from this construction are defined by infinite products.

\subsection{Products}\label{sec:int-pro}

The problem of obtaining automorphic forms as infinite products has received considerable attention. A famous example, which first appeared in \cite{Borcherds1995} (see Example 2 in \S~13 of op.\ cit.), is the infinite product formula for the $j$-function, 
\begin{gather}\label{eqn:int-pro:j}
j(\tau) = q^{-1} \prod_{n>0} (1-q^n)^{3c_0(n^2)}.
\end{gather}
Here $q=\ex(\tau)$, where $\ex(x):=e^{2\pi i x}$, and 
the product in (\ref{eqn:int-pro:j}) is constructed from the Fourier coefficients 
$c_0(D)$
of the unique 
weakly holomorphic modular form 
\begin{gather}\label{eqn:int-pro:f0}
f_0(\tau) 
= \sum_{D\geq -3}
c_0(D)
q^D = q^{-3}-248q+26752q^4-85995q^5+\dots
\end{gather} 
of weight $\frac12$ for $\Gamma_0(4)$ 
such that $c_0(D)=0$ unless $D\equiv 0,1\xmod 4$,
and such that 
$f_0(\tau)=q^{-3}+O(q)$ 
as $\Im(\tau)\to \infty$. 

Building on work of Borcherds \cite{Borcherds1995,Borcherds:1996uda}, 
Zagier \cite{ZagierTSM}, and then Bruinier and Ono \cite{BruinierOno2010}, 
established a variety of generalizations of this formula. 
In particular, they studied ``twisted'' Borcherds products (see \S~\ref{sec:int-twi}), which, among other things, realize 
particular rational functions of the elliptic modular invariant (\ref{eqn:int-pro:j}) involving its values at CM points (cf.\ (\ref{eqn:out-inv:PsiWD1r1-trc})). 

Now the $j$-function is, up to an additive constant, the McKay--Thompson series 
\begin{gather}\label{eqn:int-pro:Te}
T_e(\tau)
=j(\tau)-744
=q^{-1}+196884q+21493760q^2+864299970q^3
+\dots
\end{gather}
attached to the identity element of the monster group, $\MM$, by monstrous moonshine \cite{Conway:1979qga}.
It develops that similar results involving twisted Borcherds products hold for other functions appearing in this setting. 
For example, infinite product formulae are obtained in \cite{kim_2006} for rational functions of McKay--Thompson series $T_g$
attached to elements of prime order $o(g)=p$ in the monster, for each prime $p$ that divides 
$|\MM|$.

These results on infinite products follow from the application of singular theta lifts---either similar to or the same as those considered in \cite{Harvey:1995fq,Borcherds:1996uda}---to 
weakly holomorphic modular forms
of weight $\frac12$
that transform with respect to a suitably chosen Weil representation 
of the metaplectic double cover $\widetilde{\SL}_2(\ZZ)$ of the modular group $\SL_2(\ZZ)$. 
In the untwisted case, the idea is that the infinite product constructed from the Fourier coefficients of a suitable modular form $\check{F}$ is essentially the exponential of a regularized integral 
\begin{gather}\label{eqn:int-pro:int}
\int_{\mathcal{F}}^\reg 
\overline{\Theta_{L}(\tau,v^+)}\check{F}(\tau)
\frac{{\rm d}\tau_1{\rm d}\tau_2}{\tau_2}.
\end{gather}
Here 
$\mathcal{F}$ is the standard fundamental domain for the action of $\SL_2(\ZZ)$ on the complex upper half-plane $\HH$,
we take $L$ to be a lattice 
that defines the Weil representation in question (see \S~\ref{sec:pre-lat}), 
we write $\Theta_L$ for the associated Siegel theta function (see \S~\ref{sec:pre-mod}),
and $\tau=\tau_1+i\tau_2$ is the decomposition of $\tau$ into its real and imaginary parts. 
(See \S~\ref{sec:res-lif} and especially (\ref{eqn:res-lif:stl}) for more detail on (\ref{eqn:int-pro:int}).)

For an example of this framework in action, define $\check{F}=(\check{F}_0,\check{F}_1)$ by requiring that 
$\check{F}_r$ belongs to $q^{\frac{r^2}{4}}\ZZ[[q]][q^{-1}]$
and
\begin{gather}\label{eqn:int-pro:checkF}
\check{F}_0(4\tau)+\check{F}_1(4\tau)=
f_0(\tau),
\end{gather} 
where 
$f_0$
is as in (\ref{eqn:int-pro:f0}).
Then $\check{F}$ transforms under the Weil representation of $\widetilde{\SL}_2(\ZZ)$ (see (\ref{eqn:pre-lat:wei}--\ref{eqn:pre-lat:weiLmN})) defined by the lattice $L=\sqrt{2}\ZZ\oplus \Gamma^{1,1}$ (cf.\ (\ref{eqn:pre-lat:LmN})), and for this choice of $L$ we may naturally regard the variable
$v^+$ in $\Theta_L(\tau,v^+)$ as an element of the upper half-plane 
(cf.\ (\ref{eqn:ZtoZL}--\ref{eqn:PsiellZcheckF})).
The integral \eqref{eqn:int-pro:int} then evaluates to $-4\log|j(v^+)|$ (cf.\ (\ref{eqn:STL:logPsi})),
and the method of 
\cite{Harvey:1995fq,Borcherds:1996uda} produces the product formula (\ref{eqn:int-pro:j}) for $j$ (cf.\ (\ref{eqn:STL:infpro})).

\subsection{Moonshine}\label{sec:int-moo}

In \cite{Harvey:2015mca} it was observed that the cube root of (\ref{eqn:int-pro:j}) relates the 
graded dimensions of two graded modules for the sporadic simple group, $\Th$, of Thompson \cite{MR399257,MR0399193}.
Indeed, $T_{\rm 3C}(\tau)=j(3\tau)^{\frac13}$ is the McKay--Thompson series associated to 
an element of the class 3C of the monster, so according to Borcherds' proof \cite{Borcherds1992} of the monstrous moonshine conjecture, 
$j(3\tau)^{\frac13}$ coincides with the trace of an element of the class 3C on the moonshine module $\vn$ of Frenkel, Lepowsky and Meurman \cite{flm,FLMBerk,frenkel1989vertex}. 
The centralizer of such an element takes the form 
$C_{\mathbb{M}}(\mathrm{3C})\cong \ZZ/3\ZZ\times\Th$,
and the group $\Th$ admits no non-trivial central extensions (see e.g.\ \cite{atlas}). 
Applying the yoga 
of generalized monstrous moonshine (see e.g.\ \cite{Carnahan2012,Carnahan:2012gx}) to these facts we obtain an action of $\Th$ on the 3C-twisted moonshine module 
$\vn_{{\rm 3C}}$, whose graded dimension is given by 
\begin{gather}\label{eqn:int-moo:jottauot}
j(\tfrac13\tau)^{\frac13}
=
q^{-\frac19}\prod_{n>0}(1-q^{\frac n3})^{c_0(n^2)}.
\end{gather}

The second graded $\Th$-module is the one introduced in \cite{Harvey:2015mca}, which we here denote by 
\begin{gather}\label{eqn:int-moo:breveW31}
\breve{W}^{(-3,1)}=\bigoplus_{D\geq -3} \breve{W}^{(-3,1)}_D,
\end{gather} 
following the conventions of \cite{pmo,pmt}.
Its (signed) graded dimension is given by
\begin{gather}\label{eqn:int-moo:2f0248}
	\breve{F}^{(-3,1)}(\tau) 
	= 2f_0(\tau) +248\theta(\tau)
	= 2q^{-3}+248+54000q^4-171990q^5+\dots
\end{gather}
where $f_0$ is as in (\ref{eqn:int-pro:f0}) and $\theta(\tau):=\sum_n q^{n^2}$.
Significantly, the Thompson group admits a unique irreducible representation of dimension $248$, which we henceforth denote ${\bf 248}$. 
Thus, at the level of Fourier coefficients we have
\begin{gather}\label{eqn:int-moo:2c0}
	2c_0(n^2) = 
	\tr(e|\breve{W}^{(-3,1)}_{n^2}) - 2 \tr(e\vert {\bf 248})
\end{gather}
for $e$ the identity element of $\Th$ (where we regard each $W^{(-3,1)}_D$ as a virtual $\Th$-module, in order to handle the signs in (\ref{eqn:int-moo:2f0248})).

Comparing (\ref{eqn:int-moo:jottauot}) with (\ref{eqn:int-moo:2c0}) 
we are motivated to ask if
there is a deeper relationship between the two $\Th$-modules $\vn_{{\rm 3C}}$
and $\breve{W}^{(-3,1)}$. 
For example, if we twine (\ref{eqn:int-moo:2c0}) by replacing each $e$ on the right-hand side with an arbitrary Thompson group element $g$, does the singular theta lift (\ref{eqn:int-pro:int}) of the corresponding twine of $\check{F}$ as in (\ref{eqn:int-pro:checkF}) have a meaningful relationship to the twine of the square of (\ref{eqn:int-moo:jottauot}) defined by the action of $g$ on the 3C-generalized moonshine module $\vn_{{\rm 3C}}$?
This is a question we investigate in detail in \cite{pmt}. One of the problems we face in determining an answer is that of appropriately defining ``the corresponding twine of $\check{F}$,'' because if we naively replace $f_0$ in (\ref{eqn:int-pro:checkF}) with $\frac12\breve{F}^{(-3,1)}_g-\frac12\tr(g|{\bf 248})\theta$ (cf.\ (\ref{eqn:int-moo:2f0248}--\ref{eqn:int-moo:2c0})),
where 
\begin{gather}\label{eqn:int-moo:breveF31g}
\breve{F}^{(-3,1)}_g(\tau)
=
\sum_{D\geq -3} \tr(g|\breve{W}^{(-3,1)}_D)q^D,
\end{gather} 
then we obtain a modular form that fails to transform appropriately under the full group $\widetilde{\SL}_2(\ZZ)$, and in particular fails to satisfy the hypotheses of \cite{Borcherds:1996uda}.

It turns out that Borcherds himself solved an essentially similar problem earlier, in proving \cite{Borcherds1992} the monstrous moonshine conjecture. 
In that work a key role is played by the monster Lie algebra, $\gt{m}$, which is constructed in a functorial way from $\vn$, and has
\begin{gather}\label{eqn:int-moo:den}
\sum_{m\geq -1} c_e(m)p^m 
-\sum_{n\geq -1}c_e(n)q^n
= p^{-1}\prod_{m>0,n\geq -1}(1-p^mq^n)^{c_e(mn)}
\end{gather}
for its denominator formula, where 
$c_e(n)$ is the coefficient of $q^n$ in the Fourier expansion of $T_e$ (so that the left-hand side of (\ref{eqn:int-moo:den}) coincides with $T_e(\sigma)-T_e(\tau)$ according to (\ref{eqn:int-pro:Te}), when $p=\ex(\sigma)$). 
The identification of a McKay--Thompson series $T_g$, as specified by Conway and Norton \cite{Conway:1979qga}, 
as the graded trace function associated to the action of 
$g\in\MM$
on $\vn=\bigoplus_{n\geq -1}\vn_n$
depends upon a corresponding twine of the identity (\ref{eqn:int-moo:den}), whereby 
we identify $c_e(n)=\dim(\vn_n)=\tr(e|\vn_n)$, and then replace 
each $e$ on the left-hand side of (\ref{eqn:int-moo:den}) with $g$.
From the method of \cite{Borcherds1992} it emerges that the correct 
generalization of (\ref{eqn:int-moo:den}) is 
\begin{gather}\label{eqn:int-moo:dentws}
\sum_{m\geq -1} \tr(g|\vn_m)p^m 
-\sum_{n\geq -1}\tr(g|\vn_n)q^n
= p^{-1}
\exp\left(
-
\sum_{m>0}\sum_{n\geq -1}
\sum_{k>0}
\tr(g^k|\vn_{mn})\frac{p^{mk}q^{nk}}k
\right)
\end{gather}
(see (8.3) of op.\ cit.).
We learn from this, in particular, that the $g$-twine of the infinite product in (\ref{eqn:int-moo:den}) should involve all the powers of $g$.

In light of the subsequent work \cite{Harvey:1995fq,Borcherds:1996uda} we may interpret (\ref{eqn:int-moo:den}) as the product formula for an $O_{2,2}$-type automorphic form resulting from an application of the singular theta lift (\ref{eqn:int-pro:int}) with $L=\Gamma^{1,1}\oplus \Gamma^{1,1}$ (cf.\ (\ref{eqn:pre-lat:LmN})) and $\check{F}=T_e$, whereas the examples of relevance to us, such as that of (\ref{eqn:int-pro:checkF}), produce automorphic forms of type $O_{2,1}$. 
The interpretation of (\ref{eqn:int-moo:dentws}) as an $O_{2,2}$-type singular theta lift is included in Carnahan's work on generalized monstrous moonshine (see in particular \cite{Carnahan2012}), which includes a general method for ``repackaging'' a family $\{T_{g^k}\}$ of modular forms of weight $0$
into a vector-valued modular form $\check{T}_g$ for $\widetilde{\SL}_2(\ZZ)$, for a suitably chosen Weil representation.
Our solution to the problem of twining the product side of (\ref{eqn:int-moo:jottauot}) will involve a directly similar strategy (see \S~\ref{sec:res-rpk}), but for singular theta lifts of type $O_{2,1}$ rather than $O_{2,2}$.

\subsection{Modules}\label{sec:int-mod}

Beyond the technical tool of repackaging, a conceptual 
point that is illuminated by reflection upon Borcherds work \cite{Borcherds1992} 
on monstrous moonshine is the analogy between 
$F^{(-3,1)}(\tau)$ and $j(\tau)^{\frac23}$ on the one hand (cf.\ (\ref{eqn:int-moo:jottauot}--\ref{eqn:int-moo:2c0})),
and $T_e(\tau)$ and $T_e(\sigma)-T_e(\tau)$ on the other (cf.\ (\ref{eqn:int-pro:Te}), (\ref{eqn:int-moo:den})).
At the level of modules for finite groups this suggests the possibility of a relationship between $\breve{W}^{(-3,1)}$ and $(\vn_{{\rm 3C}})^{\otimes 2}$ that is analogous to that which relates $\vn$ to $\gt{m}$. 
Is there an analogue of the functorial construction due to Borcherds, that produces a $\Th$-invariant Lie algebra from $\breve{W}^{(-3,1)}$, 
with 
root multiplicities given by the values $2c_0(n^2)$ (cf.\ (\ref{eqn:int-moo:2c0}))?

We face an even bigger challenge in answering this question, because the construction of $\gt{m}$ due to Borcherds depends crucially upon the $\MM$-invariant vertex operator algebra (VOA) structure on $\vn$, and there is---so far---no known counterpart for $\breve{W}^{(-3,1)}$. 
For this reason we do not know yet how to proceed, if the goal is a Lie algebraic counterpart to $\gt{m}$ in the context of $\breve{W}^{(-3,1)}$ and $(\vn_{\rm 3C})^{\otimes 2}$.

However, we do have $\Th$-module structures on $\breve{W}^{(-3,1)}$ and $(\vn_{\rm 3C})^{\otimes 2}$, and we may ask if these structures are related, or not, by a ``lift'' of the singular theta lift (\ref{eqn:int-pro:int}) that operates at the level of $G$-modules, for $G$ a finite group, rather than at the level of functions.

The construction of such a lift of the singular theta lift, in the $O_{2,1}$ setting, is the purpose and product of this work. 
More specifically, 
for $G$ a finite group and $m$ a non-zero integer we formulate a notion of 
rational weakly holomorphic $G$-module of weight $\frac12$ and index $m$, 
and a notion of 
weakly holomorphic $G$-module of weight $0$
(see \S~\ref{sec:res-for}), 
and we establish a general construction 
\begin{gather}\label{eqn:int-mod:SQ}
W\mapsto \SQ(W)
\end{gather} 
(see Theorem \ref{thm:res} and (\ref{eqn:res-for:SQW})) that produces 
a weakly holomorphic $G$-module $\SQ(W)$ 
of weight $0$ from an arbitrary rational weakly holomorphic $G$-module $W$ of weight $\frac12$ with positive index.
The construction (\ref{eqn:int-mod:SQ}) depends upon the 
singular theta lift (\ref{eqn:int-pro:int}) of 
\cite{Harvey:1995fq,Borcherds:1996uda}, and also on $O_{2,1}$-type repackaging, 
and produces infinite product formulae for the McKay--Thompson series associated to the output $G$-module $\SQ(W)$.

The cornerstone of the construction (\ref{eqn:int-mod:SQ}) is a 
twined counterpart to the $O_{2,1}$-type Borcherds product of (\ref{eqn:int-pro:j}) and (\ref{eqn:int-moo:jottauot}). 
In the case of a rational weakly holomorphic $G$-module 
\begin{gather}\label{eqn:int-mod:W}
W=\bigoplus_{r\xmod 2m}\bigoplus_{D\equiv r^2\xmod 4m} W_{r,\frac{D}{4m}}
\end{gather} 
of weight $\frac12$ and positive integer index $m$ (cf.\ (\ref{eqn:res-for:W})), 
it may be expressed in the form
\begin{gather}\label{eqn:int-mod:PsiWg}
\Psi^W_g(\tau) := q^{-H}\exp\left(-
\sum_{n>0} 
\sum_{k>0}
C^W_{g^k}(n^2,n)
\frac{q^{nk}}{k}   \right),
\end{gather}
(cf.\ (\ref{eqn:int-moo:dentws})), for $g\in G$, 
where $H=H^W$ is a generalized class number associated to $W$ (see (\ref{eqn:H})), 
and 
$C^W_g(D,r):=\tr(g|W_{r,\frac{D}{4m}})$.

A priori it is not clear that the prescription (\ref{eqn:int-mod:PsiWg}) makes sense, but we confirm (see \S~\ref{sec:res-CM}) that the right-hand side of (\ref{eqn:int-mod:PsiWg}) converges for $\Im(\tau)$ sufficiently large, and extends by analytic continuation to a holomorphic function on the upper half-plane. 

Note that we obtain a rational weakly holomorphic $G$-module $W^{(-3,1)}$ of weight $\frac12$ and index $1$, 
for $G=\Th$, by setting 
\begin{gather}\label{eqn:int-mod:W31}
W^{(-3,1)}_{r,\frac{D}{4}}:=\breve{W}^{(-3,1)}_{D}
\end{gather} 
(cf.\ (\ref{eqn:int-mod:W})) 
for $\breve{W}^{(-3,1)}$ as in (\ref{eqn:int-moo:breveW31}), where $r$ on the left-hand side of (\ref{eqn:int-mod:W31}) is $0$ or $1$ according as $D$ is even or odd.

We denote our general construction (\ref{eqn:int-mod:SQ}) by $\SQ$ in homage to the functorial construction of Borcherds \cite{Borcherds1992}, that produces the monster Lie algebra $\gt{m}$ from the moonshine module $\vn$, and is sometimes referred to as a second-quantization functor, 
on account of its interpretation in physics. 
We do not employ the formalities of categories and functors in this work, but we do provide an explicit realization of $\SQ(W)$ in terms of the alternating and symmetric powers of homogeneous subspaces of $W$ (see \S~\ref{sec:res-con}), and it follows from this that our construction is functorial at the level of $G$-modules.

\subsection{Twists}\label{sec:int-twi}

We 
return now to the twisted Borcherds products of 
Zagier \cite{ZagierTSM} and Bruinier and Ono \cite{BruinierOno2010} 
that we mentioned in \S~\ref{sec:int-pro}.
Inspired by their results, and by the problem of inverting our construction $\SQ$ (see \S~\ref{sec:out-inv}), we 
suggest the consideration of twined twisted $O_{2,1}$-type Borcherds products for future work. 
In the case that $W$ is a rational weakly holomorphic $G$-module of weight $\frac12$ and index $m$ 
 for $m$ a positive integer
(cf.\ 
(\ref{eqn:int-mod:W}), 
(\ref{eqn:res-for:W})), the construction we propose takes the form
\begin{gather}\label{eqn:int-twi:PsiWDeltag}
\Psi^{W}_{D_1,r_1,g}( \tau) := \exp\left(-
\sum_{n>0} 
\sum_{k>0}
\sum_{a\xmod D_1}
C^W_{g^k}(D_1 n^2,r_1 n)
\left(\frac{D_1}{a}\right) 
\ex\left(\frac{ak}{D_1}\right)\frac{q^{nk}}{k}  \right),
\end{gather}
for 
$D_1>1$ a fundamental discriminant, 
$r_1$ an integer such that $D_1=r_1^2\xmod 4m$, and $g\in G$,
where 
$C^W_{g}(D,r)$ is as in (\ref{eqn:int-mod:PsiWg}), and
$\left(\frac{D_1}{a}\right)$ is the Kronecker symbol (see e.g.\ \cite{GKZ87} for the definition).
The results of \cite{BruinierOno2010} cover the case that $g$ in \eqref{eqn:int-twi:PsiWDeltag} is the identity element,
whereby the definition (\ref{eqn:int-twi:PsiWDeltag}) we propose reduces to 
\begin{gather}\label{eqn:int-twi:PsiWDeltae}
    \Psi^{W}_{D_1,r_1}(\tau) := \prod_{n>0}
    \prod_{a\xmod D_1} 
    \left(1-\ex(\tfrac{a}{D_1})q^n\right)^{\left(\frac{D_1}{a}\right)C^{W}(D_1 n^2,r_1 n)},
\end{gather}
where $\Psi^W_{D_1,r_1}:=\Psi^W_{D_1,r_1,e}$, and $C^W(D,r):=C^W_e(D,r)$ is the dimension of $W_{r,\frac{D}{4m}}$.
In particular, it is shown in op.\ cit.\ that the right hand-side of (\ref{eqn:int-twi:PsiWDeltae}) converges for $\Im(\tau)$ sufficiently large, and extends by analytic continuation to a meromorphic function on $\HH$ with an explicitly determined divisor (cf.\ (\ref{eqn:out-inv:ZWD1r1})). 

The automorphic properties of 
$\Psi^{W}_{D_1,r_1}$ 
are established in op.\ cit.\ via an analysis of twisted Siegel theta functions $\Theta^{(m)}_{D_1,r_1}$ 
based on the lattices $L_m=\sqrt{2m}\ZZ\oplus\Gamma^{1,1}$ (cf.\ (\ref{eqn:pre-lat:LmN})).
Briefly, $\Theta^{(m)}_{D_1,r_1}$ is defined similarly to the theta function we denote $\Theta^{(m)}$ in \S~\ref{sec:pre-mod}, but with the generalized genus character $\chi^{(m)}_{D_1}$ (see (\ref{eqn:pre-hee:chiD0N})) twisting the contribution from each vector of $L_m$.
We expect that the main step towards understanding the automorphic properties of (\ref{eqn:int-twi:PsiWDeltag}) will be a counterpart analysis of similar twisted Siegel theta functions
$\Theta^{(m)}_{D_1,r_1,N}$, 
based on the more general family $L_{m,N}=\sqrt{2m}\ZZ\oplus \Gamma^{1,1}(N)$ 
(see (\ref{eqn:pre-lat:LmN})).

A main motivation for the construction (\ref{eqn:int-twi:PsiWDeltag}) is the fact that 
the untwisted products $\Psi^W_g$ of (\ref{eqn:int-mod:PsiWg}) only ``see'' the (virtual) $G$-modules $W_{r,\frac{D}{4m}}$ for which $D=n^2$ is a perfect square, whereas
the twined twisted Borcherds products $\Psi^W_{D_1,r_1,g}$ of (\ref{eqn:int-twi:PsiWDeltag}) incorporate all the remaining components of $W$.
Thus it is natural to expect that twined twisted Borcherds products
will play an important role in the final comparative analysis of moonshine in weight one-half and weight zero.

In fact,
untwined twisted Borcherds products very similar to the $\Psi^W_{D_1,r_1}$ of (\ref{eqn:int-twi:PsiWDeltae}) have already appeared in such 
analyses.
The first examples of this were given in \cite{ORTL}, wherein twisted Borcherds products were used to relate certain cases of umbral moonshine to conjugacy classes in the monster (see Theorem 1.1 of op.\ cit.). 
Twisted borcherds products were subsequently 
used to constructively relate all cases of umbral moonshine to principal moduli for genus zero groups in 
\cite{Cheng:2016klu} (see Theorem 4.5.4 of op.\ cit.). 
We discuss this more in \S~\ref{sec:out-imp}.

As we also discuss in \S~\ref{sec:out-imp}, 
it develops 
that the untwisted Borcherds product construction, (\ref{eqn:int-mod:PsiWg}), cannot be applied to umbral moonshine. 
Thus, not only are we motivated to develop twined twisted Borcherds products as in (\ref{eqn:int-twi:PsiWDeltag}) for the purpose of better understanding penumbral moonshine, and its potential monstrous aspect, 
but also in order to shine more light on umbral moonshine, 
since the lift of the untwisted Siegel theta lift that we develop in this work does not apply there. 
In particular, we are motivated to ask:
Can the
twisted Siegel theta lift, 
\begin{gather}\label{eqn:int-twi:int}
\int_{\mathcal{F}}^\reg 
\overline{\Theta^{(m)}_{D_1,r_1}(\tau,v^+)}\check{F}(\tau)
\frac{{\rm d}\tau_1{\rm d}\tau_2}{\tau_2},
\end{gather}
that is studied in \cite{BruinierOno2010},
be lifted to the level of $G$-modules too? 
We offer this question as a focus for future work.

\subsection{Overview}

The structure of this paper is as follows. 
In \S~\ref{sec:notation} we give a guide to the specialized notation we use. 
In \S~\ref{sec:preliminaries} we review some preliminary notions, including the Weil representations of the metaplectic double cover of the modular group that are defined by even lattices in \S~\ref{sec:pre-lat},
some relevant families of modular forms in \S~\ref{sec:pre-mod}, and 
the twisted Heegner divisors of \cite{BruinierOno2010} in \S~\ref{sec:pre-hee}.
We 
recall 
Adams operations on modules for finite groups in \S~\ref{sec:pre-ada}.

The new results appear in \S~\ref{sec:res}, wherein we explain and establish our general construction $\SQ$, that utilizes infinite products 
to translate 
from
moonshine in 
weight one-half
to moonshine in weight zero.
We set up for and define the construction $\SQ$ abstractly in 
\S~\ref{sec:res-for} (see Theorem \ref{thm:res}). 
Then, in \S~\ref{sec:res-rpk}, we prove two technical lemmas (Lemmas \ref{repackaged_lem} and \ref{lem:FbnGamma0}). 
The first of these generalizes work of Carnahan \cite{Carnahan2012}, and constitutes the repackaging method we discussed in \S~\ref{sec:int-moo}. 
The second lemma of \S~\ref{sec:res-rpk} is used when we check that this repackaging works in the desired way.
We also present some examples of repackaging in \S~\ref{sec:res-rpk}. 

Our next step is to
recall the singular theta lift from \cite{Borcherds:1996uda}, but specialized to the situation specified in \S\S~\ref{sec:res-for}--\ref{sec:res-rpk}.
We do this in \S~\ref{sec:res-lif}, and thereby prepare for  the arguments of \S~\ref{sec:res-CM}, which prove 
(see Proposition \ref{pro:conmod}) that the infinite products we propose in \S~\ref{sec:res-for} make sense. 
In particular, we establish the convergence and modularity of the products of \S~\ref{sec:res-for} in \S~\ref{sec:res-CM}.
The verification of the validity of the construction $\SQ$ is completed in \S~\ref{sec:res-con}, wherein we offer a representation theoretic interpretation of the infinite products constructed in \S\S~\ref{sec:res-for}--\ref{sec:res-CM}. 

We put our results in perspective in \S~\ref{sec:out}.
Specifically, we 
elucidate the relationship between our construction $\SQ$ and traces of singular moduli in \S~\ref{sec:out-inv}, and 
comment on the import of our work for penumbral and umbral moonshine in \S~\ref{sec:out-imp}.

\section*{Acknowledgements}

We are grateful to 
Nathan Benjamin, Scott Carnahan, Miranda Cheng, Shamit Kachru, Michael Mertens, Ken Ono, Natalie Paquette and Max Zimet for helpful communication and discussions.
B.R.\ acknowledges support from the U.S.\ National Science Foundation (NSF) under grant PHY-1720397. 
J.D.\ acknowledges support from the NSF (DMS-1203162, DMS-1601306), and the Simons Foundation (\#316779, \#708354). 
J.H.\  acknowledges support from the NSF under grant PHY-1520748, and the hospitality of the Aspen Center for Physics supported by the NSF under grant PHY-1607611. 
This research was also supported in part by the NSF under grant PHY-1748958.

\section{Notation}\label{sec:notation}

\begin{footnotesize}

\begin{list}{}{
	\itemsep -1pt
	\labelwidth 23ex
	\leftmargin 13ex	
	}

\item
[$\left(\frac{\cdot}{\cdot}\right)$]
The Kronecker symbol. Cf.\ (\ref{eqn:int-twi:PsiWDeltag}).

\item
[${[A,B,C]}$]
Shorthand for the binary quadratic form $Ax^2+Bxy+Cy^2$.
See \S~\ref{sec:pre-hee}.

\item
[$\Aut(L,\check{F})$]
The subgroup of $\Aut(L)$ composed of automorphisms 
that fix 
$\check{F}$.
Cf.\ (\ref{eqn:CLstarL}).

\item
[$\alpha_Q$]
The CM point in $\HH$ associated to a binary quadratic form $Q$. 
Cf.\ (\ref{eqn:pre-hee:ZmD1r1}).

\item
[$c_0(D)$]
A Fourier coefficient of $f_0$. 
See (\ref{eqn:int-pro:f0}).

\item
[$C$] 
The positive cone in $K\otimes \RR$. See (\ref{eqn:KKstarC}).

\item
[$\check{C}_{i,j}(D,r)$]
A Fourier coefficient of $\check{F}$. 
Cf.\ (\ref{eqn:STL:divcon}) and (\ref{eqn:STL:infpro}).

\item
[$\hat{C}_{(i,j)}(D,r)$]
A Fourier coefficient of $\hat{F}_{(i,j)}$. 
Cf.\ the proof of Theorem \ref{thm:STL}.

\item
[$\check{C}_K(D,r)$]
A Fourier coefficient of $\check{F}_K$. See (\ref{eqn:intThetaLcheckFK}).

\item
[$C^W(D,r)$]
Shorthand for $C^W_e(D,r)$. 
See (\ref{eqn:int-twi:PsiWDeltae}).

\item
[$C^W_g(D,r)$]
A Fourier coefficient of $F^W_{g,r}$. 
See (\ref{eqn:res-for:CWgDr}).

\item
[$\delta$]
An element of $L^*$ for some lattice $L$, or a coset of $L$ in $L^*$.
Cf.\ (\ref{eqn:CLstarL}),
(\ref{eqn:ThetaL}) and 
(\ref{eqn:lambdagammaeg}).

\item
[$\ex(\,\cdot\,)$]
We set $\ex(x):=e^{2\pi i x}$. Cf.\ (\ref{eqn:int-pro:j}).

\item
[$\eta^W_g$]
An eta product defined by the action of $g$ on $W_{0,0}$, for $W$ as in (\ref{eqn:res-for:W}).
See (\ref{eqn:res-for:etaWg}).

\item
[$f_0$]
A certain weakly holomorphic modular form of weight $\frac12$ for $\Gamma_0(4)$.
See (\ref{eqn:int-pro:f0}).

\item
[$f^V_g$]
A McKay--Thompson series for 
a 
virtual graded $G$-module $V$ as in (\ref{eqn:res-for:V}).
See (\ref{eqn:res-for:fVg}).

\item
[$F$]
A certain linear operator in \S~\ref{sec:res-con}.
See (\ref{eqn:res-con:FL0}).

\item
[$\check{F}$]
A weakly holomorphic vector-valued modular form of a certain type. 
See Lemma \ref{repackaged_lem}.

\item
[$\check{F}_K$]
A weakly holomorphic modular form of weight $\frac12$ and type $\varrho_K$. 
See (\ref{eqn:checkFKr}).

\item
[$\check{F}_{(i,j)}$]
The vector-valued function with components $\check{F}_{(i,j),r}:=\check{F}_{i,j,r}$. 
Cf.\ (\ref{eqn:FijrFijr}).

\item
[$\hat{F}_{(i,j)}$]
A function that plays a technical role in \S~\ref{sec:res-rpk}.
See 
Lemma \ref{repackaged_lem}.

\item
[$\check{F}_{i,j,r}$]
A component function of $\check{F}$. 
Cf.\ (\ref{eqn:FijrFijr}).

\item
[$\hat{F}_{(i,j),r}$]
A component function of $\hat{F}_{(i,j)}$. 
See Lemma \ref{repackaged_lem}.

\item
[$F^{(n)}$] 
A weakly holomorphic modular form of weight $\frac12$ and type $\varrho_m\vert_{\widetilde{\Gamma}_1(N/n)}$. 
See Lemma \ref{repackaged_lem}.

\item
[$F^W_g$]
A McKay--Thompson series for a graded $G$-module $W$ as in (\ref{eqn:res-for:W}). 
See (\ref{eqn:res-for:FWgr}).

\item
[$\check{F}^W_g$]
A repackaged McKay--Thompson series. See (\ref{eqn:FWgntocheckFWg}).

\item
[$F^W_{g,r}$]
A component function of $F^W_g$. See (\ref{eqn:res-for:FWgr}).

\item
[$\breve F^{(-3,1)}_g$]
A McKay--Thompson series of penumbral Thompson moonshine. 
See (\ref{eqn:int-moo:breveF31g}).

\item
[$\mathcal{F}$]
A fundamental domain for the action of $\SL_2(\ZZ)$ on $\HH$. 
Cf.\ (\ref{eqn:int-pro:int}), (\ref{eqn:res-lif:stl}).

\item
[$\Phi_L(v^+,\check{F})$]
The singular theta lift of a function $\check{F}$. 
See (\ref{eqn:res-lif:stl}).

\item
[$G(L)$]
The Grassmannian of positive definite subspaces of $L\otimes\RR$. 
See \S~\ref{sec:pre-mod}, before (\ref{eqn:ThetaL}).

\item
[$(\gamma,u)$]
An element of $\widetilde{\SL}_2(\RR)$. 
Cf.\ (\ref{eqn:widetildeSL2R}--\ref{eqn:widetildeSL2R-mlt}).

\item
[$\Gamma^{1,1}$] The unique unimodular even lattice of signature $(1,1)$. 
Cf.\ (\ref{eqn:pre-lat:LmN}).

\item
[$\widetilde{\Gamma}$]
The preimage of $\Gamma<\SL_2(\RR)$ in $\widetilde{\SL}_2(\RR)$ under the natural map 
$\widetilde{\SL}_2(\RR)\to \SL_2(\RR)$. 
See \S~\ref{sec:pre-lat}.

\item
[$\widetilde{\Gamma}_0(N)$]
The preimage of $\Gamma_0(N)$ in $\widetilde{\SL}_2(\ZZ)$ under the natural map $\widetilde{\SL}_2(\ZZ)\to \SL_2(\ZZ)$. 
See \S~\ref{sec:pre-lat}.

\item
[$\widetilde{\Gamma}_1(N)$]
The preimage of $\Gamma_1(N)$ in $\widetilde{\SL}_2(\ZZ)$ under the natural map $\widetilde{\SL}_2(\ZZ)\to \SL_2(\ZZ)$. 
See \S~\ref{sec:pre-lat}.

\item
[$\overline{\Gamma_0(m)}_Q$]
The stabilizer in $\overline{\Gamma_0(m)}:=\Gamma_0(m)/\{\pm \Id\}$ of a binary quadratic form $Q$.
Cf.\ \ref{eqn:pre-hee:ZmD1r1}.

\item
[$\Gamma_0(m)+m$]
The extension of $\Gamma_0(m)$ generated by its Fricke involution.
See (\ref{eqn:out-inv:m+m}).

\item
[$H$]
A shorthand for $H^W$. Cf.\ (\ref{eqn:res-for:PsiWg}) and (\ref{eqn:H}).

\item
[$H^W$]
The generalized class number associated to $W$ as in (\ref{eqn:res-for:W}). 
See (\ref{eqn:H}).

\item
[$H^W_\ell(g)$]
An alternative notation for $H(K,\check{F}_K)$ in a special case. 
See (\ref{eqn:Hellg}).

\item
[$H_{m}(D,r)$]
A generalized class number. Cf.\ (\ref{eqn:intThetaLcheckFK}).

\item
[$H(K,\check{F}_K)$]
A generalized class number. 
See (\ref{eqn:HKCcheckFK}).

\item
[$\mc{H}$]
The tensor product of the $\mc{H}(-n)$.
Cf.\ (\ref{eqn:mathcalFn}).

\item
[$\mc{H}_n$]
A homogeneous subspace of $\mc{H}$.
Cf.\ (\ref{eqn:mathcalFn}).

\item
[$\mc{H}(-n)$]
A certain $G$-module defined in terms of $U_n$.
See (\ref{eqn:mathcalFn}).

\item
[$k^W_g$]
A certain value defined 
for $g\in G$ 
and $W$ as in (\ref{eqn:res-for:W}). 
See (\ref{eqn:kg}) and cf.\ Theorem \ref{thm:STL}.

\item
[$K$]
A lattice defined by a choice of primitive norm zero vector $\ell$ in $L$.
See (\ref{eqn:K}--\ref{eqn:KKstarC}).

\item
[$\Kb$]
The generator for $K$ that is contained in its positive cone $C$.
See (\ref{eqn:KKstarC}).

\item
[$\Kb'$]
The generator for $K^*$ that is contained in its positive cone $C$.
See (\ref{eqn:KKstarC}).

\item
[$\ell$]
A primitive norm zero vector in $L$ 
in \S\S~\ref{sec:res-lif}--\ref{sec:res-CM}.
Cf.\ (\ref{eqn:K}). 

\item
[$\ell'$]
A vector in $L^*$ such that $(\ell,\ell')=1$, for $\ell$ as in \S\S~\ref{sec:res-lif}--\ref{sec:res-CM}. Cf.\ (\ref{eqn:K}).

\item 
[$L$] 
In \S\S~\ref{sec:res-lif}--\ref{sec:res-CM} we use $L$ to denote a specific realization of $L_{m,N}$. 
See (\ref{eqn:pre-lat:LmN2x2}).

\item
[$L_m$] Shorthand for the lattice $L_{m,1}$.
See 
(\ref{eqn:pre-lat:LmN}).

\item
[$L_{m,N}$] Shorthand for the lattice $\sqrt{2m}\ZZ\oplus\Gamma^{1,1}(N)$. 
See 
(\ref{eqn:pre-lat:LmN}).

\item
[$L(N)$] A lattice obtained by rescaling the quadratic form on $L$. 
Cf.\ (\ref{eqn:pre-lat:LmN}).

\item
[$\lambda$]
A vector in $L\otimes\RR$ for some lattice $L$. 
Cf.\ (\ref{eqn:thetaLgamma}).

\item
[$\lambda^\perp$]
A CM point in the upper half-plane.
Cf.\ (\ref{eqn:ZLlambda}) and (\ref{eqn:ZLtaulambda}).

\item
[$\lambda(c,b,a)$]
A vector in the lattice $L=L_{m,N}$ in \S\S~\ref{sec:res-lif}--\ref{sec:res-CM}.
See (\ref{eqn:pre-lat:LmN2x2}).

\item
[$\Lambda_t(U)$]
A generating function for exterior powers of a virtual $G$-module $U\in R(G)$.
See (\ref{eqn:LambdatSt}).

\item
[$m_\chi(U)$]
The multiplicity of $\chi\in \Irr(G)$ in a virtual $G$-module $U$. 
See (\ref{eqn:mchiU}).

\item
[$N_g$]
A positive integer attached to a group element $g\in G$, in the setup of \S~\ref{sec:res}.
Cf.\ (\ref{eqn:res-for:FWgnVwh}).

\item
[$P$]
A certain principal $\CC^*$-bundle over $G(L)$ for $L=L_{m,N}$.
See \S~\ref{sec:res-lif}, after (\ref{eqn:ZL}).

\item
[$\psi^k$]
An Adams operation on virtual $G$-modules, for 
$G$ a finite group. 
Cf.\ (\ref{eqn:LambdatUexp}--\ref{eqn:StUexp}).

\item 
[$\Psi^W_g$]
A twined $O_{2,1}$ Borcherds product. 
See (\ref{eqn:int-mod:PsiWg}), \eqref{eqn:res-for:PsiWg}.

\item
[$\Psi^W_{D_1,r_1}$]
A shorthand for $\Psi^W_{D_1,r_1,e}$. 
See (\ref{eqn:int-twi:PsiWDeltae}).

\item
[$\Psi^W_{D_1,r_1,g}$]
A twined twisted $O_{2,1}$-type Borcherds product.
See (\ref{eqn:int-twi:PsiWDeltag}).

\item
[$\Psi_\ell(\tau,\check{F})$]
A shorthand for $\Psi_\ell(Z_L(\tau),\check{F})$ in the special case that $\ell=\lambda(0,1,0)$.
See (\ref{eqn:PsielltaucheckF}).

\item
[$\Psi_\ell(Z,\check{F})$]
A modular form on $K\otimes \RR+iC$.  
See (\ref{eqn:PsiellZcheckF}).

\item
[$\Psi_L(Z_L,\check{F})$]
An automorphic form produced by a singular theta lift. 
See Theorem \ref{thm:STL}.

\item
[$Q$]
A binary quadratic form with integer coefficients.
Cf.\ (\ref{eqn:pre-hee:ZmD1r1}).

\item[$Q_L$]
The quadratic form 
of an even lattice $L$. 
See \S~\ref{sec:pre-lat}.

\item
[$\mc{Q}^{(m)}_{D,r}$]
A set of binary quadratic forms with integer coefficients. 
See \S~\ref{sec:pre-hee}.

\item
[$R(G)$]
The Grothendieck group of the category of finitely generated $\CC[G]$-modules. 
See (\ref{eqn:RG}).

\item
[$R^+(G)$]
The subset of $R(G)$ represented by finite-dimensional $\CC[G]$-modules.
See (\ref{eqn:RplusG}).

\item
[$\varrho_L$]
The Weil representation 
attached to an even lattice $L$. 
See \eqref{eqn:pre-lat:wei}.

\item
[$\varrho_m$]
A shorthand for $\varrho_{m,1}$. 
Cf.\ \eqref{eqn:pre-lat:weiLmN}.

\item
[$\varrho_{m,N}$]
A shorthand for $\varrho_{L_{m,N}}$. 
See \eqref{eqn:pre-lat:weiLmN}.

\item
[$\rho(K,\check{F}_K)$] 
A shorthand for $\rho(K,C,\check{F}_K)$. 
See (\ref{eqn:rhoKcheckFK}).

\item
[$\rho(K,C,\check{F}_K)$] 
A Weyl vector. Cf.\ (\ref{eqn:rhoKWcheckFK}) and (\ref{eqn:rhoKcheckFK}).

\item 
[$\widetilde{S}$]
A specific element of $\mpt(\ZZ)$.
See (\ref{eqn:tildeStildeT}).

\item
[$\SQ$]
Our lift of the singular theta lift.
See (\ref{eqn:int-mod:SQ}) and (\ref{eqn:res-for:SQW}).

\item
[$S_t(U)$]
A generating function for symmetric powers of a virtual $G$-module $U\in R(G)$.
See (\ref{eqn:LambdatSt}).

\item
[$\mpt(\RR)$]
The metaplectic double cover of $\SL_2(\RR)$. 
See (\ref{eqn:widetildeSL2R}--\ref{eqn:widetildeSL2R-mlt}).

\item
[$\mpt(\ZZ)$]
The preimage of $\SL_2(\ZZ)$ under the natural projection $\mpt(\RR)\to \SL_2(\RR)$. 
Cf.\ (\ref{eqn:tildeStildeT}).

\item 
[$\widetilde{T}$]
A specific element of $\mpt(\ZZ)$.
See (\ref{eqn:tildeStildeT}).

\item
[$T_{g}$]
A McKay--Thompson series of monstrous moonshine. 
See \S~\ref{sec:int-moo}.

\item
[$T^W_g$]
A McKay--Thompson series associated to $\SQ(W)$. 
See (\ref{eqn:res-for:TWg}) and Theorem \ref{thm:res}.

\item
[$T^{(m+m)}$]
The principal modulus associated to $\Gamma_0(m)+m$ when this group is genus zero. 
Cf.\ (\ref{eqn:out-inv:TmtauTmalpha}).

\item
[$\tr(g|U)$]
The trace of $g\in G$ on some $U\in R(G)$, or on a power series $U\in R(G)[[t]]$. 
See \S~\ref{sec:pre-ada}.

\item
[$\theta_m^0$]
The 
vector-valued modular form $\theta_m^0:=(\theta_{m,r}^0)$. 
Cf.\ (\ref{eqn:pre-mod:thetanull}).

\item
[$\theta_{m,r}$]
An certain $2$-variable theta function. 
Cf.\ (\ref{eqn:pre-mod:phi}).

\item
[$\theta^{0}_{m,r}$]
An index $m$ Thetanullwert. 
See (\ref{eqn:pre-mod:thetanull}).

\item
[$\theta_{L+\delta}$]
A component of the Siegel theta function $\Theta_L$.
See (\ref{eqn:thetaLgamma}).

\item
[$\Theta_L$] The Siegel theta function attached to a lattice $L$.
See (\ref{eqn:ThetaL}--\ref{eqn:thetaLgamma}).

\item
[$\Theta^{(m)}$]
A shorthand for $\Theta^{(m)}_1$. 
See \S~\ref{sec:pre-mod}.

\item
[$\Theta^{(m)}_N$]
A shorthand for $\Theta_L$ in case $L=L_{m,N}$. 
See \S~\ref{sec:pre-mod}.

\item
[$\tau_1$]
The real part of $\tau\in \HH$. Cf.\ (\ref{eqn:int-pro:int}), \ref{eqn:int-twi:int}, (\ref{eqn:res-lif:stl}) and (\ref{eqn:rhoKWcheckFK}--\ref{eqn:intThetaLcheckFK}).

\item
[$\tau_2$]
The imaginary part of $\tau\in \HH$. Cf.\ (\ref{eqn:int-pro:int}), \ref{eqn:int-twi:int}, (\ref{eqn:res-lif:stl}) and (\ref{eqn:rhoKWcheckFK}--\ref{eqn:intThetaLcheckFK}).

\item
[$u_d(g)$]
A shorthand for $u_d(g|U)$. 
See the proof of Lemma \ref{lem:vngp}.

\item
[$u_d(g|U)$]
A positive  
integer defined for 
$U\in R(G)$ when $\tr(g|U)$ is a rational integer.
Cf.\ (\ref{eqn:udg}).

\item
[$U_n$]
A certain virtual $G$-module defined in terms of $W$. See (\ref{eqn:Un}).

\item
[$U_\chi$]
The 
class of irreducible $G$-modules that realizes an irreducible character $\chi$ of $G$.
Cf.\ (\ref{eqn:RG}).

\item
[$U^b, U^f$]
We define operations $U\mapsto U^b$ and $U\mapsto U^f$ for $U\in R(G)$ in \S~\ref{sec:pre-ada}.
See (\ref{eqn:UfUb}).

\item
[$U(-n)$]
A copy of $U\in R(G)$, for $n$ a positive integer.
Cf.\ (\ref{eqn:res-con:FL0}).

\item
[$v^+$]
A 
positive definite subspace of $L\otimes\RR$. 
Cf.\ (\ref{eqn:ThetaL}--\ref{eqn:thetaLgamma}).

\item
[$v_b(g)$]
A shorthand for $v_b(g|U)$.
See Lemma \ref{lem:vngp}.

\item
[$v_b(g|U)$]
An 
integer defined for 
$U\in R(G)$ when $\tr(g|U)$ is a rational integer.
See (\ref{eqn:vng}--\ref{eqn:pig}).

\item
[$V^W$]
A realization of the 
virtual graded $G$-module 
$\SQ(W)$.
Cf.\ (\ref{eqn:res-for:SQW}) and (\ref{eqn:res-con:VW}).

\item
[$\textsl{V}^{\wh}_{\frac12,m}(N)$]
A certain space of weakly holomorphic vector-valued modular forms. See \S~\ref{sec:pre-mod}.

\item
[$\mathbb{V}^{\wh}_{\frac12,m}(N)$]
A certain space of weakly holomorphic vector-valued mock modular forms. See \S~\ref{sec:pre-mod}.

\item
[$W$]
A weakly holomorphic $G$-module of weight $\frac12$ and some index. 
See (\ref{eqn:int-mod:W}) and (\ref{eqn:res-for:W}).

\item
[$W_{r,\frac{D}{4m}}$]
A homogeneous component of $W$. See (\ref{eqn:res-for:W}).

\item
[$W^{(-3,1)}$]
A form of the penumbral Thompson moonshine module.
Cf.\ (\ref{eqn:int-mod:W31}).

\item
[$\breve W^{(-3,1)}$]
A form of the penumbral Thompson moonshine module.
See (\ref{eqn:int-moo:breveW31}).

\item
[$\breve W^{(-3,1)}_D$]
A homogeneous component of $\breve W^{(-3,1)}$. 
See (\ref{eqn:int-moo:breveW31}).

\item
[$W^{(-3,1)}_{r,\frac{D}4}$]
A homogeneous component of $W^{(-3,1)}$.
See (\ref{eqn:int-mod:W31}).

\item
[$(X_L,Y_L)$]
An ordered basis for a given $v^+\in G(L)$, for $L=L_{m,N}$.
Cf.\ (\ref{eqn:particularXLYL}--\ref{eqn:ZL}).

\item
[$\chi^{(m)}_{D_1}$]
A generalized genus character. See (\ref{eqn:pre-hee:chiD0N}).

\item
[$Z$]
An element of $K\otimes\RR+iC$.
Cf.\ (\ref{eqn:ZtoZL}--\ref{eqn:PsiellZcheckF}).

\item 
[$Z_L$] 
The image of $Z$ under a particular map from $K\otimes \RR+iC$ to $P$.
See (\ref{eqn:ZtoZL}) and cf.\ (\ref{eqn:ZL}).

\item
[$Z(\tau)$]
The image of $\tau\in \HH$ under a specific isomorphism of $\HH$ with $K\otimes \RR+iC$.
See (\ref{eqn:Ztau}).

\item 
[$Z_L(\tau)$] 
The image of $Z(\tau)$ under the map $Z\mapsto Z_L$. 
See (\ref{eqn:ZLtau}).

\item
[$Z^W_{D_1,r_1}$]
A twisted Heegner divisor associated to $W$ as in (\ref{eqn:res-for:W}). 
See (\ref{eqn:out-inv:ZWD1r1}).

\item
[$Z^{(m)}_{D_1,r_1}(D,r)$]
A twisted Heegner divisor. See (\ref{eqn:pre-hee:ZmD1r1}).

\end{list}

\end{footnotesize}

\section{Preliminaries}\label{sec:preliminaries}

In this section we review some preliminary notions in preparation for the arguments of \S~\ref{sec:res}.
We review the Weil representations of the metaplectic double cover of the modular group that are defined by even lattices in \S~\ref{sec:pre-lat}, 
in \S~\ref{sec:pre-mod} we define families of modular forms which arise from these representations, 
in \S~\ref{sec:pre-hee} we introduce some twisted Heegner divisors, and we 
finally review Adams operations in \S~\ref{sec:pre-ada}.

\subsection{Weil Representations}\label{sec:pre-lat}

Here we review some facts about 
Weil representations of the metaplectic group. 
For this recall that $\SL_2(\RR)$ admits an extension
\begin{gather}\label{eqn:widetildeSL2R}
\widetilde{\SL}_2(\RR) := \left.\left\{ (\gamma,u)\,\right|\, \gamma\in \SL_2(\RR),\, u(\tau)^4{\rm d}(\gamma\tau)={\rm d}\tau     \right\},
\end{gather}
where the $u$ in each pair $(\gamma,u)$ in (\ref{eqn:widetildeSL2R}) is assumed to be a holomorphic function on the complex upper half-plane, $\HH$.
Thus,
using $\sqrt{\cdot}$ 
to denote the branch of the square root function determined by requiring that $\sqrt{e^{2\pi it}}=e^{\pi i t}$ for $-\frac12 < t\leq \frac12$, we
either have $u(\tau)=\sqrt{c\tau+d}$ or $u(\tau)=-\sqrt{c\tau+d}$ when $(c,d)$ is the lower row of $\gamma$.

The multiplication rule in $\widetilde{\SL}_2(\RR)$ is
\begin{gather}\label{eqn:widetildeSL2R-mlt}
(\gamma_1,u_1)(\gamma_2,u_2) = \big(\gamma_1\gamma_2,(u_1\circ\gamma_2)u_2\big),
\end{gather}
and we call $\widetilde{\SL}_2(\RR)$ the metaplectic double cover of $\SL_2(\RR)$. 

We write $\widetilde{\Gamma}_0(N)$ and $\widetilde{\Gamma}_1(N)$ for the preimages of $\Gamma_0(N)$ and $\Gamma_1(N)$
under the natural map $\widetilde{\SL}_2(\RR)\to \SL_2(\RR)$, respectively, where
\begin{gather}
\begin{split}\label{eqn:Gamma0nGamma1n}
\Gamma_0(N) &:= \left.\left\{\left(\begin{matrix} a & b \\ c & d\end{matrix}\right)\in \SL_2(\ZZ) \,\right|\,  c \equiv 0  \xmod N\right\}, \\
\Gamma_1(N) &:= \left.\left\{\left(\begin{matrix} a & b \\ c & d\end{matrix}\right)\in \SL_2(\ZZ) \,\right|\,  a,d\equiv 1 \xmod N, \  c \equiv 0  \xmod N\right\}.
\end{split}
\end{gather}
We also write $\widetilde{\SL}_2(\ZZ)$ for the preimage of the modular group $\SL_2(\ZZ)=\Gamma_0(1)=\Gamma_1(1)$.
This group
$\widetilde{\SL}_2(\ZZ)$ 
is generated by the elements $\widetilde{S}$ and $\widetilde{T}$, where
\begin{gather}\label{eqn:tildeStildeT}
\widetilde{S}:=\left(\left(\begin{matrix}0&-1\\1&0\end{matrix}\right),\sqrt{\tau}\right),\quad
\widetilde{T}:=\left(\left(\begin{matrix}1&1\\0&1\end{matrix}\right),1\right).
\end{gather} 
More generally, given $\Gamma<\SL_2(\RR)$ we write $\widetilde{\Gamma}$ for its preimage in $\widetilde{\SL}_2(\RR)$.

Next recall that a lattice $L$ is a free, finitely generated Abelian group 
equipped with a 
symmetric bilinear form $(\cdot\,,\cdot):L\times L\to\mathbb{R}$. 
We restrict our attention to even lattices (i.e.\ those which have $(\lambda,\lambda)\in 2\ZZ$ for each $\lambda$ in $L$) for which $(\cdot\,,\cdot)$ is nondegenerate, and associate to such a lattice the integer-valued quadratic form $Q_L(\lambda) := \frac{(\lambda,\lambda)}{2}$. 
We denote the dual lattice by $L^\ast := {\hom}(L,\ZZ)$ and refer to $L^\ast/L$ as the discriminant group of $L$, where we have implicitly made use of the inclusion $L\xhookrightarrow{} L^\ast$ induced by the bilinear form. 
The automorphism group of $L$ is denoted $\Aut(L)$; it acts on the discriminant group $L^\ast/L$, and therefore on $\mathbb{C}[L^\ast/L]$-valued functions as well,
where 
\begin{gather}\label{eqn:CLstarL}
\mathbb{C}[L^\ast/L] = \mathrm{Span}_\CC\{e_\delta \mid \delta \in L^\ast/L\}
\end{gather} 
is the group algebra of $L^*/L$. 
For $\check{F}$ a $\CC[L^\ast/L]$-valued function we use $\Aut(L,\check{F})$ to denote the subgroup of $\Aut(L)$ which fixes $\check{F}$.

There is one particular family of lattices that will appear repeatedly in the rest of the paper. 
To define this family we write $\sqrt{2m}\ZZ$ for the lattice of rank 1 whose quadratic form $\ZZ\to\ZZ$ is $k\mapsto mk^2$, and $\Gamma^{1,1}$ for the unique unimodular even lattice of signature $(1,1)$,
taking the quadratic form $\ZZ\times\ZZ\to\ZZ$ to be $(u,v) \mapsto uv$. 
We obtain a new lattice $L(N)$ from any given lattice $L$ by defining it to have the same underlying $\ZZ$-module, but a quadratic form which is scaled with respect to that of $L$, so that $Q_{L(N)}(\lambda) := NQ_L(\lambda)$. 
From these building blocks we define the family of interest by setting
\begin{gather}\label{eqn:pre-lat:LmN}
L_{m,N} := \sqrt{2m}\ZZ\oplus \Gamma^{1,1}(N),
\end{gather}
for arbitrary non-zero integers $m$ and $N$.

To {any} even lattice $L$ we can associate a representation $\varrho_L$ of the metaplectic group $\widetilde{\SL}_2(\ZZ)$ (\ref{eqn:widetildeSL2R}--\ref{eqn:widetildeSL2R-mlt}) called the Weil representation associated to $L$. 
It acts on the group algebra $\CC[L^*/L]$ 
according to the rules
\begin{gather}\label{eqn:pre-lat:wei}
\varrho_L(\widetilde{T})e_\delta = \ex\big(Q_L(\delta)\big)e_\delta, \quad
\varrho_L(\widetilde{S})e_\delta = 
\frac{\ex(\frac18(b^--b^+))}
{\sqrt{|L^\ast/L|}}\sum_{\delta'\in L^\ast/L}\ex\big(-{(\delta,\delta')}\big)e_{\delta'}
\end{gather}
(cf.\ (\ref{eqn:tildeStildeT})),
where $(b^+,b^-)$ is the signature of $L$. For the choice $L=L_{m,N}$ the discriminant group is 
$L_{m,N}^\ast/L_{m,N} \cong \ZZ/N\ZZ\times\ZZ/N\ZZ\times \ZZ/2m\ZZ$ 
and the Weil representation becomes 
\begin{gather}\label{eqn:pre-lat:weiLmN}
\begin{split}
\varrho_{m,N}(\widetilde{T})e_{i,j,r} &= \ex\left(\frac{r^2}{4m}+\frac{ij}{N}\right)e_{i,j,r}, \\  
\varrho_{m,N}(\widetilde{S})e_{i,j,r} &= 
\frac{\ex(-\sgn(m)\frac18)}
{\sqrt{|2mN^2|}}
\sum_{i'\in\ZZ/N\ZZ}\sum_{j'\in\ZZ/N\ZZ}\sum_{r'\in\ZZ/2m\ZZ}\ex\left(-\frac{rr'}{2m}-\frac{(ij'+ji')}{N}\right)e_{i',j',r'}.
\end{split}
\end{gather}
When $N=1$ we make the abbreviations $L_m:=L_{m,1}$ and $\varrho_m := \varrho_{m,1}$. 

We are mostly concerned with the case that $m$ and $N$ are both positive in what follows, so that $L_{m,N}$ has signature $(2,1)$, but we note here that the definitions (\ref{eqn:pre-lat:LmN}--\ref{eqn:pre-lat:weiLmN}) of $L_{m,N}$ and $\varrho_{m,N}$ make sense for arbitrary non-zero integer values of $m$ and $N$. 
The lattices $L_{m,N}$ and $L_{m,-N}$ are isomorphic, so there is no loss of generality in restricting to positive $N$, but $L_{-m,N}$ and $L_{m,N}$ are different. Indeed, $\varrho_{-m,N}$ is the dual of the representation $\varrho_{m,N}$. 
The case that $m$ is negative 
plays an important role 
in the discussion of \S~\ref{sec:out-imp}.

\subsection{Modular Forms}\label{sec:pre-mod}

Here we define some spaces of modular forms.
Let $k^+,k^-\in\frac12\ZZ$, let $\Gamma$ be a subgroup of $\SL_2(\mathbb{R})$, let $V$ be a complex vector space, and let $\varrho:\widetilde{\Gamma}\to\GL(V)$ be a representation. 
Then a real analytic function $f:\mathbb{H}\to V$ is called a modular form of weight $(k^+,k^-)$ and type $\varrho$ for $\widetilde{\Gamma}$ if it transforms according to the equation
\begin{gather}\label{eqn:pre-mod:typevarrho}
f(\gamma\tau) = u(\tau)^{2k^+}u(\bar\tau)^{2k^-}\varrho(\gamma,u)f(\tau), \ \ \ \ \ \ \ \  (\gamma,u)\in\widetilde\Gamma.
\end{gather}
We further call $f$ a weakly holomorphic modular form of weight $k^+$ and type $\varrho$ for $\widetilde{\Gamma}$ if $f$ is holomorphic and $k^-=0$, 
and if for any $(\gamma,u)\in\widetilde{\SL}_2(\ZZ)$ there is a $C>0$ such that 
\begin{gather}\label{eqn:pre-mod:wh}
u(\tau)^{-2k^+}\alpha\left(\varrho(\gamma,u)^{-1}f(\gamma\tau)\right)=O(e^{C\Im(\tau)})
\end{gather} 
as $\Im(\tau)\to\infty$, for any linear functional $\alpha:V\to \mathbb{C}$.  
A holomorphic modular form of weight $k^+$ and type $\varrho$ for $\widetilde{\Gamma}$ satisfies the same conditions but with the exponential growth condition (\ref{eqn:pre-mod:wh}) replaced by boundedness.

If $k^+$ and $k^-$ are both integers then the transformation rule (\ref{eqn:pre-mod:typevarrho}) is independent of the choice of $u$ in $(\gamma,u)\in\widetilde{\Gamma}$, and so we may unambiguously speak of modular forms for $\Gamma<\SL_2(\RR)$ in this case. 
If $\varrho$ is trivial and $\Gamma$ contains $-\Id$ then $k^+$ and $k^-$ must both be even integers in order for non-trivial modular forms of weight $(k^+,k^-)$ and type $\varrho$ for $\Gamma$ to exist. 
By a weakly holomorphic modular form of weight $k^+$ for $\G$ we mean a weakly holomorphic modular form of weight $k^+$ and type $\varrho$ for $\G$, where $\varrho$ is the trivial representation.

In this work we are primarily interested in the case that $\varrho$ is trivial, or the Weil representation associated to an even lattice $L$.
A key example of a modular form for such a representation $\varrho=\varrho_L$ is the Siegel theta function attached to $L$. To define this let
$(b^+,b^-)$ 
be the signature of $L$
and let $G(L)$ be the Grassmannian of positive definite subspaces $v^+$ of $L\otimes \RR$. 
Letting $\lambda_+$ and $\lambda_-$ be the projection of $\lambda\in L\otimes\mathbb{R}$ onto $v^+$ and its (negative definite) orthogonal complement, respectively, we define the Siegel theta function
\begin{gather}\label{eqn:ThetaL}
\Theta_L(\tau,v^+) := \sum_{\delta\in L^\ast/L} e_\delta \theta_{L+\delta}(\tau,v^+),
\end{gather}
with components given by
\begin{gather}\label{eqn:thetaLgamma}
\theta_{L+\delta}(\tau,v^+) := 
\sum_{\lambda\in L+\delta}\ex\left(\tau Q_L(\lambda_+)+\bar\tau Q_L(\lambda_-)\right).
\end{gather}
(In (\ref{eqn:ThetaL}--\ref{eqn:thetaLgamma}) we abuse notation by using $L+\delta$ to denote the coset of $L$ in $L^*$ specified by $\delta\in L^*/L$. No confusion should arise from this.) 
Then from Theorem 4.1 of \cite{Borcherds:1996uda}, for example, 
we have that
\begin{gather}
\Theta_L\left(\gamma\tau,v^+\right) = u(\tau)^{b^+}u(\bar\tau)^{b^-}\varrho_L\left(\gamma,u\right)\Theta_L(\tau,v^+) 
\end{gather}
for each $(\gamma,u)\in \widetilde{\SL}_2(\ZZ)$. 
Thus the construction (\ref{eqn:ThetaL}--\ref{eqn:thetaLgamma}) 
furnishes an example of a modular form of weight $(\frac{b^+}{2},\frac{b^-}{2})$ and type $\varrho_L$ for the metaplectic double cover $\widetilde{\SL}_2(\ZZ)$ of the modular group. 
We also have the feature that $\Theta_L(\tau,v^+)$ is invariant with respect to the action of $\Aut(L)$ on the argument $v^+$.

In this work the Siegel theta functions associated to the lattices $L_{m,N}$ of (\ref{eqn:pre-lat:LmN}) play an important role. 
We write $\Theta^{(m)}_N$ as a shorthand for $\Theta_L$ in case $L=L_{m,N}$, and write $\Theta^{(m)}$ 
as a shorthand for $\Theta^{(m)}_1$.

When $b^-=0$ the Siegel theta function (\ref{eqn:ThetaL}) specializes to a holomorphic modular form of weight $\frac{b^+}{2}$ and type $\varrho_L$. For example, in the case that $L=\sqrt{2m}\ZZ$, the Grassmannian $G(L)$ is a point, $v^+=\sqrt{2m}\ZZ\otimes\RR$, and the 
components 
$\theta_{\sqrt{2m}\ZZ+\frac{r}{\sqrt{2m}}}$
of the Siegel theta function $\Theta_{\sqrt{2m}\ZZ}$ 
reduce to the Thetanullwerte
\begin{gather}\label{eqn:pre-mod:thetanull}
\theta^0_{m,r}(\tau) 
:= \theta_{\sqrt{2m}\ZZ+\frac{r}{\sqrt{2m}}}(\tau)
= \sum_{s\equiv r \xmod 2m}q^{\frac{s^2}{4m}} .  
\end{gather}
Here we follow \cite{Borcherds:1996uda} in dropping the variable $v^+$ from notation when $G(L)$ is a point. 
In the remainder we write $\theta_m^0=(\theta^0_{m,r})$ as a shorthand for $\Theta_{\sqrt{2m}\ZZ}$.

Recall from \S~\ref{sec:pre-lat} that $\varrho_m$ denotes the Weil representation defined by $L_m=L_{m,1}$ (cf.\ (\ref{eqn:pre-lat:weiLmN})). 
We write
$\textsl{V}^{\wh}_{\frac12,m}(N)$
for the space of weakly holomorphic modular forms of weight $\frac12$ and type $\varrho_m$ for $\widetilde{\Gamma}_0(N)$ (see (\ref{eqn:pre-mod:typevarrho}--\ref{eqn:pre-mod:wh})).
Thus, for $m$ positive, an element $F=(F_r)$ of $\textsl{V}^{\wh}_{\frac12,m}(N)$ is a 
vector-valued function whose components $F_r$, indexed by integers $r$ modulo $2m$, admit Fourier expansions of the form 
\begin{gather}\label{eqn:pre-mod:Frtau}
F_r(\tau) = \sum_{{D\equiv r^2\xmod 4m}} C_F(D,r) q^{\frac{D}{4m}},
\end{gather}
with $C_F(D,r)=0$ for $D\ll 0$. On the other hand, if $m$ is negative then $F\in \textsl{V}^{\wh}_{\frac12,m}(N)$ still has components $F_r$ indexed by $r\xmod 2m$, and the Fourier expansion is still of the form (\ref{eqn:pre-mod:Frtau}), but with $m$ negative the condition that $F$ be weakly holomorphic (see (\ref{eqn:pre-mod:wh})) requires that
$C_F(D,r)=0$ for $D\gg 0$.

We mention another simple but significant difference between the cases that $m$ is positive and negative in $\textsl{V}^{\wh}_{\frac12,m}(N)$. Namely, using two applications of the second line of (\ref{eqn:pre-lat:weiLmN}) to compute (\ref{eqn:pre-mod:typevarrho}) for $(k^+,k^-)=(\frac12,0)$ and $\varrho=\varrho_m$ with $(\gamma,u)=(-I,\sqrt{-1})$, we find that for $F=(F_r)$ in $\textsl{V}^\wh_{\frac12,m}(N)$ 
we have 
\begin{gather}\label{eqn:pre-mod:Fminusr}
	F_{-r}=
	\begin{cases}
	F_{r}&\text{ if $m>0$,}\\
	-F_{r}&\text{ if $m<0$,}
	\end{cases}
\end{gather} 
for all $r\xmod 2m$. Thus, for example, $\textsl{V}^\wh_{\frac12,-1}(N)$ vanishes for all $N$.

To help put the $\textsl{V}^\wh_{\frac12,m}(N)$ in perspective we point out that 
for $m$ positive $\textsl{V}^\wh_{\frac12,-m}(N)$ is naturally isomorphic to the space of weakly holomorphic Jacobi forms of weight $1$ and index $m$ for $\Gamma_0(N)$,  
and $\textsl{V}^\wh_{\frac12,m}(N)$ is naturally isomorphic to the complex conjugate of the space of weakly skew-holomorphic Jacobi forms of weight $1$ and index $m$ for $\Gamma_0(N)$. 
Indeed, for $m>0$ we obtain the weakly holomorphic Jacobi form $\phi$ corresponding to $F=(F_r)\in \textsl{V}^\wh_{\frac12,-m}(N)$ by setting
\begin{gather}\label{eqn:pre-mod:phi}
	\phi(\tau,z) := \sum_{r\xmod 2m}F_r(\tau)\theta_{m,r}(\tau,z),
\end{gather}
where 
$\theta_{m,r}(\tau,z):=\sum_{s\equiv r\xmod 2m}
\ex(sz)
q^{\frac{s^2}{4m}}$ 
(cf.\ (\ref{eqn:pre-mod:thetanull})), 
and the weakly skew-holomorphic Jacobi form $\varphi$ corresponding to a given $F=(F_r)\in \textsl{V}^\wh_{\frac12,m}(N)$ is 
\begin{gather}\label{eqn:pre-mod:varrphi}
	\varphi(\tau,z) := \sum_{r\xmod 2m}\overline{F_r(\tau)}\theta_{m,r}(\tau,z).
\end{gather}
We refer to \S~3.1 of \cite{Cheng:2016klu} and \S~3.3 of \cite{pmo} for concise introductory accounts of holomorphic and skew-holomorphic Jacobi forms, formulated with moonshine in mind.

To conclude this section 
we note that everything we have said about 
the spaces $\textsl{V}^\wh_{\frac12,m}(N)$ applies equally well when we replace modularity with mock modularity in the definition. 
In particular, the identity (\ref{eqn:pre-mod:Fminusr}) holds when $F$ belongs to the space $\mathbb{V}^\wh_{\frac12,m}(N)$ of mock modular forms of weight $\frac12$ and type $\varrho_m$ for $\widetilde{\Gamma}_0(N)$, and such an $F$ is naturally identified with a mock Jacobi form of weight $1$ 
for $\Gamma_0(N)$ that is either weakly holomorphic (\ref{eqn:pre-mod:phi}), or weakly skew-holomorphic (\ref{eqn:pre-mod:varrphi}), depending on the sign of $m$.
We refrain from defining mock modular forms and mock Jacobi forms here, and instead refer to \S~3.2 of \cite{Cheng:2016klu}, and references cited therein, for more details.

\subsection{Heegner Divisors}\label{sec:pre-hee}

In this section we recall the notion of twisted Heegner divisor, as formulated in \cite{BruinierOno2010}, specifically in preparation for the discussion in \S~\ref{sec:out-inv}.

Following \cite{GKZ87} we write $[A,B,C]$ as a shorthand for the binary quadratic form $Ax^2+Bxy+Cy^2$. 
The discriminant of $[A,B,C]$ is $D=B^2-4AC$, and
there is a discriminant-preserving right action of $\SL_2(\RR)$ on binary quadratic forms with real coefficients determined by requiring that
\begin{gather}\label{eqn:pre-hee:ABCabcdxy}
	[A,B,C]
	\left(
	\begin{smallmatrix}a&b\\c&d\end{smallmatrix}
	\right)
	(x,y) = [A,B,C](ax+by,cx+dy).
\end{gather}

If $D=B^2-4AC$ is negative then there is a unique point $\alpha\in\HH$ such that 
\begin{gather}\label{eqn:pre-hee:alpha}
[A,B,C](\a,1)=0.
\end{gather} 
The right action (\ref{eqn:pre-hee:ABCabcdxy}) is compatible with the usual left action of $\SL_2(\RR)$ on $\HH$ in that 
if $\a$ is the point in $\HH$ associated to $[A,B,C]$ as in (\ref{eqn:pre-hee:alpha}) and $\gamma\in \SL_2(\RR)$, then 
the point in $\HH$ associated to $[A,B,C]\gamma$ is
$\gamma^{-1}\a$.

Henceforth we restrict to binary quadratic forms 
with integer coefficients and negative discriminant. 
Given integers $m$ and $r$, and a fixed choice of $D<0$, 
let $\mc{Q}^{(m)}_{D,r}$ denote the set of integer-coefficient binary quadratic forms $[A,B,C]$ of discriminant $D$ such that $A>0$ and $A\equiv 0\xmod m$ and $B\equiv r\xmod 2m$. 
Then each set $\mc{Q}^{(m)}_{D,r}$ is stable under the restriction of the action (\ref{eqn:pre-hee:ABCabcdxy}) to $\Gamma_0(m)$, and 
is composed of finitely many orbits for this action. Also, any point $\a\in \HH$ that is associated to a quadratic form $[A,B,C]$ as in (\ref{eqn:pre-hee:alpha}) with $A,B,C\in \ZZ$ is called a CM point. (At negative discriminants there is no loss of generality in restricting to quadratic forms $[A,B,C]$ with $A>0$.)

Now suppose that $D_1$ is a positive fundamental discriminant that is a 
square modulo $4m$. 
Applying Proposition 1 of \S~I.2 of op.\ cit.\ we may define the 
generalized 
genus character 
$\chi^{(m)}_{D_1}$ on 
$\mc{Q}^{(m)}_{DD_1}$,
for any negative discriminant $D$ that is a square modulo $4m$,
by setting
\begin{gather}\label{eqn:pre-hee:chiD0N}
\chi^{(m)}_{D_1}([Am,B,C]) := \left(
\frac{D_1'}{Am'}
\right)
\left(
\frac{D_1''}{Cm''}
\right)
\end{gather}
for $[Am,B,C]\in\mc{Q}^{(m)}_{DD_1}$, 
when there exist discriminants $D_1'$, $D_1''$ and positive integers $m'$, $m''$ 
such that $D_1=D_1'D_1''$ and $m=m'm''$ 
and $\gcd(D_1',Am')=\gcd(D_1'',Cm'')=1$, and by setting 
$\chi^{(m)}_{D_1}(Q) := 0$ 
when no such $D_1'$, $D_1''$, $m'$ and $m''$ 
exist.
According to Proposition 1 in \S~I.2 of \cite{GKZ87} 
this generalized genus character (\ref{eqn:pre-hee:chiD0N}) is invariant under the action of $\Gamma_0(m)$ in the sense that $\chi^{(m)}_{D_1}(Q\gamma)=\chi^{(m)}_{D_1}(Q)$ for $Q=[Am,B,C]$ as in (\ref{eqn:pre-hee:chiD0N}) and $\gamma\in \Gamma_0(m)$, so it descends to a well-defined function on the orbit space $\mc{Q}^{(m)}_{DD_1,r}/\Gamma_0(m)$ for each $r\xmod 2m$.

Next let $X_0(m)$ denote the modular curve 
\begin{gather}\label{eqn:int-hee:X0m}
X_0(m):=\Gamma_0(m)\backslash \HH\cup\QQ\cup\{\infty\}
\end{gather} 
associated to $\Gamma_0(m)$, and for $D_1$ and $D$ as above, define the twisted Heegner divisor $Z^{(m)}_{D_1,r_1}(D,r)$ on $X_0(m)$ by setting
\begin{gather}\label{eqn:pre-hee:ZmD1r1}
	Z^{(m)}_{D_1,r_1}(D,r):=\sum_{Q\in\mc{Q}^{(m)}_{DD_1,rr_1}}
	\chi^{(m)}_{D_1}(Q)
	\frac{\overline{\a_Q}}{|\overline{\Gamma_0(m)}_Q|}
	,
\end{gather}
where in 
(\ref{eqn:pre-hee:ZmD1r1}) we write 
$\overline{\Gamma_0(m)}_Q$ for the stabilizer of $Q$ in $\overline{\Gamma_0(m)}:=\Gamma_0(m)/\{\pm \Id\}$, and 
in each summand
take $\overline{\a_Q}$ to be the image under the natural map $\HH\to X_0(m)$ of the CM point $\a_Q=\a\in \HH$ associated to $Q=[A,B,C]$ as in (\ref{eqn:pre-hee:alpha}).

Note that $Z^{(m)}_{D_1,r_1}(D,r)$ as in (\ref{eqn:pre-hee:ZmD1r1}) is generally not an integral divisor, but will be in the cases of interest to us, whereby $D_1>1$. 
This is because 
if $D_1>1$ is fundamental then we have $|\overline{\Gamma_0(m)}_Q|=1$ for every $Q\in\mc{Q}^{(m)}_{DD_1,rr_1}$.

\subsection{Adams Operations}\label{sec:pre-ada}

In this section we review Adams operations 
in the context of complex representations of a finite group.
To set up for this suppose that $G$ is a finite group, 
write $\CC[G]$ for the complex group algebra of $G$ (cf.\ (\ref{eqn:CLstarL})), 
and let $R(G)$ denote the 
Grothendieck group of the category of finitely generated $\CC[G]$-modules.
Then, writing $\Irr(G)$ for the set of irreducible characters of $G$, and writing $U_\chi$ for the isomorphism class of irreducible modules that realize a given $\chi\in \Irr(G)$, we have that $R(G)$ is naturally identified with the free $\ZZ$-module 
generated by the $U_\chi$, i.e.\ 
\begin{gather}\label{eqn:RG}
	R(G) = \sum_{\chi\in \Irr(G)}\ZZ U_\chi.
\end{gather}
The 
finitely generated (i.e.\ finite-dimensional) $\CC[G]$-modules, regarded modulo isomorphisms, may be identified with the semigroup 
$R^+(G)<R(G)$ 
of non-negative integer combinations of 
the $U_\chi$, 
\begin{gather}\label{eqn:RplusG}
	R^+(G) = \left.\left\{\sum_{\chi\in \Irr(G)} m_\chi U_\chi\in R(G)\,\right|\, m_\chi\geq 0\right\}.
\end{gather} 
Also, the tensor product operation on $\CC[G]$-modules defines a commutative semiring structure on $R^+(G)$, and this induces a commutative ring structure on $R(G)$.
We use $+$ to denote the group operation on $R(G)$, and use $\otimes$ to denote the multiplication.

By a virtual $G$-module we mean an element of $R(G)$. 
By a virtual graded $G$-module we mean an indexed collection $\{V_i\}_{i\in I}$ of virtual modules $V_i\in R(G)$, 
for some indexing set $I$, 
but we write 
\begin{gather}\label{eqn:virtualgraded}
V=\bigoplus_{i\in I}V_i
\end{gather} 
for such a thing, 
rather than $V=\{V_i\}_{i\in I}$.

Given 
a virtual $G$-module $U\in R(G)$ we define the multiplicity $m_\chi(U)$ of $\chi$ in $U$, 
for each $\chi\in \Irr(G)$,
by requiring that
\begin{gather}\label{eqn:mchiU}
	U= \sum_{\chi\in \Irr(G)} m_\chi(U)U_\chi
\end{gather}
(cf.\ (\ref{eqn:RG})). We also define $U^f$ and $U^b$ for $U\in R(G)$ by 
setting
\begin{gather}\label{eqn:UfUb}
	U^f
	:=\sum_{\substack{\chi\in \Irr(G)\\m_\chi(U)>0}}m_\chi(U) U_\chi,
	\quad
	U^b
	:=\sum_{\substack{\chi\in\Irr(G)\\m_\chi(U)<0}}(-m_\chi(U))U_\chi,
\end{gather}
so that $U^f,U^b\in R^+(G)$ and $U=U^f-U^b$.

Now taking $t$ to be an indeterminate, let us write $R(G)[[t]]$ for the ring of power series in $t$ with coefficients in $R(G)$, 
and for $U\in R^+(G)$ define elements $\Lambda_t(U)$ and $S_t(U)$ of $R(G)[[t]]$ by setting
\begin{gather}
\label{eqn:LambdatSt}
	\Lambda_t(U)
	:=\sum_{k\geq 0} \Lambda^k(U)t^k,
	\quad
	S_t(U)
	:=\sum_{k\geq 0 }S^k(U)t^k,
\end{gather}
where $\Lambda^k(U)$ and $S^k(U)$ denote the $k$-th exterior and symmetric powers of $U$, respectively. 
(Of course $\Lambda_t(U)$ actually belongs to the polynomial ring $R(G)[t]$ when $U\in R^+(G)$, but see (\ref{eqn:Lambdat-ext}) below.)
Note that we have
\begin{gather}\label{eqn:LambdatUpUpp}
\Lambda_t(U'+ U'')=\Lambda_t(U')\otimes \Lambda_t(U'')
\end{gather}
in $R(G)[[t]]$ for $U',U''\in R^+(G)$, and similarly with $S_t$ in place of $\Lambda_t$.

We now define the Adams operations on $R(G)$ to be the ring homomorphisms $\psi^k$ of $R(G)$, defined for each positive integer $k$ by requiring that 
\begin{gather}\label{eqn:LambdatUexp}
	\Lambda_{-t}(U) = \exp\left(-\sum_{k>0}\psi^k(U)\frac{t^k}{k}\right)
\end{gather}
for $U\in R^+(G)$
(cf.\ e.g.\ Exercise 9.3 in \cite{serrereps}). 
Note that with this definition we also have
\begin{gather}\label{eqn:StUexp}
	S_t(U) = \exp\left(\sum_{k>0}\psi^k(U)\frac{t^k}{k}\right)
\end{gather}
in $R(G)[[t]]$, for $U\in R^+(G)$ (cf.\ loc.\ cit.), 
and it follows that $\Lambda_{-t}(U)\otimes S_{t}(U)=1$ in $R(G)[[t]]$ for $U\in R^+(G)$. 
Thus it is natural to extend 
$\Lambda_t$, 
and therefore also 
the $\Lambda^k$ (cf.\ (\ref{eqn:LambdatSt})), 
from $R^+(G)$ to $R(G)$ by requiring that
\begin{gather}\label{eqn:Lambdat-ext}
	\Lambda_{-t}(U) = \Lambda_{-t}(U^f)\otimes S_t(U^b)
\end{gather}
when $U=U^f-U^b$ for $U^f$ and $U^b$ as in (\ref{eqn:UfUb}). 
Then we have $\Lambda^k(-U) = (-1)^kS^k(U)$ for $U\in R^+(G)$, and (\ref{eqn:LambdatUpUpp}--\ref{eqn:LambdatUexp}) now hold for all $U,U',U''\in R(G)$. 

We now consider the trace of an element $g\in G$ on $\Lambda_{-t}(U)$ for a virtual $G$-module $U=U^f-U^b$. In light of (\ref{eqn:Lambdat-ext}) we may restrict to the cases that $U=U^f$ and $U=-U^b$.
In the first of these cases, whereby $U=U^f$ belongs to $R^+(G)$, we find, by decomposing (a representative of) $U^f$ into $1$-dimensional representations of the cyclic group generated by $g$, that
\begin{gather}\label{eqn:trgLambdaminustUf}
	\tr(g|\Lambda_{-t}(U))
	= \tr(g|\Lambda_{-t}(U^f)) 
	= \prod_i (1-\xi_it)
	= \exp\left(-\sum_{k>0}\sum_i\xi_i^k\frac{t^k}{k}\right),
\end{gather}
where $\{\xi_i\}$ is the multiset of eigenvalues of $g$ on $U^f$.
By similar considerations we find for $U=-U^b$ that 
\begin{gather}\label{eqn:trgStUb}
	\tr(g|\Lambda_{-t}(U)) 
	= \tr(g|S_{t}(U^b)) 
	= \prod_i (1-\xi_it)^{-1}
	= \exp\left(\sum_{k>0}\sum_i\xi_i^k\frac{t^k}{k}\right),
\end{gather}
where now $\{\xi_i\}$ is the multiset of eigenvalues of $g$ on $U^b$. In particular, we have $\tr(g|\psi^k(U)) = \tr(g^k|U)$ for all $U\in R(G)$.

We can write 
$\tr(g|\Lambda_{-t}(U))$ as a product of terms of the form $1-t^n$ in the case that $\tr(g|U)$ is a rational integer, and this will be useful in \S~\ref{sec:res-con}. 
To explain how this works suppose that $\tr(g|U)\in\ZZ$, and for now let us impose the simplifying assumption that $U=U^f\in R^+(G)$ (cf.\ (\ref{eqn:UfUb})). 
Then, since the cyclotomic fields are all Galois extensions of $\QQ$ it must be that 
all the primitive $d$-th roots of unity appear with the same multiplicity in the multiset $\{\xi_i\}$ of eigenvalues of $g$ on $U$, for each $d>0$. 
Write $u_d(g|U)$ for this multiplicity. 
Then 
we have 
\begin{gather}\label{eqn:udg}
	\tr(g|\Lambda_{-t}(U))
	=
	\prod_{i}(1-\xi_i t)
	= 
	\prod_{d>0}c_d(t)^{u_d(g|U)}
\end{gather}
(cf.\ (\ref{eqn:trgLambdaminustUf})), where 
$c_d(t)$ denotes the $d$-th cyclotomic polynomial,
\begin{gather}\label{eqn:vng:2}
c_d(t):=\prod_{\substack{k\xmod d\\\gcd(k,d)=1}}(1-\ex(-\tfrac{k}{d})t)
\end{gather}
(normalized to have constant term equal to $1$). 

We have $1-t^n=\prod_{d|n}c_d(t)$ by construction (\ref{eqn:vng:2}), so 
we obtain 
\begin{gather}\label{eqn:vng:3}
	c_d(t)=\prod_{b|d}(1-t^b)^{\mu(\frac{d}{b})}
\end{gather}
by M\"obius inversion, where $\mu(n)$ is the M\"obius function. 
Now substituting 
(\ref{eqn:vng:3}) 
into
(\ref{eqn:udg}) 
gives us
\begin{gather}\label{eqn:vng:4}
	\tr(g|\Lambda_{-t}(U))
	= 
	\prod_{d>0}\prod_{b|d}(1-t^b)^{\mu(\frac{d}{b})u_d(g|U)},
\end{gather}
and from this 
we see that there exist integers $v_b(g|U)$ such that 
\begin{gather}\label{eqn:vng}
	\tr(g|\Lambda_{-t}(U))
	= 
	\prod_{b>0}(1-t^b)^{v_b(g|U)},
\end{gather}
as we claimed.
Indeed, rewriting (\ref{eqn:vng:4}) as a product over positive integers $b$, and writing the $d$ in each factor as $d=ab$ for some $a>0$,  we obtain
\begin{gather}\label{eqn:vngsumamuauang}
	v_b(g|U) = \sum_{a>0}\mu(a)u_{ab}(g|U).
\end{gather}

The general case, whereby $U\in R(G)$ is such that $\tr(g|U)\in\ZZ$, is just the same, except that some of the $u_d(g|U)$ may be negative. 

We summarize the above discussion as follows.
\begin{lem}\label{lem:trgLambdaminustU}
Let $g\in G$ and $U\in R(G)$ such that $\tr(g|U)\in \ZZ$. 
Then we have 
\begin{gather}\label{eqn:trgLambdaminustU}
\tr(g|\Lambda_{-t}(U)) 
=
\prod_{b>0} (1-t^b)^{v_b(g|U)}
=
\exp\left(-\sum_{k>0}\tr(g^k|U)\frac{t^k}{k}\right)
\end{gather}
where the $v_b(g|U)$ are as in (\ref{eqn:vng}--\ref{eqn:vngsumamuauang}).
\end{lem}

To conclude this section we mention that the Frame shape of $g$ on $U$, 
for $g\in G$ and $U\in R(G)$ such that $\tr(g|U)\in \ZZ$,
is
the formal product
\begin{gather}\label{eqn:pig}
	\pi(g|U):=\prod_{b>0}b^{v_b(g|U)}
\end{gather}
where the $v_b(g|U)$ are as in (\ref{eqn:vng}--\ref{eqn:vngsumamuauang}).

\section{Results}\label{sec:res}

In this section we establish our construction $\SQ$---a lift of the singular theta lift---which uses Borcherds products to produce 
$G$-modules of weight zero 
from $G$-modules of weight one-half.
We formulate this precisely in \S~\ref{sec:res-for}. 
In particular, 
we explain
what we mean by a $G$-module of weight zero or weight one-half, and we specify the particular infinite products that underpin the construction.
The main result of this paper, Theorem \ref{thm:res}, also appears in \S~\ref{sec:res-for}. It confirms that the construction we propose makes sense.

The proof of Theorem \ref{thm:res} requires preparation, which we carry out in \S\S~\ref{sec:res-rpk}--\ref{sec:res-con}. In particular,
the verification that the product formulae we propose define modular forms depends upon a method we call repackaging, which we detail in
\S~\ref{sec:res-rpk}.
It also depends upon  
the singular theta lift of 
\cite{Harvey:1995fq,Borcherds:1996uda}.
We translate 
this result from the formulation in \cite{Borcherds:1996uda} to the setup of \S\S~\ref{sec:res-for}--\ref{sec:res-rpk} in \S~\ref{sec:res-lif}, 
and use it to 
prove that the product formulae of
\S~\ref{sec:res-for} converge and are modular in \S~\ref{sec:res-CM}. 
Then we explain 
how $\SQ$ works at the level of virtual $G$-modules, and use this to complete the proof of Theorem \ref{thm:res},
in
\S~\ref{sec:res-con}.

\subsection{Formulation}\label{sec:res-for}

To formulate our construction $\SQ$ precisely we let $G$ be an arbitrary finite group, 
and suppose that $W$ is a virtual graded $G$-module (see (\ref{eqn:virtualgraded})) with a grading of the form 
\begin{gather}\label{eqn:res-for:W}
	W=\bigoplus_{r\xmod 2m}\bigoplus_{D\equiv r^2\xmod 4m}W_{r,\frac{D}{4m}}
\end{gather}
for some positive integer $m$. 
Define the
McKay--Thompson series 
associated to $W$
to be the vector-valued functions $F^W_g=(F^W_{g,r})$, for $g\in G$, that are obtained by setting
\begin{gather}\label{eqn:res-for:FWgr}
	F^W_{g,r}(\tau):=
	\sum_{D\equiv r^2\xmod 4m} \tr(g|W_{r,\frac{D}{4m}})q^{\frac{D}{4m}}.
\end{gather}
Then we say that $W$ as in (\ref{eqn:res-for:W}) is a {weakly holomorphic (virtual graded) $G$-module of weight $\frac12$ and index $m$} 
if for each $g\in G$ there exists a positive multiple $N_g$ of $o(g)$ such that $\frac{N_g}{o(g)}$ divides $o(g)$ and 
\begin{gather}\label{eqn:res-for:FWgnVwh}
F^W_{g^n}\in V^\wh_{\frac12,m}\left(\tfrac{N_g}{n}\right)
\end{gather}
(see \S~\ref{sec:pre-mod}) for each positive divisor $n$ of $N_g$.
Further, we say that such a $G$-module $W$ is rational if the coefficients 
$C^W_{g}(D,r)$
of the McKay--Thompson series $F^W_g$ of (\ref{eqn:res-for:FWgr}) are rational integers, 
\begin{gather}\label{eqn:res-for:CWgDr}
C^W_{g}(D,r):=\tr(g|W_{r,\frac{D}{4m}})\in\ZZ
\end{gather}
for all $g\in G$ and $D,r\in \ZZ$.

\begin{rmk}\label{rmk:res-for:Ngexp}
We expect $F^W_g$ as in (\ref{eqn:res-for:FWgr}) to transform in a natural way with respect to the action of $\widetilde{\Gamma}_0(o(g))$ for each $g\in G$. 
We allow $N_g$ in (\ref{eqn:res-for:FWgnVwh}) to be a multiple of $o(g)$ in order to allow for the possibility that this transformation is governed by a non-trivial character. 
We assume that any such character has order at most $o(g)$, which is why we require $N_g$ to be a multiple of $o(g)$ that divides $o(g)^2$. 
Thus, in particular, we require $F^W_e$ to belong to $V^\wh_{\frac12,m}$. 
\end{rmk}

\begin{rmk}\label{rmk:res-for:jacmod}
As we have mentioned in \S~\ref{sec:pre-mod}, 
elements of $\textsl{V}^\wh_{\frac12,m}(N)$ 
are in natural correspondence with weakly skew-holomorphic Jacobi forms of weight $1$ and index $m$ for $\Gamma_0(N)$ (see (\ref{eqn:pre-mod:varrphi})), 
and elements of $\textsl{V}^\wh_{\frac12,-m}(N)$ 
are in natural correspondence with weakly holomorphic Jacobi forms of weight $1$ and index $m$ for $\Gamma_0(N)$ (see (\ref{eqn:pre-mod:phi})), 
when $m$ is a positive integer.
Thus we may equivalently refer to a weakly holomorphic (virtual graded) $G$-module $W$ of weight $\frac12$ and index $m$ as in (\ref{eqn:res-for:W}--\ref{eqn:res-for:FWgnVwh}) as weakly skew-holomorphic Jacobi of weight $1$ and index $m$, when $m$ is positive. Also, there is no obstruction to formulating the notion of weakly holomorphic (virtual graded) $G$-module $W$ of weight $\frac12$ and index $-m$, for $m$ positive, just by replacing $m$ with $-m$ in (\ref{eqn:res-for:W}--\ref{eqn:res-for:FWgr}). If working with Jacobi forms we should call such $W$ weakly holomorphic Jacobi of weight $1$ and index $m$.
More generally, we obtain the notion of a weakly holomorphic mock modular (virtual graded) $G$-module of weight $\frac12$ and (non-zero) index $m$, and an equivalent notion in terms of mock Jacobi forms, by replacing $\textsl{V}$ with $\mathbb{V}$ in (\ref{eqn:res-for:FWgnVwh}) (cf.\ \S~\ref{sec:pre-mod}).
\end{rmk}

To describe the results of our construction 
we consider 
a virtual
$\ZZ$-graded $G$-module 
\begin{gather}\label{eqn:res-for:V}
	V=\bigoplus_{n\in\ZZ}V_{n}
\end{gather}
(cf.\ (\ref{eqn:virtualgraded})),
and say that such a virtual graded $G$-module is weakly holomorphic of weight $0$ if there exists a constant $\vh$ such that 
the 
graded trace function 
\begin{gather}\label{eqn:res-for:fVg}
f^V_g(\tau):=
	\sum_{n}
	\tr\left(g|V_{n}\right)q^{n-\vh}
\end{gather} 
is a weakly holomorphic modular form of weight $0$ (with level depending upon $g$), for each $g\in G$.

Note that the $\vh$ in (\ref{eqn:res-for:fVg}) that makes $f^V_g$ modular is unique if it exists (so there is no need to include it in the notation for $V$).
If $V$ is a weakly holomorphic $G$-module of weight $0$ as in (\ref{eqn:res-for:V}--\ref{eqn:res-for:fVg}) we call the associated trace functions $f^V_g$ the McKay--Thompson series of $V$. 

The cornerstone of our construction is the 
twined $O_{2,1}$ Borcherds product 
\begin{gather}\label{eqn:res-for:PsiWg}
\Psi^W_g(\tau) := q^{-H}\exp\left(-\sum_{n>0}\sum_{k>0} C^W_{g^k}(n^2,n)\frac{q^{nk}}{k}   \right),
\end{gather}
defined for $g\in G$ with $C^W_g(D,r)$ as in (\ref{eqn:res-for:CWgDr}), where $H=H^W$ is the generalized class number associated to $W$ (defined below in (\ref{eqn:H})).
A technical problem we resolve in \S~\ref{sec:res-CM} is that of showing that the right-hand side of (\ref{eqn:res-for:PsiWg}) converges for $\Im(\tau)$ sufficiently large, and extends by analytic continuation to a holomorphic function on the upper half-plane. 
(It is 
this analytically continued function that we have in mind when we write $\Psi^W_g$.)

A less technical, but nonetheless important ingredient is the eta product 
\begin{gather}\label{eqn:res-for:etaWg}
	\eta^W_g(\tau) := \prod_{b>0}\eta(b\tau)^{2v_b(g|W_{0,0})},
\end{gather}
defined for $g\in G$ where $v_b(g|W_{0,0})$ is as in (\ref{eqn:vng}), for $W_{0,0}$ as in (\ref{eqn:res-for:W}), and
where $\eta(\tau):=q^{\frac1{24}}\prod_{n>0}(1-q^n)$ is the Dedekind eta function.

Taking the convergence of (\ref{eqn:res-for:PsiWg}) on trust for now, we define
\begin{gather}\label{eqn:res-for:TWg}
	T^W_g:=\frac{\Psi^W_g}{\eta^W_g}
\end{gather}
for $g\in G$, when $W$ is a rational weakly holomorphic $G$-module of weight $\frac12$ and some positive integer index $m$. 
The main result of this paper 
is the following.
\begin{thm}\label{thm:res}
Let $G$ be a finite group, 
and let $W$ be a rational weakly holomorphic $G$-module of weight $\frac12$ and index $m$, for some positive integer $m$.
Then there exists a unique weakly holomorphic 
$G$-module $V=V^W$ of weight $0$ such that 
	$
	f^V_g
	= 
	T^W_g 
	$
for all $g\in G$.
\end{thm}

For $W$ a rational weakly holomorphic $G$-module of weight $\frac12$ and some positive integer index
we define
$\SQ(W)$ to be the weakly holomorphic $G$-module $V^W$ of weight $0$ whose existence and uniqueness is asserted by Theorem \ref{thm:res},
\begin{gather}\label{eqn:res-for:SQW}
	\SQ(W) := V^W.
\end{gather}

The proof of Theorem \ref{thm:res} requires some preparation, which we carry out in the rest of \S~\ref{sec:res}. The proof itself 
is completed in \S~\ref{sec:res-con}. 
Although we do not pursue the formal language of functors in this work, we explain how to realize $\SQ(W)$ explicitly in terms of tensor products of alternating and symmetric powers of subspaces of $W$ in \S~\ref{sec:res-con}. 
From this we see that the construction $\SQ$ is functorial at the level of vector spaces, and in particular, maps $G$-modules to $G$-modules.

\begin{rmk}\label{rmk:res-for:nonegm}
According to Theorem \ref{thm:res}, the construction $\SQ$ of (\ref{eqn:res-for:SQW}) applies to any virtual graded $G$-module $W$ that is rational and weakly holomorphic of weight $\frac12$ and index $m$ in the sense of (\ref{eqn:res-for:W}--\ref{eqn:res-for:CWgDr}), for $m$ a positive integer.
It does not apply to weakly holomorphic modular or mock modular $G$-modules of weight $\frac12$ and negative index, 
as formulated in Remark \ref{rmk:res-for:jacmod}, because the representation $\varrho_m$ is of the wrong type when $m$ is negative (but see \S~\ref{sec:out-imp}). 
It also does not apply to weakly holomorphic mock modular $G$-modules of weight $\frac12$ and positive index, in general. 
In principle we could establish such an extension by more closely following the analysis of
\cite{BruinierOno2010},
but according to Remark 8 of op.\ cit.\ it is not expected that there are any rational weakly holomorphic mock modular $G$-modules of weight $\frac12$ and positive index that are not already modular.
\end{rmk}

\subsection{Repackaging}\label{sec:res-rpk}

In this section we show that the McKay--Thompson series $F^W_{g^k}$, whose coefficients appear in (\ref{eqn:res-for:PsiWg}), can be ``repackaged'' into a single modular form $\check{F}^W_g$ which transforms under the Weil representation $\varrho_{m,N}$ attached to the lattice $L_{m,N} = \sqrt{2m}\ZZ\oplus\Gamma^{1,1}(N)$. 
It is this modular form whose singular theta lift will ultimately yield the desired Borcherds product (\ref{eqn:res-for:PsiWg}) that defines $\Psi^W_g$. Actually, we will provide a  more general construction, which is analogous to, and heavily inspired by, Lemma 3.6 of \cite{Carnahan2012}.
\begin{lem}\label{repackaged_lem}
Fix positive integers $m$ and  $N$. The following spaces are isomorphic:
\begin{enumerate}
\item Families $\{F^{(n)}\}_{n\vert N}$ of weakly holomorphic vector-valued modular forms where each $F^{(n)}$ 
is of weight $\frac12$ and type $\varrho_m$ for 
$\widetilde{\Gamma}_1(N/n)$. \label{itm:Fbn}
\item Families 
$\{\hat{F}_{(i,j)}\}_{i,j\in\ZZ/N\ZZ}$ 
of $\CC[L_m^*/L_m]$-valued functions 
on $\HH$ 
that obey the transformation property  \label{itm:hatF}
$\hat{F}_{(i,j)}(\gamma\tau) = u(\tau)\varrho_m({\gamma},u)\hat{F}_{(i,j)\gamma}(\tau)$ for $(\gamma,u)\in\widetilde{\SL}_2(\ZZ)$. 
\item Weakly holomorphic vector-valued modular forms $\check{F}$ 
of weight $\frac12$ and type $\varrho_{m,N}$ for $\widetilde{\SL}_2(\ZZ)$.\label{itm:checkF}
\end{enumerate}
\end{lem}
In Item \ref{itm:hatF} of the statement of Lemma \ref{repackaged_lem} we write $(i,j)\gamma$ for the usual right action of matrices on row vectors, but reduce the computation modulo $N$ since $(i,j)$ belongs to $\ZZ/N\ZZ\times\ZZ/N\ZZ$.

\begin{proof}[Proof of Lemma \ref{repackaged_lem}]
For the course of this proof it will be convenient to use the slash operators $|_m$ and $|_{m,N}$ defined by setting
\begin{gather}
	f|_m(\gamma,u)(\tau) := u(\tau)^{-1}\varrho_m(\gamma,u)^{-1}f(\gamma\tau),
	\label{eqn:fslashm}\\
	\check{F}|_{m,N}(\gamma,u)(\tau):=u(\tau)^{-1}\varrho_{m,N}(\gamma,u)^{-1}\check{F}(\gamma\tau),
	\label{eqn:FslashmN}
\end{gather}
for $(\gamma,u)\in \widetilde{\SL}_2(\ZZ)$, when $f=(f_r)$ is a vector-valued function on $\HH$ with components indexed by $\ZZ/2m\ZZ$, and $\check{F}=(\check{F}_{i,j,r})$ is a vector-valued function on $\HH$ with components indexed by $\ZZ/N\ZZ\times\ZZ/N\ZZ\times\ZZ/2m\ZZ$, as in Item \ref{itm:checkF} above. 
Also, it will be convenient to use $\check{F}_{(i,j)}=(\check{F}_{(i,j),r})$ to denote the vector-valued function with components 
\begin{gather}\label{eqn:FijrFijr}
\check{F}_{(i,j),r}:=\check{F}_{i,j,r}
\end{gather} 
indexed by $\ZZ/2m\ZZ$, when $\check{F}=(\check{F}_{i,j,r})$ is as in (\ref{eqn:FslashmN}). As an application of these conventions (\ref{eqn:fslashm}--\ref{eqn:FijrFijr}) we note that
\begin{gather}
	\begin{split}\label{eqn:FmNijFijm}
	\left.\Big(\check{F}_{}\,\right\vert_{m,N}\widetilde{T} \Big)_{(i,j)}
	&=\left.\ex\left(-\frac{ij}{N}\right)
	\Big(\check{F}_{(i,j)}\right|_m\widetilde{T}\Big),\\
	\left.\Big(\check{F}\right|_{m,N}\widetilde{S}\Big)_{(i,j)} 
	&=\frac1N\sum_{i',j'\xmod N}
	\left.\ex\left(\frac{ij'+ji'}{N}\right)
	\Big(\check{F}_{(i',j')}\right|_m\widetilde{S}\Big),
	\end{split}
\end{gather}
for $i,j\in \ZZ/N\ZZ$ (cf.\ (\ref{eqn:pre-lat:weiLmN})), when $\check{F}=(\check{F}_{i,j,r})$ is as in (\ref{eqn:FslashmN}).

{(1) $\Leftrightarrow$ (2).} 
To start we observe that
the orbits of $\SL_2(\ZZ)$ on $\ZZ/N\ZZ\times\ZZ/N\ZZ$, acting on the right, are in one-to-one correspondence with the divisors of $N$, the bijection being given by $n \mapsto (0,n) \SL_2(\ZZ)$. 
More specifically, $(i,j)$ belongs to the orbit of $(0,n)$ where $n = \gcd(i,j,N)$. To see this, note that there is a matrix of the form
\begin{gather}\label{A_matrix}
\overline{\gamma_0} = \left(\begin{matrix} \ast & \ast \\ i/n & j/n\end{matrix}\right) \in \SL_2(\ZZ/n'\ZZ),
\end{gather}
where $n'=N/n$,
and that any matrix $\gamma_0\in\SL_2(\ZZ)$ which maps to this $\overline{\gamma_0}$ in (\ref{A_matrix}) under the natural map $\SL_2(\ZZ)\to\SL_2(\ZZ/n'\ZZ)$ will relate $(0,n)$ to $(i,j)$ via its action. The ambiguity in this choice is given precisely by left translations by $\Gamma_1(n')$.

We now attach functions $\hat{F}_{(i,j)}=(\hat{F}_{(i,j),r})$ for $(i,j)$ in $\ZZ/N\ZZ\times \ZZ/N\ZZ$ to the family $\{F^{(n)}\}_{n|N}$ by defining
\begin{gather}\label{hatF definition}
\hat{F}_{(i,j)} := 
\Big(\left.F^{(\gcd(i,j,N))}\right|_m(\gamma_0,u_0)\Big)
\end{gather}
(cf.\ (\ref{eqn:fslashm})), where for each $(i,j)\in\ZZ/N\ZZ\times\ZZ/N\ZZ$ we take $(\gamma_0,u_0)$ to be 
an arbitrary preimage in $\widetilde{\SL}_2(\ZZ)$ of a matrix $\overline{\gamma_0}$ as in (\ref{A_matrix}). 

We require to show that the functions $\hat{F}_{(i,j)}$ satisfy the transformation property 
\begin{gather}\label{hatF_transformation}
\Big(\left.\hat{F}_{(i,j)}\right|_m(\gamma,u)\Big) = \hat{F}_{(i,j)\gamma}
\end{gather}
for $(\gamma,u)\in \widetilde{\SL}_2(\ZZ)$. 
So let $(i,j)\in\ZZ/N\ZZ\times\ZZ/N\ZZ$, 
let $(\gamma,u)\in\widetilde{\SL}_2(\ZZ)$,
and 
set $n=\gcd(i,j,N)$.
Also let $(\gamma_0,u_0)\in \widetilde{\SL}_2(\ZZ)$ be as in (\ref{hatF definition}). Then by (\ref{hatF definition}) we have
\begin{gather}
	\Big(\left.\hat{F}_{(i,j)}\right|_m(\gamma,u)\Big) 
	= \Big(\left.F^{(n)}\right|_m(\gamma_0',u_0')\Big)
\end{gather}
where $(\gamma_0',u_0')=(\gamma_0,u_0)(\gamma,u)$. Next let $(\gamma_0'',u_0'')\in\widetilde{\SL}_2(\ZZ)$ be as in (\ref{hatF definition}), but with $(i,j)$ replaced by $(i,j)\gamma$. 
Then the bottom rows of $\gamma_0'$ and $\gamma_0''$ agree modulo $N/n$, so there is some $(\gamma',u')\in \widetilde{\Gamma}_1(N/n)$ such that $(\gamma_0'',u_0'')=(\gamma',u')(\gamma_0',u_0')$. 
We have that $F^{(n)}$ is invariant under the action of $\widetilde{\Gamma}_1(N/n)$ defined by the slash operator $|_m$ of (\ref{eqn:fslashm}) by hypothesis, so
\begin{gather}
\begin{split}
	\hat{F}_{(i,j)\gamma} 
	= \Big(\left.F^{(n)}\right|_m(\gamma_0'',u_0'')\Big) 
	= \Big(\left.F^{(n)}\right|_m(\gamma_0',u_0')\Big) 
	= \Big(\left.\hat{F}_{(i,j)}\right|_m(\gamma,u)\Big)
\end{split}
\end{gather}
as required.

The inverse of the construction we have just validated is
\begin{gather}
\left\{\hat{F}_{(i,j)}\right\}
\mapsto \left\{F^{(n)}=\hat{F}_{(0,n)}   \right \}_{n\vert N},
\end{gather}
so the spaces of Items \ref{itm:Fbn} and \ref{itm:hatF} are isomorphic. 

{(2) $\Leftrightarrow$ (3).} 
The functions $\check{F}=(\check{F}_{i,j,r})$ we are after are obtained by taking the discrete Fourier transforms of the $\hat{F}_{(i,j)}=(\hat{F}_{(i,j),r})$ in their second index. 
Precisely, we take
\begin{gather}\label{dfn_hatF}
\check{F}_{i,j,r} := \frac{1}{N}\sum_{j'\xmod N }\ex\left(-\frac{jj'}{N}\right)\hat{F}_{(i,j'),r}
\end{gather}
for $i,j\in\ZZ/N\ZZ$ and $r\in \ZZ/2m\ZZ$, and we note that the inverse transform is
\begin{gather}\label{eqn:inversetransform}
\hat{F}_{(i,j),r} = \sum_{j'\xmod N }\ex\left(\frac{jj'}{N}\right)\check{F}_{i,j',r}.
\end{gather}

We require to show that $(\check{F}|_{m,N}(\gamma,u))=\check{F}$ for $(\gamma,u)\in \widetilde{\SL}_2(\ZZ)$. 
It suffices to check that this is true on the generators $\widetilde{S}$ and $\widetilde{T}$ of $\widetilde{\SL}_2(\ZZ)$ (see (\ref{eqn:tildeStildeT})). 
Employing
the notation (\ref{eqn:fslashm}--\ref{eqn:FijrFijr}), the
identities (\ref{eqn:FmNijFijm}) and (\ref{hatF_transformation}),
and the transforms (\ref{dfn_hatF}--\ref{eqn:inversetransform}),
we have
\begin{gather}
\begin{split}\label{eqn:checkFslashmNT}
	\left.\left(\check{F}\right|_{m,N}\widetilde{T}\right)_{(i,j)} 
	&=\ex\left(-\frac{ij}{N}\right)
	\left.\Big(\check{F}_{(i,j)}\right|_m\widetilde{T}\Big)\\
	&=\ex\left(-\frac{ij}{N}\right)\frac1N\sum_{j'\xmod N}\ex\left(-\frac{jj'}N\right)
	\left.\Big(\hat{F}_{(i,j')}\right|_m\widetilde{T}\Big)\\
	&=\ex\left(-\frac{ij}{N}\right)\frac1N\sum_{j'\xmod N}\ex\left(-\frac{jj'}N\right)
	\hat{F}_{(i,i+j')},
\end{split}\\
\begin{split}\label{eqn:checkFslashmNS}
	\left.\left(\check{F}\right|_{m,N}\widetilde{S}\right)_{(i,j)} 
	&=\frac1N\sum_{i',j'\xmod N}
	\ex\left(\frac{ii'+jj'}{N}\right)
	\left.\Big(\check{F}_{(j',i')}\right|_m\widetilde{S}\Big)\\
	&=\frac1N\sum_{j'\xmod N}
	\ex\left(\frac{jj'}{N}\right)
	\left.\Big(\hat{F}_{(j',i)}\right|_m\widetilde{S}\Big)\\
	&=\frac1N\sum_{j'\xmod N}
	\ex\left(\frac{jj'}{N}\right)\hat{F}_{(i,-j')}.
\end{split}
\end{gather}
Now replacing $j'$ with $j'-i$ in the last line of (\ref{eqn:checkFslashmNT}), and replacing $j'$ with $-j'$ in the last line of (\ref{eqn:checkFslashmNS}), we find that
\begin{gather}
	\left.\left(\check{F}\right|_{m,N}\widetilde{T}\right)_{(i,j)} 
=
	\left.\left(\check{F}\right|_{m,N}\widetilde{S}\right)_{(i,j)} 
=
\frac1N\sum_{j'\xmod N}\ex\left(-\frac{jj'}N\right)\hat{F}_{(i,j')}
=\check{F}_{(i,j)}
\end{gather}
for $i,j\in\ZZ/N\ZZ$. 
Thus $(\check{F}|_{m,N}\widetilde{T})=(\check{F}|_{m,N}\widetilde{S})=\check{F}$ as we claimed.

The statement that the spaces of Items \ref{itm:hatF} and \ref{itm:checkF} are in bijection now follows from standard facts about discrete Fourier transforms. This concludes the proof.
\end{proof}

Composing the transformations detailed in the proof of Lemma \ref{repackaged_lem} we obtain an explicit isomorphism
\begin{gather}\label{eqn:FbntocheckF}
\left\{F^{(n)}\right\}_{n\vert N} \mapsto \check{F}=\left(\check{F}_{i,j,r}\right)
\end{gather} 
from collections of weakly holomorphic modular forms of weight $\frac12$ and type $\varrho_m$ for groups $\widetilde{\Gamma}_1(N/n)$, to weakly holomorphic modular forms of weight $\frac12$ and type $\varrho_{m,N}$ for $\widetilde{\SL}_2(\ZZ)$. 
Therefore, given a weakly holomorphic $G$-module $W$ of weight $\frac12$ and some index as in (\ref{eqn:res-for:W}), 
we can apply this map (\ref{eqn:FbntocheckF}) to
$\{F^{(n)}=F^W_{g^n}\}_{n|N}$
(cf.\ (\ref{eqn:res-for:FWgr})),
and thereby
encode the McKay--Thompson series $F^W_{g^n}$ associated with powers of a given group element $g$ in a modular form $\check{F}^W_g$ transforming under $\varrho_{m,N}$, where $N=N_g$ is 
as in (\ref{eqn:res-for:FWgnVwh}).
In this situation we write
\begin{gather}\label{eqn:FWgntocheckFWg}
\left\{F^W_{g^n}\right\}_{n\vert N_g} \mapsto \check{F}^W_g=\left(\check{F}^W_{g,i,j,r}\right)
\end{gather} 
for the map (\ref{eqn:FbntocheckF}).
Extending the notational convention (\ref{eqn:FijrFijr}), given $i,j\in \ZZ/N\ZZ$ we write $\check{F}^W_{g,(i,j)}$ for the vector-valued function that takes the $\check{F}^W_{g,i,j,r}$ for $r\xmod 2m$ as its components, so that
\begin{gather}\label{eqn:checkFWgij}
	\check{F}^W_{g,(i,j)} 
	= \left( \check{F}^W_{g,(i,j),r}\right)
	= \left( \check{F}^W_{g,i,j,r}\right).
\end{gather}

\begin{eg}\label{eg:rpk:1}
Here we give an example of a repackaging of a family of modular forms with $N=2$ that is of relevance to the analysis of \cite{pmt}.
To begin we define a weakly holomorphic vector-valued modular form 
$F^{(1)}=(F^{(1)}_0,F^{(1)}_1)$ 
for the restriction to $\widetilde{\Gamma}_0(2)$ of the Weil representation $\varrho_1=\varrho_{1,1}$ (cf.\ (\ref{eqn:pre-lat:weiLmN})) by setting
\begin{gather}\label{2aclass}
\begin{split}
F^{(1)}_0(\tau) &= 128\frac{\eta(\tau)^6\eta(4\tau)^{14}}{\eta(2\tau)^{19}}
=128q-768q^2+3584q^3 -13312q^4+43008q^5+\dots ,\\
F^{(1)}_1(\tau) &= \frac{\eta(2\tau)^{23}}{\eta(\tau)^8\eta(4\tau)^{14}}
=q^{-\frac{3}{4}}+8q^{\frac{1}{4}}+21q^{\frac{5}{4}}+8q^{\frac{9}{4}}-42q^{\frac{13}{4}}+155q^{\frac{21}4}+\dots,
\end{split}
\end{gather}
where $\eta$ is as in (\ref{eqn:res-for:etaWg}).
We also let
$F^{(2)}=(F^{(2)}_0,F^{(2)}_1)$ 
be the unique weakly holomorphic modular form of weight $\frac12$ for the 
representation $\varrho_1$ for $\widetilde{\Gamma}_0(1)=\widetilde{\SL}_2(\ZZ)$ that satisfies $F^{(2)}_r(\tau)=\delta_{r,1}q^{-\frac34}+O(q^{\frac14})$ as $\Im(\tau)\to\infty$, so that 
\begin{gather}\label{zagierf3}
\begin{split}
F^{(2)}_0(\tau) &= 26752q+1707264q^2+44330496q^3+708938752q^4 + \dots, \\
F^{(2)}_1(\tau) &= q^{-\frac{3}{4}}-248q^{\frac{1}{4}}-85995q^{\frac{5}{4}}-4096248q^{\frac{9}{4}}-91951146q^{\frac{13}{4}}+\dots
\end{split}
\end{gather}
(cf.\ (\ref{eqn:int-pro:f0})). 

We can now construct the function $\hat{F}$ corresponding to $\{F^{(n)}\}_{n|2}$ in Lemma \ref{repackaged_lem} by setting
\begin{equation}\label{eqn:hatFN=2}
\begin{alignedat}{2}
\hat{F}_{(0,0)}(\tau) &:= F^{(2)}(\tau),
&\qquad
\hat{F}_{(0,1)}(\tau) &:= F^{(1)}(\tau),
\\
\hat{F}_{(1,0)}(\tau) &:= 
\tau^{-\frac{1}{2}}\varrho_1(\widetilde{S})^{-1}F^{(1)}(-\tfrac{1}{\tau}), 
&\qquad 
\hat{F}_{(1,1)}(\tau) &:= 
(\tau+1)^{-\frac{1}{2}}\varrho_1(\widetilde{S}\widetilde{T})^{-1}F^{(1)}(-\tfrac{1}{\tau+1}).
\end{alignedat}
\end{equation}
Then, using (\ref{eqn:hatFN=2}) and the formula $\eta(-\frac1\tau)=\ex(-\frac18)\tau^{\frac12}\eta(\tau)$ we find that the image $\check{F}$ of $\{F^{(n)}\}_{n|2}$ under the map (\ref{eqn:FbntocheckF}) satisfies
\begin{equation}\label{eqn:checkFN=2}
\begin{alignedat}{2}
\check{F}_{(0,0)} &=
\frac{1}{2}F^{(2)}+\frac12F^{(1)}, 
&\qquad
\check{F}_{(0,1)} &= 
\frac{1}{2}F^{(2)}-\frac12F^{(1)},
\\
\check{F}_{(1,0)} &=
\frac{1}{2}F^{(2)}-\frac12F^{(1)}+8\theta^0_1, 
&\qquad 
\check{F}_{(1,1)} &=
\hat{F}_{(1,0)}
-\frac{1}{2}F^{(2)}
+\frac12F^{(1)}-8\theta^0_1.
\end{alignedat}
\end{equation}

To make $\check{F}_{(1,1)}$ completely explicit we note that
\begin{gather}
\begin{split}\label{eqn:rpk-eg1:hatF}
\hat F_{(1,0),0}(\tau) &= 
4\frac{\eta(\tau)^6\eta(\frac{\tau}4)^{14}}{\eta(\frac{\tau}2)^{19}}+4\frac{\eta(\frac\tau2)^{23}}{\eta(\tau)^8\eta(\frac{\tau}4)^{14}}
=8+768q^{\frac12}+13328q+125440q^{\frac32}+\dots ,\\
\hat F_{(1,0),1}(\tau) &= 
4\frac{\eta(\tau)^6\eta(\frac{\tau}4)^{14}}{\eta(\frac{\tau}2)^{19}}-4\frac{\eta(\frac\tau2)^{23}}{\eta(\tau)^8\eta(\frac{\tau}4)^{14}}
=-112q^{\frac14}-3584q^{\frac34}-43008q^{\frac54}
+\dots.
\end{split}
\end{gather}
\end{eg}

\begin{eg}\label{eg:rpk:2}
It is instructive to carry out the repackaging procedure in the specific case that the input family $\{F^{(n)}\}_{n|N}$ is composed of theta functions. 
For a basic example of this let $N$ be a prime, and let $c^{(1)}$ and $c^{(N)}$ be integers such that $c^{(1)}\equiv c^{(N)}\xmod N$. 
Then 
for $G=\ZZ/N\ZZ$ there exists a weakly holomorphic $G$-module $W$ of weight $\frac12$ and index $1$, as in (\ref{eqn:res-for:W}), such that
\begin{gather}\label{eqn:res-rpk:Fgtheta10}
	F^W_g:=
	\begin{cases}
	c^{(N)}\theta_1^0 &\text{ if $o(g)=1$,}\\
	c^{(1)}\theta_1^0&\text{ if $o(g)\neq 1$,}
	\end{cases}
\end{gather}
where $\theta_1^0=(\theta_{1,0}^0,\theta_{1,1}^0)$ is as defined in \S~\ref{sec:pre-mod}.

So let us take $\{F^{(n)}=F_{g^n}^W\}_{n|N}$ for $F^W_g$ as in (\ref{eqn:res-rpk:Fgtheta10}). 
Then, since $\theta_1^0$ is invariant for the weight $\frac12$ action of $\widetilde{\SL}_2(\ZZ)$ defined by the Weil representation $\varrho_1$, we have
\begin{gather}
	\hat{F}_{(i,j)} = 
	\begin{cases}
		c^{(N)}\theta_1^0&\text{ if $i\equiv j\equiv 0\xmod N$,}\\
		c^{(1)}\theta_1^0&\text{ else,}
	\end{cases}
\end{gather} 
and it follows from this that $\check{F}^W$, being the image of $\{F^W_{g^n}\}_{n|N}$ under (\ref{eqn:FbntocheckF}),  is given by 
\begin{gather}\label{eqn:checkF2}
	\check{F}^W_{(i,j)} = 
	\begin{cases}
		\frac1N\left(c^{(N)}+(N-1)c^{(1)}\right)\theta_1^0&\text{ if $i\equiv 0\xmod N$ and $j\equiv 0\xmod N$,}\\		
		\frac1N\left(c^{(N)}-c^{(1)}\right)\theta_1^0&\text{ if $i\equiv 0\xmod N$ and $j\not\equiv 0\xmod N$,}\\		
		c^{(1)}\theta_1^0&\text{ if $i\not\equiv 0\xmod N$ and $j\equiv 0\xmod N$,}\\	
		0&\text{ if $i\not\equiv 0\xmod N$ and $j\not\equiv 0\xmod N$.}
	\end{cases}	
\end{gather}
\end{eg}

\begin{eg}\label{eg:rpk:3}
We now put together the previous two examples so as to demonstrate how repackaging works for the input family $\{F^{(-3,1)}_{g^n}\}_{n|2}$ of penumbral moonshine \cite{Harvey:2015mca,pmo}, for $g$ an element of order $N=2$ in the Thompson group $\Th$. 
(We use this example in \cite{pmt}.)

We first inspect Table A.6 of \cite{Griffin2016}, and in so doing arrive at the identity
\begin{gather}\label{eqn:F31gn2Fn}
F^{(-3,1)}_{g^n}= 2F^{(n)}+\tr(g^n|{\bf 248})\theta_1^0,
\end{gather}
for $g\in Th$ with $o(g)=2$, and $n|2$, 
where $F^{(n)}$ is the function so denoted 
in Example \ref{eg:rpk:1},
and ${\bf 248}$ denotes the unique irreducible representation of $\Th$ of dimension $248$.

We have $\tr(g|{\bf 248})=-8$ (see e.g.\ \cite{atlas}), so taking $N=2$, $c^{(2)}=248$ and $c^{(1)}=-8$ in Example \ref{eg:rpk:2}
we find that (\ref{eqn:checkF2}) specializes to 
\begin{equation}\label{eqn:checkF2_spec}
\begin{alignedat}{2}
\check{F}^2_{(0,0)} &=
120\theta_1^0,
&\qquad
\check{F}^2_{(0,1)} &= 
128\theta_1^0,
\\
\check{F}^2_{(1,0)} &=
-8\theta_1^0,
&\qquad 
\check{F}^2_{(1,1)} &=
0,
\end{alignedat}
\end{equation}
where we here write $\check{F}^2$ for the function denoted simply by $\check{F}$ in Example \ref{eg:rpk:2}.
We now put
(\ref{eqn:checkFN=2})
and
(\ref{eqn:checkF2_spec})
together according to (\ref{eqn:F31gn2Fn}), and find that 
\begin{equation}\label{eqn:checkF31N=2}
\begin{alignedat}{2}
\check{F}_{(0,0)} &=
\frac{1}{2}F^{(2)}+\frac12F^{(1)}, 
&\qquad
\check{F}_{(0,1)} &= 
\frac{1}{2}F^{(2)}-\frac12F^{(1)},
\\
\check{F}_{(1,0)} &=
\frac{1}{2}F^{(2)}-\frac12F^{(1)}-120\theta^0_1, 
&\qquad 
\check{F}_{(1,1)} &=
\hat{F}_{(1,0)}
-\frac{1}{2}F^{(2)}
+\frac12F^{(1)}+120\theta^0_1,
\end{alignedat}
\end{equation}
is the result of applying the map (\ref{eqn:FbntocheckF}) to $\{F^{(-3,1)}_{g^n}\}_{n|2}$.

For concreteness we note that $\hat{F}_{(1,0)}=2\hat{F}^1_{(1,0)}-8\theta_1^0$, where $\hat{F}^1_{(1,0)}$ is the function denoted $\hat{F}_{(1,0)}$ in (\ref{eqn:hatFN=2}) (cf.\ also \ref{eqn:rpk-eg1:hatF}).
\end{eg}

The top line of (\ref{eqn:hatFN=2}) 
demonstrates the general fact that $\hat{F}_{(0,n)} = F^{(n)}$ for $n$ a divisor of $N$, which follows directly from the construction (\ref{hatF definition}). 
In the situations of interest to us each $F^{(n)}$ will actually be modular for the restriction of $\varrho_m$ to $\widetilde{\Gamma}_0(N/n)$. In this case we have the following stronger result.
\begin{lem}\label{lem:FbnGamma0}
Suppose that $\{F^{(n)}\}_{n|N}$ is a family of vector-valued functions such that $F^{(n)}$ is a weakly holomorphic modular form of weight $\frac12$ and type $\varrho_{m}$ for $\widetilde{\Gamma}_0(N/n)$. 
Then 
for $\{\hat{F}_{(i,j)}\}$ the corresponding family in Lemma \ref{repackaged_lem} we have $\hat{F}_{(0,j)}=F^{(\gcd(j,N))}$ for all $j\in \ZZ/N\ZZ$.
\end{lem}
\begin{proof}
Let $n=\gcd(j,N)$. Then we have $\hat{F}_{(0,j)}=(F^{(n)}|_m(\gamma_0,u_0))$ for some $(\gamma_0,u_0)\in \widetilde{\SL}_2(\ZZ)$ by construction (\ref{hatF definition}), where the bottom row of $\gamma_0$ is congruent to $(0,j/n)$ modulo $N/n$. So $(\gamma_0,u_0)$ belongs to $\widetilde{\Gamma}_0(N/n)$, and thus $(F^{(n)}|_m(\gamma_0,u_0))=F^{(n)}$ by hypothesis. This proves the claim.
\end{proof}

\subsection{Lifting}\label{sec:res-lif}

We have defined a map (\ref{eqn:FWgntocheckFWg}) which may be used to exchange the 
McKay--Thompson series $\{F^W_g\}_{g\in G}$ of a $G$-module $W$ as in (\ref{eqn:res-for:W}), each $F^W_g$ (cf.\ (\ref{eqn:res-for:FWgr})) being a weakly holomorphic modular form of type $\varrho_m$ with level $N_g$, for an equivalent family $\{\check{F}^W_g\}_{g\in G}$, where each $\check{F}^W_g$ is a weakly holomorphic modular form of type $\varrho_{m,N_g}$ with level $1$. It is this latter collection of functions to which the machinery of \cite{Borcherds:1996uda} is most directly applied. 
In this section we will 
explain how Borcherds' computation (in op.\ cit.) of the singular theta lift 
\begin{gather}\label{eqn:res-lif:stl}
\Phi_{L}(v^+,\check{F}) := \int_{\mathcal{F}}^\reg
\overline{\Theta_{L}(\tau,v^+)}\check{F}(\tau)
\frac{{\rm d}\tau_1{\rm d}\tau_2}{\tau_2}
\end{gather}
of $\check{F}=\check{F}^W_g$ reads, in the $O_{2,1}$ situation of relevance to us.

The 
notation in (\ref{eqn:res-lif:stl}) 
is just as in 
(\ref{eqn:int-pro:int}),
but here we note that complex conjugation acts on the basis elements of $\mathbb{C}[L_m^\ast/L_m]$ as $\overline{e_r} = e_{-r}$, so that the pairing of $\overline{\Theta_{L}(\tau,v^+)}$ with $F(\tau)$ is the one which multiplies the $e_{-r}$ component of the former with the $e_r$ component of the latter, and sums over $r$. 
Also, 
we will be taking $L=L_{m,N}$ for $N=N_g$ (cf.\ (\ref{eqn:res-for:FWgnVwh})), and 
the superscript in the integral sign indicates an application of the regularization explained in \S~6 of \cite{Borcherds:1996uda} (see also \cite{Harvey:1995fq}). 
We require some further preparation in order to formulate the relevant results from \cite{Borcherds:1996uda}. We perform this preparation next.

Given positive integers $m$ and $N$ define traceless $2\times 2$ matrices $\lambda(c,b,a)$, and a set $L$ of such things, by taking
\begin{gather}\label{eqn:pre-lat:LmN2x2}
	\lambda(c,b,a):=\left(\begin{matrix}a & b/m \\ Nc & -a\end{matrix}\right),\quad
	L := \left\{ \lambda(c,b,a)\mid a,b, c \in \ZZ    \right\}.
\end{gather}
Then $L$ becomes a copy of the lattice $L_{m,N}$ of (\ref{eqn:pre-lat:LmN}) once we equip it with the quadratic form $Q_L$ given by 
$Q_L(\lambda) := -m\det(\lambda)$.
In terms of the $\lambda(c,b,a)$,
the quadratic form and associated bilinear form on $L\cong L_{m,N}$ are 
given explicitly by
\begin{gather}
\begin{split}\label{eqn:QLmNexp}
Q_L(\lambda(c,b,a)) & = ma^2+Nbc,\\
(\lambda(c,b,a),\lambda(c',b',a')) &= 2maa'+N(b' c + b c').
\end{split}
\end{gather}
Using this description we find that the dual of $L\cong L_{m,N}$ is
\begin{gather}\label{eqn:Lstar}
L^\ast = 
\left.\left\{\lambda\left(\tfrac{c}{N},\tfrac{b}{N},\tfrac{a}{2m}\right)\,\right|\, a,b,c\in \ZZ\right\}.
\end{gather}

One convenience of this realization (\ref{eqn:pre-lat:LmN2x2}--\ref{eqn:QLmNexp}) of $L_{m,N}$ is that it makes the automorphisms of the lattice easier to see. 
For this we define a left action of $\SL_2(\mathbb{R})$ on $L\otimes \RR$ by isometries by setting
\begin{gather}\label{sl2r_action}
\gamma\cdot\lambda:=\gamma\lambda\gamma^{-1}
\end{gather}
for $\gamma\in \SL_2(\RR)$ and $\lambda\in L\otimes \RR$.
Then the automorphism group $\Aut(L_{m,N})$ is the subgroup of $\SL_2(\mathbb{R})$ composed of the elements that preserve the lattice, when acting as in (\ref{sl2r_action}). 
In general it contains at least $\Gamma_0(mN)$ as a subgroup.

The computation of (\ref{eqn:res-lif:stl}) in \cite{Borcherds:1996uda} (see also \cite{Harvey:1995fq}) produces an automorphic form on a Grassmannian. We now explain what such an object is in the setup of interest to us.
Since $L$ has signature $(2,1)$ we may consider the Grassmannian $G(L)$ of $2$-dimensional positive definite subspaces of $L\otimes \RR$. 
Following op.\ cit.\ we equip this space with a complex structure as follows.
First consider the particular $2$-dimensional positive definite subspace 
\begin{gather}
v^+=\left\{\lambda(\tfrac1Nx,mx,y)\mid x,y\in\RR\right\}\in G(L),
\end{gather} 
and give it an orientation by declaring that the ordered basis $(X_L,Y_L)$ for $v^+$ is positively oriented, for the particular choice 
\begin{gather}\label{eqn:particularXLYL}
	X_L=
	\lambda(\tfrac1N,m,0)
	,\quad 
	Y_L=
	\lambda(0,0,1)
	.
\end{gather}
Then we obtain an orientation on every element of $G(L)$ by requiring that the choice varies continuously, as we move away from $v^+$. 
Next, for an arbitrary oriented $2$-dimensional positive definite subspace $v^+\in G(L)$ we consider the points 
\begin{gather}\label{eqn:ZL}
Z_L=X_L+iY_L\in L\otimes \CC,
\end{gather} 
such that $(X_L,Y_L)$ is a positively oriented ordered orthogonal basis of $v^+$ with $Q_L(X_L)=Q_L(Y_L)$. Each such $Z_L$ has norm zero, and all define the same point $\CC Z_L$ in the complex projective space $\PP(L\otimes \CC)$. So finally we use the association $v^+\mapsto \CC Z_L$ to identify $G(L)$ with an open subset of $\PP(L\otimes \CC)$, and thereby obtain a complex structure on $G(L)$. 

Now to define automorphic forms on $G(L)$ we observe that the norm zero points $Z_L$ of (\ref{eqn:ZL}) form a principal $\CC^*$-bundle $P$ over $G(L)$, so realized as a subset of $\PP(L\otimes \CC)$.  
An automorphic form of weight $k$ on $G(L)$, in the sense of op.\ cit., is a function on $P$ which is homogeneous of degree $-k$ and invariant under some subgroup $\Gamma$ of finite index in $\Aut(L)^+$, the orientation preserving automorphisms of $L$. 

Borcherds provides product formulae for the automorphic forms produced by the singular theta lift; one for each cusp of the invariance group. To formulate this concretely let $\ell$ be a primitive norm zero vector in $L$, and let $\ell'\in L^*$ be such that $(\ell,\ell')=1$. Then for any finite-index subgroup $\Gamma<\Aut(L)^+$ as in the last paragraph, $\ell$ represents a cusp of $\Gamma$. 
Define a corresponding lattice 
$K$ by setting
\begin{gather}\label{eqn:K}
	K:=(L\cap \ell^\perp)/\ZZ\ell.
\end{gather}
Identify $K\otimes \RR$ with the orthogonal complement of $\ell'$ in $L\otimes \RR\cap\ell^\perp$ in the natural way, and in so doing identify $K$ with a subgroup of $L\otimes \RR$.

In our situation $K$ is even, positive definite, and has rank $1$ by construction, so 
$K\cong \sqrt{2\Km}\ZZ$
for some positive integer $\Km$. 
In particular, there are two choices of positive cone in $K\otimes \RR$. 
Following op.\ cit.\ we take $C\subset K\otimes \RR$ to be the one that contains a vector $Y_L$, where $Z_L=X_L+iY_L$ is a point of $P$ as in (\ref{eqn:ZL}), and furthermore $(X_L,\ell)>0$. (We necessarily have $(Y_L,\ell)=0$ for $Y_L\in K\otimes \RR$.) In the remainder of this section we use $\Kb$ to denote the generator of $K$ that lies in $C$, and set $\Kb':=\frac{1}{2\Km}\Kb$, so that
\begin{gather}\label{eqn:KKstarC}
	K=\ZZ \Kb,\quad
	K^*=
	\ZZ\Kb',\quad
	C=\RR^+\Kb=\RR^+\Kb'.
\end{gather}

Now consider the set of points $Z=X+iY\in K\otimes \CC$ with $Y\in C$. Since the rank of $K$ is $1$ this is actually a copy of the upper half-plane $\HH$. We may map it into $P$, and thereby interpret the automorphic form produced by (\ref{eqn:res-lif:stl}) as a modular form, by sending such a $Z=X+iY$ to the unique norm zero vector $Z_L\in P$ as in (\ref{eqn:ZL}) that satisfies $(Z_L,\ell)=1$, and whose projection onto $K$ recovers $Z$. Explicitly, the map $Z\mapsto Z_L$ is given by setting
\begin{gather}\label{eqn:ZtoZL}
	Z_L:=Z-Q_L(Z)\ell-Q_L(\ell')\ell+\ell'.
\end{gather} 

With $Z_L$ as in (\ref{eqn:ZtoZL}) the association $Z\mapsto \CC Z_L$ defines an isomorphism of $K\otimes\RR + iC \simeq \HH$ with $G(L)$.
In particular, automorphic forms on $G(L)$ may naturally be regarded as modular forms on the upper half-plane. 
(Note that the weight picks up a factor of two when the automorphic form produced by (\ref{eqn:res-lif:stl}) is regarded as a modular form. 
Cf.\ Example 14.4 in \cite{Borcherds:1996uda}.)
We write
\begin{gather}\label{eqn:PsiellZcheckF}
	\Psi_\ell(Z,\check{F}) := \Psi_L\left(Z-Q_L(Z)\ell-Q_L(\ell')\ell+\ell',\check{F}\right)
\end{gather}
for the modular form defined by (\ref{eqn:ZtoZL}).

A further key feature of the singular theta lift of op.\ cit.\ is that it gives us direct knowledge of the divisors of the automorphic forms that it produces. 
In our situation the zeros and poles on the upper half-plane 
occur at points $\CC Z_L$, for $Z_L$ as in (\ref{eqn:ZtoZL}) that satisfy 
\begin{gather}\label{eqn:ZLlambda}
(Z_L,\lambda)=0
\end{gather}
for some $\lambda\in L$ with $Q_L(\lambda)<0$. Following Borcherds we denote such a point $\CC Z_L$ by $\lambda^\perp$. 

We compute the zeros and poles at cusps by computing the corresponding Weyl vectors,
$\rho(K,W,\check{F}_K)$, 
as defined after the proof of Theorem 10.3 in op.\ cit. 
Here $K$ is as in (\ref{eqn:K}), for a given cusp representative $\ell$, and the symbol $W$ denotes a Weyl chamber, as defined after the proof of Theorem 6.2 in op.\ cit. 
In our situation the Weyl chamber is simply a choice of positive cone in $K\otimes \RR$, so we may and do take $W=C$.
The function $\check{F}_K=(\check{F}_{K,r})$ 
is the weakly holomorphic modular form of weight $\frac12$ and type $\varrho_K=\varrho_\Km$ (cf.\ (\ref{eqn:KKstarC})) for $\widetilde{\SL}_2(\ZZ)$ that we obtain 
by setting
\begin{gather}\label{eqn:checkFKr}
	\check{F}_{K,r}:=
	\sum_{\substack{i,j\xmod N,\,s\xmod 2m\\(i,j,s)|(L\cap \ell^\perp) = r}}
	\check{F}_{i,j,s}
\end{gather}
for $r \xmod 2\Km$,
where the summation in (\ref{eqn:checkFKr}) is over the 
$(i,j,s)$, for $i,j\xmod N$ and $s\xmod 2m$, such that the vectors in the coset of $L$ represented by 
$\lambda(\frac{i}{N},\frac{j}{N},\frac{s}{2m})$ 
define the same subset of $\hom(L\cap\ell^\perp,\ZZ)$ as those in the coset of $K$ 
represented by 
$r\Kb'$
(cf.\ (\ref{eqn:KKstarC})).
(This is a specialization of the definition preceding the proof of Theorem 5.3 in \cite{Borcherds:1996uda}.)

With the above understanding 
the Weyl vector $\rho(K,C,\check{F}_K)$ may now be defined by requiring that
\begin{gather}\label{eqn:rhoKWcheckFK}
	(\rho(K,C,\check{F}_K),\Kb) = 
	\frac{\sqrt{Q_L(\Kb)}}{8\pi}
	\int^\reg_\mathcal{F}
	\overline{\theta_\Km^0(\tau)}\check{F}_K(\tau)
	{\tau_2^{-\frac32}}
	{{\rm d}\tau_1{\rm d}\tau_2}
\end{gather}
(cf.\ (4.8) in \cite{BruinierOno2010}), where 
$\theta_\Km^0=(\theta_{\Km,r}^0)$
is the 
theta function associated to $K$ (cf.\ (\ref{eqn:pre-mod:thetanull})), 
and the notation in (\ref{eqn:rhoKWcheckFK}) is otherwise as in (\ref{eqn:res-lif:stl}). 
We henceforth write $\rho(K,\check{F}_K)$ for $\rho(K,C,\check{F}_K)$, since there is a natural choice $W=C$ of Weyl chamber in the situation at hand.

Note that the 
integral in (\ref{eqn:rhoKWcheckFK}) 
is 
computed in Corollary 9.6 in \cite{Borcherds:1996uda}, with the result being that 
\begin{gather}\label{eqn:intThetaLcheckFK}
	\int_{\mathcal{F}}^\reg\overline{\theta_\Km^0(\tau)}\check{F}_K(\tau)\tau_2^{-\frac32}{{\rm d}\tau_1{\rm d}\tau_2}
	=-\frac{8\pi}{\sqrt{\Km}}\sum_{r\xmod 2\Km}\sum_{\substack{D\equiv r^2\xmod 4\Km\\D\leq 0}} 
	\check{C}_K(D,r)
	H_{\Km}(D,r)
	,
\end{gather}
where $\check{C}_K(D,r)$ is the coefficient of $q^{\frac{D}{4\Km}}$ in the Fourier expansion of $\check{F}_{K,r}$,
and 
$H_{\Km}(D,r)$ is the coefficient of $q^{-\frac{D}{4\Km}}$ in the Fourier expansion of $G_{\Km,r}$, 
for $G_\Km=(G_{\Km,r})$ 
as defined in loc.\ cit. 
We have $Q_L(\Kb)=\Km$ according to (\ref{eqn:KKstarC}), so from (\ref{eqn:rhoKWcheckFK}--\ref{eqn:intThetaLcheckFK}) we obtain that
\begin{gather}\label{eqn:rhoKcheckFK}
\rho(K,\check{F}_K):=
\rho(K,C,\check{F}_K) = 
-H(K,\check{F}_K)\Kb',
\end{gather} 
where
\begin{gather}\label{eqn:HKCcheckFK}
	H(K,\check{F}_K)
	:=
	\sum_{r\xmod 2\Km}\sum_{\substack{D\equiv r^2\xmod 4\Km\\D\leq 0}} \check{C}_K(D,r)H_{\Km}(D,r).
\end{gather}

Note that the $H_\Km(D,r)$ are generalized class numbers. For example, when $\Km=1$ and $D$ is a negative discriminant $H_{1}(D,D)$ is the Hurwitz class number of $D$ (commonly denoted $H(D)$, cf.\ e.g.\ \cite{ZagierTSM}).
Taking $\breve{G}_1(\tau):=G_{1,0}(4\tau)+G_{1,1}(4\tau)$ we have
\begin{gather}\label{eqn:breveG1}
	\breve{G}_1(\tau)=\sum_{D\leq 0}H_1(D,D)q^D=-\frac1{12}+\frac13q^3+\frac12q^4+q^7+q^8+\dots,
\end{gather}
and $H_1(D,D)$ is the usual class number of the quadratic imaginary number field $\QQ(\sqrt{D})$ when $D<-4$ is fundamental. In light of this we call $H(K,\check{F}_K)$ the {\em generalized class number} associated to $\check{F}$ and the cusp $\ell$, and also use the notation
\begin{gather}\label{eqn:Hellg}
	H^W_\ell(g) := H(K,\check{F}_K),
\end{gather}
when $K=(L\cap \ell^\perp)/\ZZ\ell$, in the specific situation that $\check{F}=\check{F}^W_g$ is the image under the map (\ref{eqn:FbntocheckF}) of a family 
$\{F^W_{g^n}\}_{n|N}$, for $g\in G$ and $N=N_g$, for some $W$ as in 
the statement of Theorem \ref{thm:res}.

We are now ready to state the result of \cite{Borcherds:1996uda} that will allow us to validate and interpret the singular theta lift (\ref{eqn:res-lif:stl}). 
To ease notation in the statement we define 
\begin{gather}\label{eqn:kg}
	k^W_g:=\frac{1}{N}\sum_{n|N}\phi\left(\frac{N}{n}\right)C_{g^n}^W(0,0),
\end{gather}
where $\phi$ denotes the Euler totient function, and $C^W_g(D,r)$ 
is as in (\ref{eqn:res-for:CWgDr}).

\begin{thm}[Borcherds]
\label{thm:STL}
Let $G$ and $W$ be as in the statement of Theorem \ref{thm:res}, choose $g\in G$, and let $\check{F}=\check{F}^W_g$ be obtained as the image under the map (\ref{eqn:FWgntocheckFWg}) of $\{F^W_{g^n}\}_{n|N}$ where $N=N_g$ and $F^W_g=(F^W_{g,r})$ is as in (\ref{eqn:res-for:FWgr}). 
Also, let $L$ be as in (\ref{eqn:pre-lat:LmN2x2}--\ref{eqn:QLmNexp}) and let $P$ be as above.
Then there is a meromorphic function $\Psi_L(Z_L,\check{F})$ for $Z_L\in P$ with the following properties.
\begin{enumerate}
\item $\Psi_L(Z_L,\check{F})$ is an automorphic form of weight \label{itm:borthm_1}
$\frac12k^W_g$
for the group $\Aut(L,\check{F})$, with respect to some unitary character of finite order. 
\item The zeros and poles of $\Psi_L(Z_L,\check{F})$ in $\HH$ occur at points $\lambda^\perp$ for $\lambda\in L$, and when $\lambda=\lambda(c,b,a)$ the contribution of $\lambda^\perp$ to the divisor of $\Psi_L(Z_L,\check{F})$ is 
\label{itm:borthm_2}
\begin{gather}\label{eqn:STL:divcon}
	\sum_{\substack{x\in\RR^+\\Nxc,Nxb,2mxa\in\ZZ}}
	\check{C}_{xc,xb}(x^2(ma^2+Nbc),xa).
\end{gather}
\item The function $\Psi_L$ is related to the singular theta lift $\Phi_L$ of (\ref{eqn:res-lif:stl}) by \label{itm:borthm_3}
\begin{gather}\label{eqn:STL:logPsi}
\log|\Psi_L(Z_L,\check{F})| = -\frac{\Phi_L(Z_L,\check{F})}{4} - \frac{k^W_g}2\left(\log|Y_L| + \frac{\Gamma'(1)}2 + \log\sqrt{2\pi}     \right).
\end{gather}
\item 
For each primitive norm zero vector $\ell\in L$  \label{itm:borthm_4}
the infinite product
\begin{gather}\label{eqn:STL:infpro}
\ex(-(H(K,\check{F}_K)\Kb',Z))
\prod_{n>0}\;
\prod_{\substack{
i,j\xmod N, r\xmod 2m \\
(i,j,r)
\vert (L\cap\ell^\perp) 
=n}}
\left(1-\ex((n\Kb',Z)+(\lambda,\ell'))\right)^{\check{C}_{i,j}(n^2,r)},
\end{gather}
where in each factor $\lambda=\lambda(\frac{i}{N},\frac{j}{N},\frac{r}{2m})$,
converges for $Z$ in some neighborhood of the cusp of $\Aut(L,\check{F})$ represented by $\ell$, and is proportional to $\Psi_\ell(Z,\check{F})$ (cf.\ (\ref{eqn:PsiellZcheckF})) in the domain of its convergence.
\end{enumerate}
\end{thm}
In (\ref{eqn:STL:divcon}) and (\ref{eqn:STL:infpro}) 
we write $\check{C}_{i,j}(D,r)$ for the coefficient of $q^{\frac{D}{4m}}$ in the Fourier expansion of 
$
\check{F}_{(i,j),r}
=
\check{F}_{i,j,r}
$, 
and we interpret 
the condition $(i,j,r)|(L\cap\ell^\perp)=n$ as we did for (\ref{eqn:checkFKr}).
In (\ref{eqn:STL:logPsi}) we write $\Gamma'(1)$ for the value at $s=1$ of the derivative of the Gamma function.
In the proof of Theorem \ref{thm:STL} we will write $\hat{C}_{(i,j)}(D,r)$ for the coefficient of $q^{\frac{D}{4m}}$ in the Fourier expansion of $\hat{F}_{(i,j),r}$.

\begin{proof}[Proof of Theorem \ref{thm:STL}.]
The claimed result is a specialization of Theorem 13.3 of \cite{Borcherds:1996uda} to the situation of interest in this work.
More specifically, Items \ref{itm:borthm_1}--\ref{itm:borthm_3} above correspond to Items 1--3 in loc.\ cit., and Item \ref{itm:borthm_4} above corresponds to Item 5 in loc.\ cit. 
(Item 4 of loc.\ cit.\ does not apply to our choice of $L$.)
The only part that requires some translation is the computation of the weight $\frac12k^W_g$, which is $\frac12\check{C}_{0,0}(0,0)$ according to the statement in loc.\ cit.
To perform this translation recall that $\check{C}_{0,0}(0,0)=\frac1N\sum_{j'\xmod N}\hat{C}_{(0,j')}(0,0)$ according to (\ref{dfn_hatF}), and $\hat{C}_{(0,j')}(0,0)=C^W_{g^n}(0,0)$, where $n=\gcd(j',N)$, according to Lemma  
\ref{lem:FbnGamma0} and our hypothesis that $F^{(n)}=F_{g^n}^W$. 
So 
\begin{gather}
	\check{C}_{0,0}(0,0)=\frac1N\sum_{j'\xmod N}C^W_{g^n}(0,0),
\end{gather}
where in each summand $n=\gcd(j',N)$.
The coincidence of $\check{C}_{0,0}(0,0)$ with $k^W_g$ as defined in (\ref{eqn:kg}) now follows.
\end{proof}

\subsection{Convergence}\label{sec:res-CM}

The purpose of this section is to verify that the functions $T^W_g$, 
defined by (\ref{eqn:res-for:PsiWg}--\ref{eqn:res-for:TWg}), 
are weakly holomorphic modular forms of weight $0$. (Cf.\ Theorem \ref{thm:res}.)

The main challenge at hand is
to check the convergence and modularity of the product formula (\ref{eqn:res-for:PsiWg}). 
We use Theorem \ref{thm:STL} 
to do this. More specifically, we apply 
(\ref{eqn:STL:infpro})
in the special case that $\ell$ represents the infinite cusp, so we begin this section with a closer look at the constructions of (\ref{eqn:K}--\ref{eqn:Hellg}) for such $\ell$. 
We then carry out the application of Theorem \ref{thm:STL} in Proposition \ref{pro:conmod}. 
The modularity of $\eta^W_g$ follows directly from the definition (\ref{eqn:res-for:etaWg}), but it is not immediate that it has the same weight as $\Psi^W_g$. 
For this we require to understand how the $v_b(g|U)$ of (\ref{eqn:vng}) transform under power maps $g\mapsto g^p$. 
We present the necessary analysis in Lemmas \ref{lem:vngp} and \ref{lem:weight}.

For the rest of this section we take $\ell=\lambda(0,1,0)$, whereby $\ell$ represents the infinite cusp of $\Aut(L)^+$ (and any finite-index subgroup thereof). 
Then we may take $\ell'=\lambda(\frac{1}{N},0,0)$ (cf.\ (\ref{eqn:QLmNexp})), in which case $Q_L(\ell')=0$. 
With these choices 
\begin{gather}\label{eqn:Keg}
	L\cap\ell^\perp = \{\lambda(0,b,a)\mid a,b\in\ZZ\},
	\quad
	K=\{\lambda(0,0,a)\mid a\in\ZZ\}
\end{gather}
(cf.\ (\ref{eqn:K})), and in particular $\Km=m$ in (\ref{eqn:KKstarC}).

Now taking $X_L$ and $Y_L$ as in (\ref{eqn:particularXLYL}) we see that $Y_L\in K\otimes \RR$ and $(X_L,\ell)=1>0$, so the positive cone $C$ is composed of the positive real multiples of $\lambda(0,0,1)$. That is, $\Kb=\lambda(0,0,1)$ and $\Kb'=\lambda(0,0,\frac1{2m})$, and $C=\{\lambda(0,0,y)\mid y>0\}$, in the notation of (\ref{eqn:KKstarC}).

Identify the upper half-plane $\HH$ with $K\otimes \RR+iC$ by mapping $\tau\in \HH$ to 
\begin{gather}\label{eqn:Ztau}
Z(\tau):=\lambda(0,0,\tau)=\tau\Kb.
\end{gather} 
Then, writing $Z_L(\tau)$ for the point $Z_L$ corresponding to $Z=Z(\tau)$ as in (\ref{eqn:ZtoZL}), we have $Z_L(\tau) = Z(\tau)-Q_L(Z(\tau))\ell+\ell'$ since $Q_L(\ell')=0$, so that
\begin{gather}\label{eqn:ZLtau}
	Z_L(\tau) = 
	\lambda(\tfrac{1}{N},-m\tau^2,\tau)
	=
	\left(
	\begin{matrix}
		\tau&-\tau^2\\1&-\tau
	\end{matrix}
	\right)
\end{gather}
(cf.\ (\ref{eqn:pre-lat:LmN2x2})).

It is a pleasant exercise to check now that the action (\ref{sl2r_action}) of $\SL_2(\RR)$ induces the usual action by M\"obius transformations on $G(L)\simeq \HH$. Concretely, we have
\begin{gather}
	\CC(\gamma\cdot Z_L(\tau)) = \CC Z_L(\gamma\tau)
\end{gather}
for $\gamma\in \SL_2(\RR)$ and $\tau\in \HH$. We set 
\begin{gather}\label{eqn:PsielltaucheckF}
\Psi_\ell(\tau,\check{F}):=\Psi_\ell(Z(\tau),\check{F})=\Psi_L(Z_L(\tau),\check{F})
\end{gather} 
(cf.\ (\ref{eqn:PsiellZcheckF})) when $\ell=\lambda(0,1,0)$. 

To explain the computation of the points $\lambda^\perp$ (cf.\ (\ref{eqn:ZLlambda})) suppose that $\lambda=\lambda(c,b,a)\in L$ satisfies $Q_L(\lambda)=ma^2+Nbc<0$. Then we have 
\begin{gather}
	\begin{split}\label{eqn:ZLtaulambda}
	(Z_L(\tau),\lambda) &= (\lambda(\tfrac1N,-m\tau^2,\tau),\lambda(c,b,a))\\
	&=2ma\tau - mNc\tau^2 + b,
	\end{split}
\end{gather}
so $(Z_L(\tau),\lambda)=0$ just when $\tau$ is a solution of the quadratic polynomial $mNcX^2-2maX-b$. This polynomial has discriminant $4m(ma^2+Nbc)$, which is negative by our choice of $\lambda$. So there is indeed a solution in $\HH$, and it is unique. This is the point we call $\lambda^\perp$.

To compute $\check{F}_K=(\check{F}_{K,r})$ 
explicitly in terms of $\check{F}=(\check{F}_{i,j,r})$
(cf.\ (\ref{eqn:checkFKr})) 
let $\lambda\in L^*$ and let $\delta\in K^*$ belong to the coset defined by ${r}\Kb'=\lambda(0,0,\frac{r}{2m})$. Then $\lambda=\lambda(\frac{c}{N},\frac{b}{N},\frac{a}{2m})$ for some $a,b,c\in \ZZ$, and $\delta=n\Kb'=\lambda(0,0,\frac{n}{2m})$ for some $n\equiv r\xmod 2m$. Letting $\lambda(0,b',a')$ be an arbitrary element of $L\cap\ell^\perp$ (cf.\ (\ref{eqn:Keg})) we have
\begin{gather}\label{eqn:lambdagammaeg}
	(\lambda,\lambda(0,b',a'))
	=aa'+b'c,
	\quad
	(\delta,\lambda(0',b',a'))
	=na',
\end{gather}
according to (\ref{eqn:QLmNexp}), so $\lambda$ and $\delta$ coincide as homomorphisms on $L\cap\ell^\perp$ just when $a=n$ and $c=0$.
We conclude from this that $\check{F}_{K,r}$ is the sum over the components of $\check{F}$ indexed by $\lambda(0,\frac{j}{N},\frac{r}{2m})$ for arbitrary $j\xmod N$. That is, the definition (\ref{eqn:checkFKr}) works out to 
\begin{gather}\label{eqn:checkFKeg}
	\check{F}_{K,r} = \sum_{j\xmod N} \check{F}_{0,j,r},
\end{gather}
or more simply $\check{F}_K=\sum_{j\xmod N}\check{F}_{0,j}$
in this case.

Observe now using (\ref{eqn:inversetransform}) that 
$\check{F}_K=\hat{F}_{(0,0)}=F^{(N)}=F^W_e$ under our hypotheses on $\check{F}$. Thus we have
\begin{gather}\label{eqn:Hellgeg}
	H^W_\ell(g)=
	H(K,\check{F}_K)
	=
	\sum_{r\xmod 2m}\sum_{\substack{D\equiv r^2\xmod 4m\\D\leq 0}} 
	{C}^W(D,r)H_{m}(D,r)
\end{gather}
(cf.\ (\ref{eqn:HKCcheckFK}--\ref{eqn:Hellg})) for the generalized class number associated to $\ell$ and $\check{F}=\check{F}^W_g$, when $\ell$ is the infinite cusp, where
$C^W(D,r)$ denotes the coefficient of $q^{\frac{D}{4m}}$ in $F^W_{e,r}$, for $F^W_e=(F^W_{e,r})$.
In particular, the generalized class number $H^W_\ell(g)$ is independent of $g$ when $\ell$ represents the infinite cusp. For this reason we drop $g$ from notation in $H^W_\ell(g)$ when $\ell=\lambda(0,1,0)$. Also dropping $\ell$ we obtain the definition of $H=H^W$ in (\ref{eqn:res-for:PsiWg}) that we promised in \S~\ref{sec:res-for}. Specifically, the definition is
\begin{gather}\label{eqn:H}
	H^W
	:=
	\sum_{r\xmod 2m}\sum_{\substack{D\equiv r^2\xmod 4m\\D\leq 0}}
	C^W(D,r)H_{m}(D,r),
\end{gather}
where $C^W(D,r)$ and $H_m(D,r)$ are as in (\ref{eqn:Hellgeg}).

With the definition (\ref{eqn:H}) of $H^W$ in place 
we are ready to 
confirm the validity of the Borcherds product construction (\ref{eqn:res-for:PsiWg}).

\begin{pro}\label{pro:conmod}
Suppose that 
$G$ and $W$ are as in the statement of Theorem \ref{thm:res}, and set $H=H^W$.
Then for each $g\in G$ the expression 
\begin{gather}\label{eqn:conmodpro}
q^{-H}\exp\left(-\sum_{n>0}\sum_{k>0} C^W_{g^k}(n^2,n)\frac{q^{nk}}{k}   \right)
\end{gather}
converges for $\Im(\tau)$ sufficiently large, and extends by analytic continuation to a modular form $\Psi^W_g$ of weight $k^W_g$ (see (\ref{eqn:kg}))
for 
$\Gamma_0(mN_g)$.
\end{pro}

\begin{proof}
We will prove the claimed result by showing that the infinite product expression (\ref{eqn:conmodpro}) coincides with the product (\ref{eqn:STL:infpro}) that we get from Theorem \ref{thm:STL} when $\ell=\lambda(0,1,0)$.

To begin we take $Z=Z(\tau)=\tau\Kb$ (cf.\ (\ref{eqn:Ztau})) in (\ref{eqn:STL:infpro}), 
interpret the second product in (\ref{eqn:STL:infpro}) as in (\ref{eqn:lambdagammaeg}--\ref{eqn:checkFKeg}), 
apply the computation (\ref{eqn:Hellgeg}--\ref{eqn:H}) of $H^W=H^W_\ell(g)$, and recall the choice $\ell'=\lambda(\frac1N,0,0)$ 
so as to obtain the expression
\begin{gather}\label{eqn:conmod_1}
	q^{-H}
\prod_{n>0}\;
\prod_{j\xmod N}
\left(1-\ex(\tfrac{j}{N})q^n\right)^{\check{C}_{0,j}(n^2,n)},
\end{gather}
for (\ref{eqn:STL:infpro}). 
According to Item \ref{itm:borthm_4} of Theorem \ref{thm:STL} this product (\ref{eqn:conmod_1}) converges for $\Im(\tau)$ sufficiently large, and defines a modular form of weight $k^W_g$ on $\HH$ where $k^W_g$ is as in (\ref{eqn:kg}). 

Next we remove the factor $q^{-H}$ from (\ref{eqn:conmod_1}) for a moment, take a logarithm, and expand so as to arrive at
\begin{gather}\label{eqn:cor-2}
\log\prod_{n>0} \prod_{j\xmod N }\Big(1-\ex(\tfrac{j}{N})q^n\Big)^{\check{C}_{0,j}({n^2,n})}
=-\sum_{n>0}\sum_{k>0}\sum_{j\xmod N}\check{C}_{0,j}\left({n^2},n\right)\frac{(\ex(\frac{j}{N})q^n)^k}{k}
\end{gather}
for $\Im(\tau)$ sufficiently large. 
Then we apply the inverse discrete Fourier transform (\ref{eqn:inversetransform}) to the sum over $j$ in the right-hand side of (\ref{eqn:cor-2}), 
and thus obtain the simpler expression
\begin{gather}\label{eqn:cor-3}
\log\prod_{n>0} \prod_{j\xmod N }\Big(1-\ex(\tfrac{j}{N})q^n\Big)^{\check{C}_{0,j}({n^2,n})}
=-\sum_{n>0}\sum_{k>0}\hat{C}_{(0,k)}\left({n^2,n}\right)\frac{q^{nk}}{k},
\end{gather}
where $\hat{C}_{(i,j)}(D,r)$ denotes the coefficient of $q^{\frac{D}{4m}}$ in the Fourier expansion of $\hat{F}_{(i,j),r}$, and the relationship between the $\hat{F}_{(i,j)}=(\hat{F}_{(i,j),r})$ and $\check{F}=(\check{F}_{i,j,r})$ is as in 
Lemma \ref{repackaged_lem}.

The next step is to exponentiate, and replace the factor $q^{-H}$, and thus find that (\ref{eqn:conmod_1}) coincides with
\begin{gather}
q^{-H}\exp\left( -\sum_{n>0} \sum_{k>0}\hat{C}_{(0,k)}({n^2,n})\frac{q^{nk}}{k}     \right).
\end{gather}
Finally we apply Lemma \ref{lem:FbnGamma0}, which tells us that $\hat{F}^W_{(0,k)} = F^W_{g^{k}}$ when $F^{(n)}=F^W_{g^n}$. 
Thus we have $\hat{C}_{(0,k)}({n^2,n}) = C^W_{g^k}(n^2,n)$, and the coincidence between (\ref{eqn:conmodpro}) and (\ref{eqn:conmod_1}) follows. This completes the proof.
\end{proof}

Now we know that the $T^W_g$ are weakly holomorphic modular forms, we can move on to computing their weight. 
For this we need two lemmas. 
The first tells us how the $v_b(g|U)$ of (\ref{eqn:vng}) change under power maps $g\mapsto g^p$. The second uses the result of the first to identify the weight $k^W_g$ (\ref{eqn:kg}) of $\Psi^W_g$ with the weight of $\eta^W_g$, and thus confirms that the $T^W_g$ all have weight $0$.

Note that if 
$\tr(g|U)$ is a rational integer then $\tr(g^n|U)$ is a rational integer too, for every $n\in \ZZ$,
because the eigenvalues defined by the action of $g^n$ on $U$ are just the $n$-th powers $\xi_i^n$ of the eigenvalues $\xi_i$ defined by $g$.
Thus the question of how the $v_b(g|U)$ of (\ref{eqn:vng}) change under power maps 
makes sense.
\begin{lem}\label{lem:vngp}
Suppose that $G$ is a finite group and $U$ is a finite-dimensional $G$-module, and suppose that $g\in G$ is such that $\tr(g|U)$ is a rational integer. 
Then for $p$ a prime and $v_b(g)=v_b(g|U)$ as in (\ref{eqn:vng}) we have
\begin{gather}\label{eqn:vngp}
	v_b(g^p) = 
	\begin{cases}
		pv_{bp}(g) &\text{ if $p$ divides $b$,}\\
		pv_{bp}(g) + v_b(g) &\text{ if $p$ does not divide $b$.}
	\end{cases}
\end{gather}
\end{lem}
\begin{proof}
To begin we consider the behavior of the multiplicities $u_d(g)=u_d(g|U)$ of (\ref{eqn:udg}) under power maps $g\mapsto g^p$ for $p$ prime.
Thus we let $p$ be a prime, and note first that if $\gcd(d,p)=1$ then the $p$-th power of a primitive $d$-th root of unity is again a primitive $d$-th root of unity. 
So $u_d(g^p)=u_d(g)$ if $d$ is not divisible by $p$. 
Next we observe that if $d$ is divisible by $p$ then the $p$-th power of a primitive $d$-th root of unity is a primitive $\frac{d}{p}$-th root of unity, but the multiplicity of the map $\xi\mapsto \xi^p$ in this case depends upon whether $d$ is divisible by $p^2$ or not. If so, then every primitive $\frac{d}{p}$-th root has $p$ primitive $d$-th root preimages under $\xi\mapsto \xi^p$, while if $p$ exactly divides $d$ then there are just $p-1$ primitive preimages. Replacing $d$ with $dp$ in the previous sentence we obtain that the $p$-th powers of the primitive $dp$-th roots of unity constitute $p$ copies 
of each primitive $d$-th root of unity 
if $p$ divides $d$, and constitute $p-1$ copies 
of each primitive $d$-th root of unity 
if $p$ does not divide $d$.
So we have
\begin{gather}\label{eqn:udgp}
	u_d(g^p)
	=
	\begin{cases}
	pu_{dp}(g)&\text{ if $p$ divides $d$,}\\
	(p-1)u_{dp}(g) + u_d(g)&\text{ if $p$ does not divide $d$.}
	\end{cases}
\end{gather}

To finish we combine (\ref{eqn:vngsumamuauang}) and (\ref{eqn:udgp}) in order to calculate $v_b(g^p)$. Taking $b$ to be a multiple of $p$ we find that
\begin{gather}
\begin{split}\label{eqn:vngp-1}
	v_b(g^p) 
	&= \sum_{a>0}\mu(a)\u_{ab}(g^p)\\
	&= \sum_{a>0}\mu(a)pu_{abp}(g)\\
	&=pv_{bp}(g)
\end{split}
\end{gather}
(cf.\ (\ref{eqn:vngsumamuauang})),
in agreement with the first part of (\ref{eqn:vngp}).
Taking $b$ now to be coprime to $p$ we compute that
\begin{gather}
\begin{split}\label{eqn:vngp-2}
	v_b(g^p) 
	&= \sum_{a>0}\mu(a)\u_{ab}(g^p)\\
	&= 
	\sum_{\substack{a>0\\ \gcd(a,p)=1}}\mu(a)((p-1)u_{abp}(g)+u_{ab}(g))
	+
	\sum_{\substack{a>0\\ p|a}}\mu(a)pu_{abp}(g)
	\\
	&=
	S_1+S_0,
\end{split}
\end{gather}
where 
$S_1=p\sum_{a>0}\mu(a)u_{abp}(g)$,
and 
\begin{gather}
\begin{split}
	S_0&=
	\sum_{\substack{a>0\\\gcd(a,p)=1}}\mu(a)u_{ab}(g) 
	-\sum_{\substack{a>0\\\gcd(a,p)=1}}\mu(a)u_{abp}(g) \\
	&=
	\sum_{\substack{a>0\\\gcd(a,p)=1}}\mu(a)u_{ab}(g) 
	+\sum_{\substack{a>0\\\gcd(a,p)=1}}\mu(ap)u_{abp}(g) \\
	&=
	\sum_{\substack{a>0\\\gcd(a,p)=1}}\mu(a)u_{ab}(g) 
	+\sum_{a>0}\mu(ap)u_{abp}(g) \\
	&=
	\sum_{a>0}\mu(a)u_{ab}(g).
\end{split}
\end{gather}
Inspecting (\ref{eqn:vngsumamuauang}) we obtain that 
$S_1=pv_{bp}(g)$ and $S_0=v_b(g)$,
so we have verified the second part of (\ref{eqn:vngp}), and the proof of the lemma is complete.
\end{proof}

By repeated application of Lemma \ref{lem:vngp} we can compute $v_b(g^n)$ in terms of $v_b(g)$ for any $n\in \ZZ$. For example, it follows from an inductive argument using Lemma \ref{lem:vngp} that if $g$ is as in the lemma and 
$p$ is a prime, then
\begin{gather}\label{eqn:vngpk}
	v_b(g^{p^k}) = 
	\begin{cases}
		p^kv_{bp^k}(g) &\text{ if $p$ divides $b$,}\\
		\sum_{j=0}^k p^jv_{bp^j}(g) &\text{ if $p$ does not divide $b$,}
	\end{cases}
\end{gather}
for all $k\geq 0$.

\begin{lem}\label{lem:weight}
Suppose that $G$ is a finite group and $U$ is a finite-dimensional $G$-module, and suppose that $g\in G$ is such that 
$\tr(g|U)$ is a rational integer.
Then for 
$N$ an arbitrary positive integer 
and for 
$v_b(g)=v_b(g|U)$ as in (\ref{eqn:vng}--\ref{eqn:vngsumamuauang}) 
we have
\begin{gather}\label{eqn:weight}
	\frac1N\sum_{n|N}\phi\left(\frac{N}{n}\right)\tr(g^n|U)
	= 
	\sum_{n|N}v_n(g).
\end{gather}
\end{lem}

\begin{proof}
To begin we 
note that $\tr(g^n|U)=v_1(g^n)$ for all $n\in\ZZ$, under our hypothesis that $\tr(g|U)\in\ZZ$.
This is because we have 
\begin{gather}
\begin{split}\label{eqn:weight-0a}
	\tr(g|\Lambda_{-t}(U))
	= 
	\prod_i(1-\xi_i t)
	&
	 = 
	1-\left(\sum_{i}\xi_i\right)t+O(t^2)
	\\
	&=
	1-\tr(g|U)t+O(t^2) 
\end{split}
\end{gather}
for the left-hand side of the defining identity (\ref{eqn:vng}) (cf.\ (\ref{eqn:udg})), where the $\xi_i$ are the eigenvalues defined by the action of $g$ on $U$, while the right-hand side satisfies 
\begin{gather}
\begin{split}\label{eqn:weight-0b}
	\prod_{b>0}(1-t^b)^{v_b(g)}
	=
	1-v_1(g)t+O(t^2). 
\end{split}
\end{gather}
The same argument applies with any power of $g$ in place of $g$, 
so we require to show that 
\begin{gather}\label{eqn:weight-1}
	\sum_{n|N}\phi\left(\frac{N}{n}\right)v_1(g^n)
	= 
	\sum_{n|N}Nv_n(g)
\end{gather}
for arbitrary $N>0$. To obtain (\ref{eqn:weight-1}) we will actually prove the more general statement that
\begin{gather}\label{eqn:weight-2}
	\sum_{n|N}\phi\left(\frac{N}{n}\right)v_m(g^n)
	= 
	\sum_{n|N}Nv_{mn}(g)
\end{gather}
for $m,N>0$, when $\gcd(m,N)=1$.

To establish (\ref{eqn:weight-2}) we begin with the case that $N$ is a prime power. 
Using (\ref{eqn:vngpk}) and the fact that $\sum_{n|N}\phi(n)=N$ we find that for $N=p^k$ with $p$ a prime such that $\gcd(m,p)=1$, the left-hand side of (\ref{eqn:weight-2}) works out to be
\begin{gather}
\begin{split}\label{eqn:weight-3}
	\sum_{j=0}^k\phi\left(p^{k-j}\right)v_m(g^{p^j})
	&=
	\sum_{j=0}^k\phi(p^{k-j})\sum_{i=0}^jp^iv_{mp^i}(g)
	\\
	&=
	\sum_{i=0}^kp^i\sum_{j=i}^k\phi(p^{k-j})v_{mp^i}(g)
	\\
	&=
	\sum_{i=0}^kp^i\sum_{j=0}^{k-i}\phi(p^{j})v_{mp^i}(g)
	\\
	&=\sum_{i=0}^kp^kv_{mp^i}(g),
\end{split}
\end{gather}
which agrees with the right-hand side of (\ref{eqn:weight-2}).

Now suppose that we have proven (\ref{eqn:weight-2}) for $N=N'$ and for $N=N''$, where $m$, $N'$ and $N''$ are pairwise coprime. Then for $N=N'N''$ we have
\begin{gather}
\begin{split}\label{eqn:weight-4}
	\sum_{n|N}\phi\left(\frac{N}{n}\right)v_m(g^n)
	&=
	\sum_{n'|N'}\sum_{n''|N''}\phi\left(\frac{N'N''}{n'n''}\right)v_m((g^{n'})^{n''})
	\\
	&=
	\sum_{n'|N'}\phi\left(\frac{N'}{n'}\right)\sum_{n''|N''}\phi\left(\frac{N''}{n''}\right)v_m((g^{n'})^{n''})
	\\
	&=
	\sum_{n'|N'}\phi\left(\frac{N'}{n'}\right)\sum_{n''|N''}N''v_{mn''}(g^{n'})
	\\
	&=
	\sum_{n''|N''}N''
	\sum_{n'|N'}
	\phi\left(\frac{N'}{n'}\right)
	v_{mn''}(g^{n'})
	\\
	&=
	\sum_{n''|N''}N''
	\sum_{n'|N'}
	N'
	v_{mn'n''}(g),
\end{split}
\end{gather}
which is $\sum_{n|N}Nv_{mn}(g)$.
The identity (\ref{eqn:weight-2}) now follows for arbitrary $N>0$ with $\gcd(m,N)=1$ by 
induction 
on the number of prime divisors of $N$.
Specializing to $m=1$ we have thus proven (\ref{eqn:weight-1}), as required.
\end{proof}

\subsection{Construction}\label{sec:res-con}

In this section we explain how to realize the infinite products that define $T^W_g$ (cf.\ (\ref{eqn:res-for:PsiWg}--\ref{eqn:res-for:TWg})) 
in terms of alternating and symmetric powers of subspaces of $W$ (cf.\ (\ref{eqn:res-for:W})), and finally complete the proof of Theorem \ref{thm:res}.

Taking $G$ and $W$ to be as in the statement of Theorem \ref{thm:res}, 
the task at hand is to define a virtual graded $G$-module $V=V^W$ as in (\ref{eqn:res-for:V}) with the property that $f^V_g=T^W_g$ for each $g\in G$, where $f^V_g$ is as in (\ref{eqn:res-for:fVg}) and $T^W_g$ is as in (\ref{eqn:res-for:TWg}).
To ease notation 
as we do this
let us define 
\begin{gather}\label{eqn:Un}
U_n:=
\begin{cases}
W_{n,\frac{n^2}{4m}}&\text{ for $n>0$,}\\
-2W_{0,0}&\text{ for $n=0$,}
\end{cases}
\end{gather}
and for each $n\geq 0$ write
$U_n=U_n^f-U_n^b,$ where $U_n^f$ and $U_n^b$ are as in (\ref{eqn:UfUb}) (with $U_n$ in place of $U$). 
Then by applying Lemma \ref{lem:trgLambdaminustU} to the definitions (\ref{eqn:res-for:PsiWg}--\ref{eqn:res-for:TWg}) we obtain
\begin{gather}
\begin{split}\label{eqn:TWg-trace}
	T^W_g(\tau)
	&=q^{-\vh}\prod_{n>0} \tr(g|\Lambda_{-q^n}(U_n))\tr(g|\Lambda_{-q^n}( U_0))
\\
	&=q^{-\vh}\prod_{n>0}\tr(g|\Lambda_{-q^n}(U_n^f)\otimes S_{q^n}(U_n^b)\otimes \Lambda_{-q^n}(U_0^f)\otimes S_{q^n}(U_0^b)),
\end{split}
\end{gather}
where 
\begin{gather}\label{eqn:res-con:c}
\vh=H+\frac1{24}(\dim(U_0^b)-\dim(U_0^f)).
\end{gather}

Next we make the convention that $U(-n)$ denotes a copy of $U$, for $U\in R(G)$ and $n$ a positive integer, 
and define operators $F$ and $\deg$ on such $U(-n)$ 
by requiring that 
\begin{gather}
	F 
	=\label{eqn:res-con:FL0}
	\begin{cases}
		\Id\text{ on $U^f(-n)$, }\\
		0 \text{ on $U^b(-n)$,}
	\end{cases}\quad
	{\deg} 
	= n\Id \text{ on $U(-n)$.}
\end{gather}
Then, we extend these operators to alternating, symmetric and tensor products of such spaces $U(-n)$ by applying the Lie-like coproduct, $X\mapsto X\otimes 1+1\otimes X$, so that ${\deg}$ acts as $(n'+n'')\Id$ on $U'(-n')\otimes U'(-n'')$, \&c.

With $F$ and $\deg$ so defined the right-hand side of (\ref{eqn:TWg-trace}) may be rewritten as
\begin{gather}\label{eqn:TWg-trace2}
	q^{-\vh}\prod_{n>0} \tr(g(-1)^Fq^{\deg}|\Lambda(U_n^f(-n))\otimes S(U_n^b(-n))\otimes \Lambda(U_0^f(-n))\otimes S(U_0^b(-n)))
\end{gather}
where $\vh$ is as in (\ref{eqn:res-con:c}).
Thus, if we 
define $\mc{H}(-n)$ for $n>0$ by setting
\begin{gather}\label{eqn:mathcalFn}
	\mc{H}(-n):=\Lambda(U_n^f(-n))\otimes S(U_n^b(-n))\otimes \Lambda(U_0^f(-n))\otimes S(U_0^b(-n)),
\end{gather}
and take 
$\mc{H}:=\bigotimes_{n> 0}\mc{H}(-n)$, 
then 
$\deg$ defines a $\ZZ$-grading $\mc{H}=\bigoplus_{n\geq 0} \mc{H}_n$ on $\mc{H}$, and 
$F$ defines a superspace structure $\mc{H}_n=\mc{H}_n^0\oplus \mc{H}_n^1$ 
on each $\mc{H}_n$ via the requirement that
$(-1)^F=(-1)^j$ on $\mc{H}_n^j$.

We now define $V^W=\bigoplus_{n\geq 0} V^W_n$ to be the virtual graded $G$-module that is determined by requiring that
\begin{gather}\label{eqn:res-con:VW}
	m_\chi(V^W_n) = m_\chi(\mc{H}^0_n)-m_\chi(\mc{H}^1_n)
\end{gather}
(cf.\ (\ref{eqn:mchiU})) for each $\chi\in \Irr(G)$ and $n\geq 0$.
Then taking $V=V^W$ we have that
\begin{gather}\label{eqn:fVg}
f^V_g(\tau)=\sum_{n\geq 0 } \tr(g|V_n)q^{n-\vh}=
\tr\left(\left.g(-1)^Fq^{\deg-\vh}\right|\mc{H}\right)
\end{gather} 
coincides with (\ref{eqn:TWg-trace2}) for $\vh$ as in (\ref{eqn:res-con:c}). That is, $f^V_g=T^W_g$ for all $g\in G$, 
and we have realized the $T^W_g$ as graded traces on a virtual graded $G$-module constructed (\ref{eqn:mathcalFn}--\ref{eqn:res-con:VW}) from tensor products of alternating and symmetric powers of subspaces of $W$, as we promised we would.
In particular, since we have shown in \S~\ref{sec:res-CM} that the $T^W_g$ are weakly holomorphic modular forms of weight $0$, 
and since a virtual graded $G$-module is determined by its graded traces, we have completed the proof of Theorem \ref{thm:res}.

\begin{rmk}\label{rmk:res-con:noVOA}
Each alternating algebra factor $\Lambda(U^f_n(-n))$ in (\ref{eqn:mathcalFn}) may be realized as a Fock space built out of the repeated application of fermionic creation operators, and similarly for the $\Lambda(U^f_0(-n))$, while each symmetric algebra factor $S(U_n^b(-n))$ may be realized as a Fock space built out of bosonic creation operators, and similarly for the $S(U^b_0(-n))$. 
Thinking along these lines, it is tempting to interpret $\deg$ in (\ref{eqn:res-con:FL0}--\ref{eqn:TWg-trace2}) and (\ref{eqn:fVg}) as the grading operator $L_0$ defined by an action of the Virasoro algebra with central charge $c=24\vh$ on $V$, for $\vh$ as in (\ref{eqn:res-con:c}).
Indeed, such an action exists in the case that $W=W_{{\rm 3C}}$, in the notation of \cite{pmt}, as $V^W$ is isomorphic to a tensor power of the 3C-twisted module $\vn_{\rm 3C}$ for the moonshine module VOA for this choice of $W$, according to Theorem 4.5 of op.\ cit. 
(See also Example \ref{eg:out-inv} below.) However, the construction we have given does not (obviously) entail any Virasoro action, or VOA or CFT structure on $V^W$ in general.
\end{rmk}

\section{Outlook}\label{sec:out}

We provide further perspective on the results of this paper in this section. 
To do this we first explain an approach for applying traces of singular moduli to the problem of constructing an inverse to $\SQ$ in \S~\ref{sec:out-inv}.
Then, with this approach in place, we are in a better position to appreciate the import of our results for penumbral and umbral moonshine. We comment on this in \S~\ref{sec:out-imp}.

\subsection{Inversion}\label{sec:out-inv}

Here we consider the problem of inverting the construction $\SQ$. 
That is, given a weakly holomorphic $G$-module $V$ of weight $0$ as in (\ref{eqn:res-for:V}), 
and given the hypothesis that $V=V^W=\SQ(W)$ (cf.\ (\ref{eqn:res-for:SQW})),
for some weakly holomorphic $G$-module $W$ of weight $\frac12$ and positive integer index, we here pursue a strategy for determining the values $C^W_g(D,r)$ as in (\ref{eqn:res-for:CWgDr}), that characterize the $G$-module structure on $W$, 
given just the functions $f^V_g=T^W_g$ (cf.\ (\ref{eqn:res-for:TWg})), that characterize the $G$-module structure on $V=\SQ(W)$. 
The challenge in this is that the $T^W_g$ depend only on the $C^W_g(D,r)$ for perfect-square values of $D$, according to (\ref{eqn:res-for:PsiWg}--\ref{eqn:res-for:etaWg}).
In particular, we can read off the $C^W_g(n^2,n)$ for $n\in \ZZ$ from these formulas (\ref{eqn:res-for:PsiWg}--\ref{eqn:res-for:etaWg}), if given the functions $f^V_g=T^W_g$. 
Thus we are tasked with determining the values $C^W_g(D_1n^2,r_1n)$ for $D_1>1$ fundamental, for $r_1$ such that $D_1\equiv r_1^2\xmod 4m$, and for integers $n$, given just the $C^W_g(n^2,n)$.
As we will see, it develops that we can solve this problem for $g=e$ the identity element of $G$ in terms of traces of singular moduli, 
as introduced in \cite{ZagierTSM}, by applying results of \cite{BruinierOno2010}.
The problem of extending our approach to non-identity values of $g$ depends upon a natural extension of the methods of op.\ cit.\ (cf.\ \S~\ref{sec:int-twi}).

To explain our approach let $G$ be a finite group, let $V$ be a weakly holomorphic $G$-module of weight $0$, 
and let us assume, as above, that a rational weakly holomorphic $G$-module $W$ of weight $\frac12$ and index $m$ such that $V=V^W=\SQ(W)$ exists. 
Next let $D_1>1$ be a fundamental discriminant that is a square modulo $4m$, and let $r_1$ be a witness to this fact, so that $D_1\equiv r_1^2\xmod 4m$. 
Also let $W_m$ denote the coset of $\Gamma_0(m)$ represented by the Fricke involution, $\frac1{\sqrt{m}}\left(\begin{smallmatrix}0&-1\\m&0\end{smallmatrix}\right)$, and write $\Gamma_0(m)+m$ for the extension of $\Gamma_0(m)$ it defines, 
\begin{gather}\label{eqn:out-inv:m+m}
	\Gamma_0(m)+m:=\Gamma_0(m)\cup W_m.
\end{gather}
Then according to Theorem 6.1 of \cite{BruinierOno2010} we have that the twisted Borcherds product $\Psi^{W}_{D_1,r_1}$, defined by (\ref{eqn:int-twi:PsiWDeltae}), is a meromorphic function on $\HH$ that is invariant for the action of $\Gamma_0(m)+m$.
Moreover, the induced meromorphic function on $X_0(m)$ has divisor given by 
\begin{gather}\label{eqn:out-inv:ZWD1r1}
	Z^W_{D_1,r_1}:=
	\sum_{r\xmod 2m}
	\sum_{\substack{
	D<0
	\\
	D\equiv r^2\xmod 4m
	}}
	C^W(D,r)Z^{(m)}_{D_1,r_1}(D,r),
\end{gather}
where $Z^{(m)}_{D_1,r_1}(D,r)$ is as in (\ref{eqn:pre-hee:ZmD1r1}).

Now 
suppose that $m$ is chosen so that 
the compact Riemann surface
$X_0(m)+m$ 
associated to $\Gamma_0(m)+m$ (cf.\ (\ref{eqn:int-hee:X0m})) 
has genus zero. 
Then $\Gamma_0(m)+m$ admits a unique normalized principal modulus (a.k.a.\ Hauptmodul), $T^{(m+m)}$. From the statement just made about $\Psi^W_{D_1,r_1}$, 
and in particular from the explicit description (\ref{eqn:out-inv:ZWD1r1}) of the divisor $Z^W_{D_1,r_1}$ (cf.\ (\ref{eqn:pre-hee:ZmD1r1})), 
we may write $\Psi^W_{D_1,r_1}$ as a product of (finitely many) integer powers of terms of the form $T^{(m+m)}(\tau)-T^{(m+m)}(\alpha_Q)$, for various CM points $\alpha_Q$ (cf.\ (\ref{eqn:pre-hee:ZmD1r1})).
This is useful because $T^{(m+m)}$ is replicable in the sense of \cite{Conway:1979qga} (see also \cite{CN95}). 
According to \S~1 of \cite{CN95} (see also \S~5 of \cite{Carnahan2010}) this means that there exist holomorphic functions $T^{(m+m)}_{(a)}$ for $a>0$ such that 
\begin{gather}\label{eqn:out-inv:TmtauTmalpha}
	T^{(m+m)}(\tau)-T^{(m+m)}(\a_Q) 
	=
	q^{-1}\exp\left(
		-\sum_{n>0}\sum_{ad=n}\sum_{b\xmod d} 
		T^{(m+n)}_{(a)}\left(\frac{a\a_Q+b}{d}\right)
		\frac{q^n}{n}
	\right).
\end{gather}
(In fact, according to \cite{Conway:1979qga}
we have $T^{(m+m)}_{(a)}=T^{(m+m)}$ if $\gcd(a,m)=1$, and
$T^{(m+m)}_{(a)}=T^{(m')}$ for $m'=\frac{m}{\gcd(a,m)}$ otherwise, where $T^{(m)}$ denotes the normalized principal modulus associated to $\Gamma_0(m)$. In particular,
$X_0(m')$ is genus zero whenever $X_0(m)+m$ is genus zero and $m'$ is a proper divisor of $m$.)

Now by comparing powers of $q$ in $\log \Psi^W_{D_1,r_1}(\tau)$ and $\log(T^{(m+m)}(\tau)-T^{(m+m)}(\a_Q))$, where the latter is computed using (\ref{eqn:out-inv:TmtauTmalpha}), 
we can deduce explicit expressions for the $C^W(D_1n^2,r_1n)$ in terms of the values $T^{(m+m)}_{(a)}(\tfrac{a\a_Q+b}{d})$, for explicitly determined CM points $\a_Q$.

\begin{eg}\label{eg:out-inv}
Suppose that $V$ is such that $f^V(\tau)=f^V_e(\tau)$ is the elliptic modular invariant $j$, that appears in (\ref{eqn:int-pro:j}). 
(Such a $G$-module $V=(\vn_{\rm 3C})^{\otimes}$ is considered in \cite{pmt}, for $G=\Th$ the sporadic simple Thompson group. Cf.\ the discussion in \S~\ref{sec:int-moo}, and see also Remark \ref{rmk:res-con:noVOA})
Noting that $j$ is invariant for the action of $\SL_2(\ZZ)=\Gamma_0(1)$ we guess that the corresponding weakly holomorphic $G$-module $W$ of weight $\frac12$ has index $m=1$.
The relevant normalized principal modulus in this case is $T^{(1)}=j-744$, and we have $T^{(1)}_{(a)}=T^{(1)}$ for all $a>0$ in (\ref{eqn:out-inv:TmtauTmalpha}). 
(We write $T^{(1)}$ instead of $T^{(1+1)}$ because $\Gamma_0(1)+1=\Gamma_0(1)$.) Thus, given $D_1>1$ fundamental, and taking $r_1\equiv D_1\xmod 2$, 
we expect $\Psi^W_{D_1,r_1}$ as in (\ref{eqn:int-twi:PsiWDeltae}) to coincide with a product of expressions of the form
\begin{gather}\label{eqn:out-inv:JtauJalpha}
	J(\tau)-J(\a_Q)=
	q^{-1}\exp\left(
		-\sum_{n>0}
		\sum_{{ad=n}}
		\sum_{b\xmod d} 
		J\left(\frac{a\a_Q+b}{d}\right)
		\frac{q^n}{n}
	\right),
\end{gather}
for CM points $\a_Q$,
where we follow tradition in writing $J$ for $T^{(1)}=j-744$.

The particular CM points arising are determined by the divisor $Z^W_{D_1,r_1}$, which according to (\ref{eqn:out-inv:ZWD1r1}) depends only on the singular part of $W$. 
The singular part of $W$ also determines the generalized class number $H=H^W$ that appears in the definition (\ref{eqn:int-mod:PsiWg}) of $\Psi^W=\Psi^W_e$, according to (\ref{eqn:H}),
so we can use knowledge of the latter to constrain the possibilities for the former. 
Let us suppose that $W_{0,0}$ is vanishing, so that $f^V=T^W_e=\Psi^W_e$ (cf.\ (\ref{eqn:res-for:PsiWg})--(\ref{eqn:res-for:TWg})).
Then from (\ref{eqn:int-pro:j}) we have that $H=1$. Comparing with (\ref{eqn:breveG1}) we conclude that the simplest possibility is that $C^W(D,r)$ vanishes for $D\leq 0$ unless $D=-3$ and $r=1$, in which case $C^W(-3,1)=3$.
Supposing that this possibility holds we can compute $Z^W_{D_1,r_1}$ concretely. Indeed, using (\ref{eqn:pre-hee:ZmD1r1}) and (\ref{eqn:out-inv:ZWD1r1}) we obtain that
\begin{gather}\label{eqn:out-inv:ZWD1r1-eg}
	Z^W_{D_1,r_1}=3Z^{(1)}_{D_1,r_1}(-3,1)=
		\sum_{Q\in\mc{Q}^{(1)}_{-3D_1,r_1}}
	3\chi^{(1)}_{D_1}(Q)
	{\overline{\a_Q}},
\end{gather}
where $\chi^{(1)}_{D_1}$ is as in (\ref{eqn:pre-hee:chiD0N}).
Thus, applying (\ref{eqn:out-inv:JtauJalpha}), and the fact that 		
$\sum_{Q\in\mc{Q}^{(1)}_{-3D_1,r_1}}\chi^{(1)}_{D_1}(Q)=0$,
we have
\begin{gather}
\begin{split}\label{eqn:out-inv:PsiWD1r1-trc}
	\Psi^W_{D_1,r_1}(\tau)
	&=
	\prod_{Q\in\mc{Q}^{(1)}_{-3D_1,r_1}}
	\left(J(\tau)-J(\alpha_Q)\right)^{3\chi^{(1)}_{D_1}(Q)}\\
	&=
	\exp\left(
	-\sum_{n>0}
	\sum_{Q\in\mc{Q}^{(1)}_{-3D_1,r_1}}
	3\chi^{(1)}_{D_1}(Q)
	\sum_{{ad=n}}
	\sum_{{b\xmod d}}
	J\left(\frac{a\a_Q+b}{d}\right)
	\frac{q^n}{n}
	\right).
\end{split}
\end{gather}

To obtain concrete expressions for the coefficients $C^W(D_1n^2,r_1n)$ we apply the Gauss identity 
\begin{gather}
\sum_{a\xmod D}\left(\frac{D}{a}\right)\ex\left(\frac{ab}{D}\right)=\sqrt{D}\left(\frac{D}{b}\right)
\end{gather}
to the $g=e$ case of (\ref{eqn:int-twi:PsiWDeltag}) so as to obtain
\begin{gather}\label{eqn:out-inv:PsiWD1r1-cff}
\Psi^W_{D_1,r_1}(\tau)
=
	\exp\left(
	-\sum_{n>0}
	\sqrt{D_1}
	\sum_{{ad=n}}
	a\left(\frac{D_1}{d}\right)
	C^W(D_1a^2,r_1a)
	\frac{q^n}{n}
	\right).
\end{gather} 
Then, identifying the coefficients of $q^n$ in the logarithms of (\ref{eqn:out-inv:PsiWD1r1-trc}) and (\ref{eqn:out-inv:PsiWD1r1-cff}) we obtain
\begin{gather}\label{eqn:out-inv:logcoeff}
	\sqrt{D_1}
	\sum_{{ad=n}}
	a\left(\frac{D_1}{d}\right)
	C^W(D_1a^2,r_1a)
=
	\sum_{Q\in\mc{Q}^{(1)}_{-3D_1,r_1}}
	3\chi^{(1)}_{D_1}(Q)
	\sum_{{ad=n}}
	\sum_{{b\xmod d}}
	J\left(\frac{a\a_Q+b}{d}\right)
\end{gather} 
for $n>0$. Thus we determine the $C^W(D_1n^2,r_1n)$ recursively in terms of the $C^W(D_1a^2,r_1a)$ for $a<n$, and the values $J(\frac{a\a_Q+b}{d})$ for $Q\in \mc{Q}^{(1)}_{-3D_1,r_1}$. In particular, taking $n=1$ we obtain the formula
\begin{gather}\label{eqn:out-inv:CWD1r1-tsm}
	C^W(D_1,r_1)
=
	\frac{1}{\sqrt{D_1}}
	\sum_{Q\in\mc{Q}^{(1)}_{-3D_1,r_1}}
	3\chi^{(1)}_{D_1}(Q)
	J\left(\a_Q\right)
\end{gather} 
for $D_1>1$ fundamental. We refer to the right hand-side of (\ref{eqn:out-inv:CWD1r1-tsm}) as a (twisted) trace of singular moduli, following \cite{ZagierTSM}.
\end{eg}
From the above discussion we conclude that we can solve the problem of recovering $W$ from $V=\SQ(W)$ in terms of traces of singular moduli, at least in the case that $G=\{e\}$ is trivial,
so long as the graded dimension function $f^V=f^V_e$ is invariant for a group $\Gamma_0(m)+m$ that has genus zero, and so long as we have some knowledge of the singular terms in the graded dimension function $F^W=F^W_e$ of $W$. (We can actually weaken the condition that 
$\Gamma_0(m)+m$ have genus zero,
and handle non-trivial $G$ in certain circumstances. 
See Remarks \ref{rmk:out-inv:gz} and \ref{rmk:out-inv:mlt} below.)

In principle the problem of recovering the $C^W_g(D_1n^2,r_1n)$ for non-trivial $g\in G$ can be handled in a directly similar way, except that the results of \cite{BruinierOno2010} do not extend to the twined twisted Borcherds products $\Psi^W_{D_1,r_1,g}$ of (\ref{eqn:int-twi:PsiWDeltag}), for $g\neq e$ (but see Remark \ref{rmk:out-inv:mlt} below). 
As we have mentioned in \S~\ref{sec:int-twi}, we expect that the main step in determining such an extension of the results of op.\ cit.\ will be an 
analysis of twisted Siegel theta functions $\Theta^{(m)}_{D_1,r_1,N}$, based on the lattices $L_{m,N}$ (see (\ref{eqn:pre-lat:LmN})) for $N>1$.
For concrete hints as to what to expect from such an analysis we 
refer the reader to \S~4.3 of \cite{pmt}, wherein we present (mostly conjectural) 
expressions in terms of traces of singular moduli for (most of) the McKay--Thompson series associated to a weakly holomorphic $\Th$-module $W_{{\rm 3C}}$ of weight $\frac12$. This weakly holomorphic $\Th$-module $W_{\rm 3C}$ has index $1$, and is mapped by $\SQ$ to a $\Th$-module $(\vn_{\rm 3C})^{\otimes 3}$ with graded dimension given (up to rescaling) by $j$, just as in Example \ref{eg:out-inv}.

To better appreciate the generality, and specificity, of the method we have described for inverting $\SQ$ via traces of singular moduli, we offer the following remarks.

\begin{rmk}\label{rmk:out-inv:gz}
It is not necessary that $m$ be such that $\Gamma_0(m)+m$ has genus zero in order for the approach we have sketched to apply. 
This is because the invariance group of $V=V^W=\SQ(W)$ will be an extension 
\begin{gather}\label{eqn:out-inv:gz-ALext}
\Gamma_0(m)+n,n',\dots,m
\end{gather} 
of $\Gamma_0(m)+m$ by further Atkin--Lehner involutions, $W_n$, $W_{n'}$, \dots, in addition to $W_m$ (cf.\ (\ref{eqn:out-inv:m+m})), under suitable conditions on $W$. 
If this extension (\ref{eqn:out-inv:gz-ALext})
has genus zero then the above approach goes through, with the associated normalized principal modulus 
taking on the role of $T^{(m+m)}$.
\end{rmk}

\begin{rmk}\label{rmk:out-inv:opt}
The
fact that the right hand-side of (\ref{eqn:out-inv:CWD1r1-tsm}) reduces to a formula in terms of quadratic forms of a single discriminant depends upon our assumption in Example \ref{eg:out-inv} that there is a unique negative discriminant $D=D_0<0$ such that $C^W(D,r)$ does not vanish.
As we see from (\ref{eqn:out-inv:ZWD1r1}), the divisor of $\Psi^W_{D_1,r_1}$ generally involves CM points of every discriminant $DD_1$ such that $D<0$ and $C^W(D,r)$ is not zero.
Thus the same is true for the sums in general counterparts to 
(\ref{eqn:out-inv:ZWD1r1-eg}) and (\ref{eqn:out-inv:CWD1r1-tsm}).
\end{rmk}

\begin{rmk}\label{rmk:out-inv:mlt}
It develops that the untwined twisted Borcherds products $\Psi^W_{D_1,r_1}$ of (\ref{eqn:int-twi:PsiWDeltae}) can sometimes be used to recover expressions in terms of traces of singular moduli for Fourier coefficients $C^W_g(D,r)$, for non-trivial $g\in G$. This is 
because 
the assignment $F(\tau)\mapsto F(M\tau)$, 
for $M$ a positive integer, 
defines a level-raising map 
\begin{gather}
	\textsl{V}^\wh_{\frac12,mM}(N)\to\textsl{V}^\wh_{\frac12,m}(MN)
\end{gather}
(and similarly with $\mathbb{V}$ in place of $\textsl{V}$),
which in special situations can be used to relate coefficients $C^W_g(D,r)$ and $C^{W'}_{g'}(D,r)$, for 
weakly holomorphic modules $W$ and $W'$ of weight $\frac12$, 
and respective indexes $m$ and $m'$, when it holds that $mo(g)=m'o(g')$. 
That is, in certain circumstances we can circumvent the problem that $F^W_g$ has level $N_g>1$ (cf.\ (\ref{eqn:res-for:FWgnVwh})), by identifying its coefficients as those of $F^{W'}_e$, for some weakly holomorphic $G'$-module $W'$ of weight $\frac12$ and index $m'=mo(g)$, for some auxiliary group $G'$. 
\end{rmk}

\subsection{Import}\label{sec:out-imp}

We conclude with some comments on the import of this work for moonshine.

To begin  we point out that for each lambdency $\pd=(D_0,\ell)$  
of penumbral moonshine, as described in \cite{pmo}, 
the $G^{(\pd)}$-module $W^{(\pd)}$ is rational weakly holomorphic of weight $\frac12$ with (positive integer) index $m$, in the sense of \S~\ref{sec:res-for}, where $m$ is the level of $\ell=m+n,n',\dots,m$. 
By a similar token, for each lambency $\ell=m+n,n',\dots$ of umbral moonshine, as described in \cite{Cheng:2012tq,cheng2014umbral,cheng2018weight}, the $G^{(\ell)}$-module $K^{(\ell)}$ is rational weakly holomorphic mock modular of weight $\frac12$ with index $-m$, in the sense of Remark \ref{rmk:res-for:jacmod}.
Thus, 
according to Remark \ref{rmk:res-for:nonegm},
the construction $\SQ$ of (\ref{eqn:res-for:SQW}) applies to all cases of penumbral moonshine, but to no cases of umbral moonshine. 
However, the notion of twisted Borcherds product, as formulated in \S~\ref{sec:int-twi}, allows us to put this in a broader perspective, as we will presently see.

With respect to twisted Borcherds products, 
it is significant that
the special circumstances 
mentioned in Remarks \ref{rmk:out-inv:gz}--\ref{rmk:out-inv:mlt} all manifest in penumbral 
and umbral moonshine.
To explain this write $F^{(\pd)}_g$ in place of $F^W_g$ 
when $W=W^{(\pd)}$ and $g\in G^{(\pd)}$, for $\pd=(D_0,\ell)$ a lambdency of penumbral moonshine.
Then for such $W=W^{(\pd)}$
the $D_0$ arising in Remark \ref{rmk:out-inv:opt} is the $D_0$ in $\pd=(D_0,\ell)$, 
and the genus zero group (\ref{eqn:out-inv:gz-ALext}) 
arising in Remark \ref{rmk:out-inv:gz} is specified by the 
lambency 
symbol $\ell=m+n,n',\dots,m$.
In umbral moonshine the situation is similar, except that $D_0$ is $1$ in every case,
so it is suppressed from notation, and the lambency symbol $\ell=m+n,n',\dots$ specifies a genus zero group that does not include the Fricke involution $W_m$.
(It develops that the identity (\ref{eqn:pre-mod:Fminusr}) is responsible for the presence, or absence, of $W_m$, and there are analogous identities for the $n,n',\dots$ in (\ref{eqn:out-inv:gz-ALext}). 
See \cite{Cheng:2016klu} or  \cite{pmo} for more detail.)
Also, concrete examples of the multiplicative relations of Remark \ref{rmk:out-inv:mlt}, wherein coefficients of $F^{(\pd)}_g$ are written in terms of those of $F^{(\pd')}=F^{(\pd')}_e$, for suitable 
lambdencies $\pd$ and $\pd'$, and $g\in G^{(\pd)}$, can be found in Tables 4--5 of \cite{pmo},
and multiplicative relations for umbral moonshine can be found in Tables 8--9 of \cite{cheng2014umbral} (see also Table 2 of \cite{Cheng:2016klu}). 

The primary motivation for the data $D_0$ and $\ell$ that constitute a penumbral lambdency $\pd=(D_0,\ell)$ is that they 
allow us to formulate the sense in which $F^{(\pd)}=F^{(\pd)}_e$ is genus zero and optimal, and thereby serves as a natural weight $\frac12$ counterpart to a principal modulus. 
We refer to the introduction of \cite{pmo}, and especially \S~1.6 of op.\ cit., for a detailed discussion of this. 
(See also \S~4.4 of \cite{pmt}, and the forthcoming work \cite{pmz}.) 
We emphasize here that the multiplicative relations of penumbral and umbral
moonshine,
which we have sketched in a general way in Remark \ref{rmk:out-inv:mlt},
depend upon the genus zero and optimality properties that 
we have just described, 
in relation to 
Remarks \ref{rmk:out-inv:gz}--\ref{rmk:out-inv:opt}. 
We refer to \S~4.2 of \cite{pmo}, and also the latter part of \S~4.1 of op.\ cit., for more detail on the multiplicative relations of penumbral moonshine. 
A counterpart discussion for umbral moonshine can be found in \cite{cheng2014umbral}.

We have indicated that the special circumstances of Remarks \ref{rmk:out-inv:gz}--\ref{rmk:out-inv:mlt} hold in umbral moonshine, but our lift of the singular theta lift (\ref{eqn:res-for:SQW}) does not apply to umbral moonshine, according to Remark \ref{rmk:res-for:nonegm}. 
This brings us back 
to the results of \cite{ORTL,Cheng:2016klu} that we referred to in \S~\ref{sec:int-twi}, wherein untwined twisted Borcherds products 
very similar to the $\Psi^W_{D_1,r_1}$ of (\ref{eqn:int-twi:PsiWDeltae}) 
are used to constructively connect cases of umbral moonshine with principal moduli for genus zero groups. 
In fact the essential difference in the umbral case is that 
$D_1$ should be a negative fundamental discriminant rather than a positive one, because 
the modules of umbral moonshine are weakly holomorphic (mock) modular of weight $\frac12$ with negative index, rather than positive index (cf.\ Remark \ref{rmk:res-for:jacmod}).
The untwisted Borcherds product construction (\ref{eqn:int-mod:PsiWg}) that underpins $\SQ$ should be regarded as the $D_1=r_1=1$ case of the more general twisted construction (\ref{eqn:int-twi:PsiWDeltag}) that we specified in \S~\ref{sec:int-twi}, so we may say that our main result in this paper, Theorem \ref{thm:res}, fails to apply to umbral moonshine, for the simple reason that $D_1=1$ is not negative. 

Our final comment is that, apart from the fact that $D_1$ in (\ref{eqn:int-twi:PsiWDeltae}) should be negative in the umbral setting, the aforementioned results of \cite{ORTL,Cheng:2016klu} are very similar to what we have described in this section, in that they produce formulae very similar to (\ref{eqn:out-inv:CWD1r1-tsm}), for the coefficients of the graded dimension functions of umbral moonshine in terms of traces of singular moduli.
We conclude that the apparent 
prejudice of our Borcherds product construction $\SQ$ (\ref{eqn:res-for:SQW}), for penumbral moonshine over umbral moonshine, 
is appeased if we can positively answer the question posed
at the end of \S~\ref{sec:int-twi}, on lifting the twisted Siegel theta lift (\ref{eqn:int-twi:int}), for arbitrary fundamental $D_1$, to the level of $G$-modules too.


\addcontentsline{toc}{section}{References}

\setstretch{1.08}

\newcommand{\etalchar}[1]{$^{#1}$}


\begin{thebibliography}{DHVW86}


\bibitem[Bor92]{Borcherds1992}
R.~Borcherds,
\newblock \emph{Monstrous moonshine and monstrous {L}ie superalgebras}.
\newblock Invent.\ Math.\ {\bf 109} (1992), no.~2, 405--444.

\bibitem[Bor95]{Borcherds1995}
R.~Borcherds,
\newblock \emph{Automorphic forms on {$O_{s+2,2}(\mathbb{R})$} and infinite products}.
\newblock Invent.\ Math.\ {\bf 120} (1995), no.~1, 161--213.

\bibitem[Bor98a]{Borcherds:1996uda}
R.~Borcherds,
\newblock \emph{Automorphic forms with singularities on {G}rassmannians}.
\newblock Invent.\ Math.\ {\bf 132} (1998), no.~3, 491--562.

\bibitem[BO10]{BruinierOno2010}
J.~Bruinier and K.~Ono,
\newblock \emph{Heegner divisors, {$L$}-functions and harmonic weak {M}aass forms}.
\newblock Ann.~of Math.~(2) {\bf 172} (2010), no.~3, 2135--2181.

\bibitem[Car10]{Carnahan2010}
S.~Carnahan,
\newblock \emph{Generalized moonshine {I}: genus-zero functions}.
\newblock Algebra Number Theory {\bf 4} (2010), no.~6, 649--679.

\bibitem[Car12a]{Carnahan2012}
S.~Carnahan,
\newblock \emph{Generalized moonshine {II}: {B}orcherds products}.
\newblock Duke Math.~J.\ {\bf 161} (2012), no.~5, 893--950.

\bibitem[Car12b]{Carnahan:2012gx}
S.~Carnahan,
\newblock \emph{Generalized moonshine {IV}: monstrous Lie algebras}.
\newblock [arXiv:1208.6254 [math.RT]].

\bibitem[CD20]{Cheng:2016klu}
M.~Cheng and J.~Duncan,
\newblock \emph{Optimal Mock Jacobi Theta Functions}.
\newblock Adv.~Math., {\bf 372} (2020), 107284.


\bibitem[CDH14a]{Cheng:2012tq}
M.~Cheng, J.~Duncan, and J.~Harvey,
\newblock \emph{Umbral Moonshine}.
\newblock Commun.\ Number Theory Phys.\ {\bf 8} (2014), no.~2, 101--242.

\bibitem[CDH14b]{cheng2014umbral}
M.~Cheng, J.~Duncan, and J.~Harvey,
\newblock \emph{Umbral moonshine and the Niemeier lattices}.
\newblock Res.\ Math.\ Sci.\ {\bf 1} (2014), Art.~3, 81 pp.

\bibitem[CDH18]{cheng2018weight}
M.~Cheng, J.~Duncan, and J.~Harvey,
\newblock \emph{Weight one Jacobi forms and Umbral moonshine}.
\newblock J.~Phys.~A {\bf 51} (2018), no.~10, 104002, 37 pp.

\bibitem[CCN+85]{atlas}
J.~Conway, R.~Curtis, S.~Norton, R.~Parker and R.~Wilson,
\newblock \emph{Atlas of finite groups: maximal subgroups and ordinary characters for simple groups}.
\newblock With computational assistance from J.~G.~Thackray.
\newblock Oxford University Press, Eynsham, 1985.

\bibitem[CN79]{Conway:1979qga}
J.~Conway and S.~Norton,
\newblock \emph{Monstrous Moonshine}.
\newblock Bull.\ London Math.\ Soc.\ {\bf 11} (1979), no.~3, 308--339.

\bibitem[CN95]{CN95}
C.~Cummins and S.~Norton,
\newblock \emph{Rational Hauptmoduls are replicable}.
\newblock Canad.\ J.\ Math.\ {\bf 47} (1995), no.~6, 1201--1218.


\bibitem[DHR21]{pmo}
J.~Duncan, J.~Harvey, and B.~Rayhaun,
\newblock \emph{An {O}verview of {P}enumbral {M}oonshine}.
\newblock [arXiv:2109.09756 [math.RT]].

\bibitem[DHR22a]{pmt}
J.~Duncan, J.~Harvey, and B.~Rayhaun,
\newblock \emph{Two new avatars of moonshine for the Thompson group}.
\newblock [arXiv:2202.***** [math.RT]].

\bibitem[DHR22b]{pmz}
J.~Duncan, J.~Harvey, and B.~Rayhaun, 
\newblock \emph{Skew-holomorphic Jacobi forms and genus zero groups}.
\newblock In preparation.



\bibitem[FLM84]{flm}
I.~Frenkel, J.~Lepowsky and A.~Meurman,
\newblock \emph{A natural representation of the {F}ischer-{G}riess monster with the modular function {$J$} as character}.
\newblock Proc.\ Nat.\ Acad.\ Sci.\ U.S.A.\ {\bf 81} (1984), no.~10, Phys.\ Sci., 3256--3260.

\bibitem[FLM85]{FLMBerk}
I.~Frenkel, J.~Lepowsky and A.~Meurman,
\newblock \emph{A moonshine module for the {M}onster}. 
\newblock Vertex operators in mathematics and physics (Berkeley, Calif., 1983), 231--273,
\newblock Math.\ Sci.\ Res.\ Inst.\ Publ., {\bf 3}, Springer, New York, 1985.

\bibitem[FLM88]{frenkel1989vertex}
I.~Frenkel, J.~Lepowsky and A.~Meurman,
\newblock \emph{Vertex operator algebras and the Monster}.
\newblock Pure and Applied Mathematics, 134.
\newblock Academic Press, Inc., Boston, MA, 1988.


\bibitem[GM16]{Griffin2016}
M.~Griffin and M.~Mertens,
\newblock \emph{A proof of the {T}hompson moonshine conjecture}.
\newblock Res.\ Math.\ Sci.\ {\bf 3} (2016), Paper No.~36, 32 pp.

\bibitem[GKZ87]{GKZ87}
B.~Gross, W.~Kohnen, and D.~Zagier, 
\newblock \emph{Heegner points and derivatives of L-series. II.} 
\newblock Math.\ Ann.\ {\bf 278} (1987), no.~1--4, 497--562.

\bibitem[HM96]{Harvey:1995fq}
J.~Harvey and G.~Moore,
\newblock \emph{Algebras, BPS states, and strings}.
\newblock Nucl.~Phys.~B {\bf 463} (1996), no.~2--3, 315--368.

\bibitem[HR16]{Harvey:2015mca}
J.~Harvey and B.~Rayhaun,
\newblock \emph{Traces of singular moduli and moonshine for the Thompson group}.
\newblock Commun.\ Number Theory Phys.\ {\bf 10} (2016), no.~1, 23--62.

\bibitem[Kim06]{kim_2006}
C.~Kim,
\newblock \emph{Traces of singular values and {B}orcherds products}.
\newblock Bull.\ London Math.\ Soc.\ {\bf 38} (2006), no.~5, 730--740.

\bibitem[ORT-L14]{ORTL}
K.~Ono, L.~Rolen and S.~Trebat-Leder,
\newblock \emph{Classical and umbral moonshine: connections and $p$-adic properties.} 
\newblock J.~Ramanujan Math.~Soc.\ {\bf 30} (2015), no.~2, 135--159.

\bibitem[Ser77]{serrereps}
J.-P.~Serre,
\newblock \emph{Linear representations of finite groups.} 
\newblock Translated from the second French edition by Leonard L.~Scott, Graduate Texts in Mathematics, Vol.~42.
\newblock Springer-Verlag, New York-Heidelberg, 1977.

\bibitem[Tho76a]{MR399257}
J.~Thompson,
\newblock \emph{Finite groups and even lattices}.
\newblock J.~Algebra \textbf{38} (1976), no.~2, 523--524. 

\bibitem[Tho76b]{MR0399193}
J.~Thompson,
\newblock \emph{A conjugacy theorem for {$E_{8}$}}.
\newblock J.~Algebra \textbf{38} (1976), no.~2, 525--530.

\bibitem[Zag11]{ZagierTSM}
D.~Zagier,
\newblock \emph{Traces of singular moduli}.
\newblock Motives, polylogarithms and {H}odge theory, {P}art {I} ({I}rvine, {CA}, 1998), 211--244,
\newblock Int.\ Press Lect.\ Ser., 3, I, 
\newblock Int.\ Press, Somerville, MA, 2002.

\end{thebibliography}
\end{document}